\documentclass[11pt]{article}
\usepackage[a4paper,margin=26mm]{geometry}
\usepackage[T1]{fontenc}
\usepackage{lmodern,amsmath,amssymb,amsthm,mathtools}
\usepackage{booktabs,array,microtype,enumitem}
\usepackage{mathrsfs}
\usepackage[hidelinks]{hyperref}
\newtheorem{theorem}{Theorem}[section]
\newtheorem{lemma}[theorem]{Lemma}
\newtheorem{proposition}[theorem]{Proposition}
\newtheorem{corollary}[theorem]{Corollary}
\newtheorem{conjecture}[theorem]{Conjecture}
\newcommand{\repeatedtheoremname}{}
\newtheorem*{repeatedtheorem}{\repeatedtheoremname}
\newenvironment{reptheorem}[1]
  {\renewcommand{\repeatedtheoremname}{Theorem~\ref{#1}}\begin{repeatedtheorem}}
  {\end{repeatedtheorem}}
\theoremstyle{definition}
\newtheorem{problem}[theorem]{Problem}

\newtheorem{example}[theorem]{Example}
\theoremstyle{remark}
\newtheorem{remark}[theorem]{Remark}
\newcommand{\one}{\mathbf 1}
\newcommand{\Id}{\mathrm I}
\newcommand{\F}{\mathbb F}
\newcommand{\Z}{\mathbb Z}
\newcommand{\T}{\mathsf T}
\newcommand{\per}{\one^{\perp}}
\newcommand{\Kzero}{K_0}
\newcommand{\lineindex}{\mathfrak l}
\newcommand{\Bcal}{\mathcal B}
\newcommand{\Gcal}{\mathcal G}
\newcommand{\ip}[2]{\langle #1,#2\rangle}
\newcommand{\norm}[1]{\lVert #1\rVert}
\newcommand{\floor}[1]{\left\lfloor #1\right\rfloor}
\newcommand{\ceil}[1]{\left\lceil #1\right\rceil}
\newcommand{\gqbinom}[2]{\genfrac{[}{]}{0pt}{}{#1}{#2}_q}
\DeclareMathOperator{\dist}{dist}
\DeclareMathOperator{\diam}{diam}
\DeclareMathOperator{\Spec}{Spec}
\DeclareMathOperator{\SRG}{SRG}
\DeclareMathOperator{\Cay}{Cay}

\DeclareMathOperator{\Lip}{Lip}
\DeclareMathOperator{\Var}{Var}
\DeclareMathOperator{\tr}{tr}

\DeclareMathOperator{\ran}{ran}
\setlist{itemsep=3pt,topsep=4pt}
\allowdisplaybreaks[2]
\title{Diameter Bounds and Finiteness of Graphs\\
from Local Intersection Data}
\author{Kaizhe Chen\thanks{Department of Mathematics, Princeton University, Princeton, USA. Email: kc8556@princeton.edu}}
\date{}
\begin{document}
\maketitle

\begin{abstract}
We extend the Bakry--\'Emery curvature framework on graphs, broadening its applications and strengthening the resulting diameter bounds. We study diameter bounds and finiteness for amply regular graphs, distance-regular graphs, $(s,c,a,k)$-graphs, and locally strongly regular graphs. In particular, for an amply regular graph with parameters $(v,k,\lambda,\mu)$, we prove that
\[
 \diam(G)\le4+8\log\frac{k}{\max\{\lambda,\mu\}}
\]
when $\mu>2$ and $k\ge3\lambda^2/\mu$, improving the previous linear bound $4k/(\lambda+4)$ to a logarithmic bound.

We also address several related conjectures. In their monograph \emph{Distance-Regular Graphs}, Brouwer, Cohen, and Neumaier conjectured that every amply regular graph with $\mu>1$, except the icosahedron, satisfies $\diam(G)\le2k/(\lambda+2)$. We confirm this bound when $\mu>\lambda+1$ and give counterexamples to the conjecture with $\mu=\lambda$. For amply regular graphs with $\mu\ge\lambda$ and $\mu>1$, we obtain $\diam(G)\le 2k/\mu+1$. This proves the Qiao--Park--Koolen conjecture.

In 1983, Terwilliger conjectured that every $(s,c,a,k)$-graph with $c>1$ is finite. We prove this conjecture when ``$s>2$, $c>2$, $(s,c)\ne(4,3)$''
and when ``$s=2$, $c>7$''. Moreover, we prove Weetman's finiteness conjecture for locally strongly regular graphs when the local parameter $\mu>1$, and disprove his stronger conjecture that every strongly regular graph is an $\mathscr{S}$-graph. We also prove that every graph whose local graphs all have rank 3 is finite, which in particular proves a conjecture of Soicher in 1992.
\end{abstract}
\tableofcontents

\section{Introduction}\label{sec:intro}
Local intersection data often impose strong restrictions on the global geometry of a graph. Prescribing the number of common neighbours of nearby vertices can bound the diameter and, consequently, the order. Several conjectures in algebraic graph theory ask when such local conditions force finiteness or yield explicit diameter bounds. In this paper, we study these questions for amply regular graphs, distance-regular graphs, $(s,c,a,k)$-graphs, and locally strongly regular graphs. Our diameter bounds apply to possibly infinite graphs and therefore also prove finiteness. All graphs are simple and undirected, and ambient graphs are assumed to be connected.

Our main approach extends the Bakry--\'Emery curvature framework \cite{BE}. In its classical graph form, a positive curvature lower bound gives a diameter bound through a Bonnet--Myers inequality \cite{LMP}. We allow the quadratic gradient and the Markov generator to be chosen separately, adapting them to the available intersection data. Modified gradients yield logarithmic diameter bounds where the classical curvature estimate gives a linear bound. Longer-range generators and gradients also give finiteness results when the local conditions occur beyond the first two distance layers. The common analytic framework is developed in Section~\ref{sec:prelim}.

\subsection{Amply regular graphs}\label{sec:intro-amply}
A non-complete $k$-regular graph on $v$ vertices is called \emph{amply regular} with parameters $(v,k,\lambda,\mu)$ if any two adjacent vertices have exactly $\lambda$ common neighbours, and any two vertices at distance two have exactly $\mu$ common neighbours. Here, the order $v$ is allowed to be infinite.

Bounding the diameter of an amply regular graph in terms of its local parameters is a widely studied problem. Brouwer, Cohen, and Neumaier \cite[Corollary~1.9.2]{BCN} proved that $\diam(G)\le\max\{3,k-2\mu+4\}$ when $\mu>1$ and $\mu\ge\lambda$. Terwilliger \cite{Ter} obtained
\begin{equation}\label{eq:intro-terwilliger-bound}
 \diam(G)\le\max\left\{5,\frac{2\mu(k-1)}{3\mu-2}-2\mu+7\right\}
 \qquad\text{when }\mu\ge\max\{3,\lambda\}.
\end{equation}
Using Lin--Lu--Yau curvature, Huang, Liu, and Xia \cite{HLX} proved $\diam(G)\le2k/3$ when $\mu>\lambda\ge1$. Chen, Hu, Liu, and Zhang \cite{CHLZ} subsequently established
\[
 \diam(G)\le
 \frac{2k}{2+\max\{\lambda/2,\ceil{\lambda(\mu-\lambda)/(\mu-1)}\}}
\]
when $\mu\ge\max\{3,\lambda\}$, by combining Lin--Lu--Yau and Bakry--\'Emery curvature estimates. Using long-scale Ollivier Ricci curvature, Chen and Liu \cite[Corollary~1.6]{CL} proved
\[
 \diam(G)\le
 \max\left\{3,
 \frac{2(k-2\mu)}{2+\ceil{\lambda(\mu-\lambda)/(\mu-1)}}+4\right\}
\]
when $\mu>1$ and $\mu\ge\lambda$.

Two conjectures ask for stronger control of the diameter by these parameters. Inspired by the classical bound $\diam(G)\le2k/(\lambda+2)$ \cite{TerGirth} for distance-regular graphs containing an induced quadrangle, Brouwer, Cohen, and Neumaier asked whether it extends to amply regular graphs in their well-known monograph \emph{Distance-Regular Graphs}.

\begin{conjecture}[Brouwer--Cohen--Neumaier {\cite{BCN}}]\label{conj:bcn}
Every amply regular graph $G$ with $\mu>1$, except the icosahedron, satisfies
\[
 \diam(G)\le\frac{2k}{\lambda+2}.
\]
\end{conjecture}

Inspired by Terwilliger's bound \eqref{eq:intro-terwilliger-bound}, Qiao, Park, and Koolen proposed the following conjecture, asking for a uniform diameter bound when $\mu$ occupies a positive proportion of $k$.

\begin{conjecture}[Qiao--Park--Koolen {\cite[Conjecture~1.2]{QPK}}]\label{conj:qpk}
For every $0<\varepsilon<1$, there exists a constant $K(\varepsilon)$ such that every amply regular graph $G$ with $\mu>\varepsilon k\ge1$ and $\mu\ge\lambda$ satisfies $\diam(G)\le K(\varepsilon)$.
\end{conjecture}

We mention that the condition $\mu>1$ is essential for general finiteness and diameter bounds: the two-sided infinite path is an amply regular graph with $k=2$, $\lambda=0$, and $\mu=1$.

Our first diameter bound follows from an estimate of long-scale Ollivier Ricci curvature between vertices at distance two. It improves all the diameter bounds recalled above when $\mu>3$ (Remark~\ref{rem:arg-comparison}).

\begin{theorem}\label{thm:arg}
Let $G$ be an amply regular graph with parameters $(v,k,\lambda,\mu)$, where $\mu>1$ and $\mu\ge\lambda$. Then
\[
 \diam(G)\le
 \begin{cases}
  \displaystyle\frac{2k}{\mu},&\text{if }\diam(G)\text{ is even},\\[12pt]
  \displaystyle\frac{2k-\lambda-2}{\mu}+1,&\text{if }\diam(G)\text{ is odd}.
 \end{cases}
\]
In particular, $\diam(G)\le 2k/\mu+1$.
\end{theorem}

Theorem~\ref{thm:arg} proves Conjecture~\ref{conj:qpk}, with $K(\varepsilon)=2/\varepsilon+1$. It also gives $\diam(G)\le2k/(\lambda+2)$ when $\mu>\lambda+1$, confirming Conjecture~\ref{conj:bcn} in this range: for odd diameter the estimate reads $(\diam(G)-1)\mu+\lambda+2\le2k$, and the even case is immediate. When $\mu=\lambda+1>1$, we obtain a slightly stronger bound (Theorem~\ref{thm:local}). On the other hand, we give counterexamples with $\mu=\lambda$, showing that Conjecture~\ref{conj:bcn} does not hold in full generality (Theorem~\ref{thm:conjectures}).

Obtaining diameter bounds when $\lambda>\mu$ is substantially harder than when $\mu\ge\lambda$. In \cite{CHLZ}, the authors obtained the first general diameter bound for sufficiently large degree in this regime: $\diam(G)\le4k/(\lambda+4)$ when $\mu>2$ and $k\ge\lambda^2/[2(\mu-2)]$. This covers many parameters with $\lambda>\mu$, but the dependence on $k$ remains linear, as in the earlier bounds above. By modifying the gradient in the curvature inequality, we obtain the first logarithmic diameter bounds in this generality. Under stronger degree hypotheses, discrete variance contraction further improves the order from $\log k$ to $\log_\mu k$. The following theorem is intended to provide simple uniform estimates; for specific parameters $(k,\lambda,\mu)$, optimizing the curvature estimate often gives sharper bounds. All logarithms below are natural.

\begin{theorem}\label{thm:arg-logarithmic}
Let $G$ be an amply regular graph with parameters $(v,k,\lambda,\mu)$, where $\mu>2$.
\begin{enumerate}[label=\textup{(\roman*)}]
\item If $k\ge3\lambda^2/\mu$, then
\[
 \diam(G)\le4+8\log\frac{k}{\max\{\lambda,\mu\}}.
\]
\item For any $0\le\eta<1$, if $k\ge16\lambda^2/\mu^\eta$, then
\[
 \diam(G)\le\frac{8\log k}{(1-\eta)\log\mu}+10.
\]
\item If $\lambda=0$, then
\[
 \diam(G)\le\frac{4\log k}{\log\mu}+10.
\]
\end{enumerate}
\end{theorem}

Examples with $(\lambda,\mu)=(1,3)$ show that the order $\log k$ cannot be improved, even beyond arbitrary degree thresholds (Proposition~\ref{prop:polar}). We also show that the degree hypothesis in Theorem~\ref{thm:arg-logarithmic}\textup{(i)} does not permit any additional denominator tending to infinity with $\lambda$ (Theorem~\ref{th:bilinear}). Moreover, for any \(\varepsilon,C>0\), neither \(C\log k/(\log\lambda)^{1+\varepsilon}\) nor \(C\log k/(\log\mu)^{1+\varepsilon}\) is a diameter bound in general, even beyond arbitrary degree thresholds (Corollary~\ref{th:negative}).

For all \(\mu>7\), we prove finiteness for every \(\lambda\) and \(k\). The proof combines a clique covering argument at small degrees with positive curvature at larger degrees and yields the following diameter bounds.

\begin{theorem}\label{thm:arg-even-finiteness}
Let $G$ be an amply regular graph with parameters $(v,k,\lambda,\mu)$, where $\mu>7$. Then $G$ is finite and 
\[
 \diam(G)\le
 \begin{cases}
  \displaystyle\frac{4\mu k}{2\mu k+4\mu +\lambda \mu -\lambda^2-4k},&\text{if } 0<\lambda^2-2\mu k+4k<4\mu +\lambda \mu,\\[12pt]
  \displaystyle\frac{4k}{\lambda+4},&\text{otherwise}.
 \end{cases}
\]
\end{theorem}

We also obtain diameter bounds under additional local restrictions. A non-complete graph is called \emph{Terwilliger} if it contains no induced quadrangle. For structural properties of amply regular Terwilliger graphs, see \cite{TerGirth} and \cite[Theorem~1.16.3]{BCN}. For $\mu>2$, we obtain the bound $\diam(G)\le4k/(\lambda+4)$. We state this result in Theorem~\ref{thm:terw} below, together with its consequence for distance-regular graphs.

All results in this subsection are proved in Section~\ref{sec:amply}.

\subsection{Distance-regular graphs}\label{sec:intro-drg}
For a graph $G$, write $N(x)$ for the set of neighbours of a vertex $x$ and $S_i(x)=\{y:\dist(x,y)=i\}$. For any two vertices $x,y$ with $\dist(x,y)=i$, define
\[
 c_i(x,y)=|N(y)\cap S_{i-1}(x)|,\quad
 a_i(x,y)=|N(y)\cap S_i(x)|,\quad
 b_i(x,y)=|N(y)\cap S_{i+1}(x)|.
\]
A graph $G$ of diameter $D$ is called \emph{distance-regular} if there are integers $b_i,c_i$ $(0\le i\le D)$ such that $b_i(x,y)=b_i$ and $c_i(x,y)=c_i$ for any two vertices $x,y$ with $\dist(x,y)=i$. Such a graph is $k$-regular with $k=b_0$, and $a_i(x,y)=a_i:=k-b_i-c_i$ for every such pair. Its \emph{intersection array} is
\[
 \{b_0,b_1,\ldots,b_{D-1};c_1,c_2,\ldots,c_D\},
\]
with $a_0=c_0=b_D=0$ and $c_1=1$. For background on distance-regular graphs, see \cite{BCN,vDKT}.

Bounding the diameter in terms of intersection numbers is one of the central topics in the theory of distance-regular graphs; see, for example, \cite{BDKM,BHK,DG,Godsil,Hiraki,Ivanov,KoolenSub,Pyber,Smith,SuzukiI,SuzukiII,TerBip}. Neumaier and Penji\'c \cite{NP} studied how much information about the diameter can be obtained from a small initial part of the intersection array and posed the following problem.

\begin{problem}[Neumaier--Penji\'c {\cite[Problem~1.1]{NP}}]\label{prob:np}
Let $G$ be a distance-regular graph of diameter $D$, and suppose that only the first $q+2$ entries of each row of its intersection array
\[
 \{b_0,b_1,\ldots,b_{q+1},\ldots;
   c_1,c_2,\ldots,c_{q+2},\ldots\}
\]
are known. Find an upper bound for $D$ in terms of these data.
\end{problem}

In particular, they asked whether the diameter admits an upper bound given by a linear or nonlinear combination of the numbers in $\{k,b_1,1,c_2\}$. The results of Section~\ref{sec:intro-amply} give answers to this question. Indeed, every distance-regular graph on $v$ vertices with initial intersection data $\{k,b_1,\ldots;1,c_2,\ldots\}$ is amply regular with parameters $(v,k,k-b_1-1,c_2)$. In particular, Theorem~\ref{thm:arg-logarithmic} gives logarithmic diameter bounds using only these initial data.

The following theorem gives a bound throughout the range $c_2>2$, in terms of $k$ and $a_1=k-b_1-1$ alone.

\begin{theorem}\label{thm:terw}
Let $G$ be an amply regular Terwilliger graph with parameters $(v,k,\lambda,\mu)$, where $\mu>2$. Then $G$ is finite and
\[
 \diam(G)\le\floor{\frac{4k}{\lambda+4}}.
\]
In particular, every distance-regular graph of diameter $D$ and valency $k$, with $c_2>2$, satisfies
\[
 D\le\floor{\frac{4k}{a_1+4}}.
\]
\end{theorem}

For longer initial segments, we obtain the following estimates using long-scale Ollivier Ricci curvature.

\begin{theorem}\label{thm:penalty}
Let $G$ be a distance-regular graph of diameter $D$. Suppose $2\le q<D$, $c_{q+1}>c_q$, and $a_q\le c_{q+1}$. Define
\[
 \gamma_q=2c_q+
 \frac{a_q(c_{q+1}-a_q)
 -\min\{a_q,a_{q-1}\}\bigl(c_q+|c_{q+1}-c_q-a_q|\bigr)}{c_{q+1}-c_q}.
\]
If $\gamma_q>0$, then for every $0\le p\le D$,
\begin{align*}
 D&\le\max_{0\le r<q}\left\{
 p+r+q\floor{\frac{b_p-c_p+b_r-c_r}{\ceil{\gamma_q}}}\right\},\\
 D&\le2p-1+q\max\left\{0,\floor{\frac{2(b_p-c_p)}{\ceil{\gamma_q}}}+1\right\}.
\end{align*}
\end{theorem}

Theorem~\ref{thm:penalty} extends the bound in \cite[Theorem~1.3]{CL} beyond the assumption $a_{q-1}=0$ and improves the bound in \cite[Theorem~1.2]{NP}; see Remarks~\ref{rem:older-transport} and~\ref{rem:np-comparison}. Choosing $p$ within the known initial segment gives a bound in terms of those data, as requested in Problem~\ref{prob:np}.

Neumaier and Penji\'c \cite{NP} also singled out the initial intersection data $\{3,2,2,2,2,\ldots;\allowbreak 1,1,1,1,2,\ldots\}$. They noted that their bound only gives $D\le9$, suggesting that a sharper bound might be possible. In Theorem~\ref{thm:foster}, we prove the sharp bound $D\le8$, with equality attained by the Foster graph.

Theorem~\ref{thm:terw} is proved in Section~\ref{sec:amply}; the estimates using longer initial segments are proved in Section~\ref{sec:transport}.

\subsection{\texorpdfstring{$(s,c,a,k)$-graphs}{(s,c,a,k)-graphs}}\label{sec:intro-scak}
Let $s,c,k\ge2$ and $a\ge0$ be integers. A graph $G$ of maximum degree $k$ and girth $2s-1$ or $2s$ is called an \emph{$(s,c,a,k)$-graph} if $c_s(x,y)=c$ for any two vertices $x,y$ with $\dist(x,y)=s$, and $a_{s-1}(x,y)=a$ for any two vertices $x,y$ with $\dist(x,y)=s-1$.\footnote{This concise formulation is directly equivalent to Terwilliger's original definition \cite{Ter}.} Note that when $s=2$, a non-complete $(2,c,a,k)$-graph is precisely an amply regular graph with parameters $(v,k,a,c)$.

These graphs are closely related to distance-regular graphs. In particular, a distance-regular graph of valency $k$ and even girth $2s$ is an $(s,c_s,0,k)$-graph, with $c_s\ge2$. Thus, when $c_s>2$, the logarithmic diameter bound for $a=0$ in Theorem~\ref{th:general}, part~3, also applies to such distance-regular graphs.

Terwilliger \cite{Ter} introduced $(s,c,a,k)$-graphs in 1983 while studying diameter bounds for distance-regular graphs. He showed that the girth condition and the local parameters $c_s=c$ and $a_{s-1}=a$ already suffice to force finiteness when $a\le c$. He then asked whether these conditions force finiteness for all $a\ge0$ and $c>1$, leading to the following conjecture.

\begin{conjecture}[Terwilliger {\cite{Ter}}]\label{conj:terwilliger}
Every $(s,c,a,k)$-graph with $c>1$ is finite.
\end{conjecture}

In Section~\ref{sec:long-finiteness}, we prove Conjecture~\ref{conj:terwilliger} in the following range. Note that no restriction relating $a$ and $c$, or largeness assumption on $k$, is required.

\begin{theorem}\label{thm:long-finiteness}
Every $(s,c,a,k)$-graph with $s>2$, $c>2$, and $(s,c)\ne(4,3)$ is finite.
\end{theorem}

For the proofs, we express the generalized curvature forms using incidence matrices between distance layers. Completing squares while retaining key mixed terms reduces curvature positivity to matrix inequalities. This yields a uniform diameter bound for each fixed parameter tuple, and hence only finitely many isomorphism classes.

There are only three remaining cases: $s=2$, $c=2$, and $(s,c)=(4,3)$. The case \(s=2\) concerns amply regular graphs; Theorem~\ref{thm:arg-even-finiteness} settles this case when \(c>7\). For $(s,c)=(4,3)$, we prove finiteness when $k\ge5000$ and $a^3\ge6k^2$ via a separate spectral estimate (Theorem~\ref{thm:s4-diameter-three}). The case $c=2$ appears to be the most difficult. Beyond the known range $a\le c$, our finiteness result uniform in $s$ is currently restricted to $a=3$ and sufficiently large $k$.

\begin{theorem}\label{thm:c2-tail}
Every $(s,2,3,k)$-graph with $k\ge10^8$ is finite.
\end{theorem}

Recall that Theorem~\ref{thm:arg-logarithmic} gives logarithmic diameter bounds for amply regular graphs. In Section~\ref{sec:longscale}, we also obtain logarithmic bounds for $(s,c,a,k)$-graphs. We aim for simple uniform bounds and do not attempt to optimize the constants.

\begin{theorem}\label{th:general}
Let $G$ be an $(s,c,a,k)$-graph with $c>2$.
\begin{enumerate}
\item If $s$ is even, $a\ge1$, and $k\ge8sa^2$, then $\diam(G)\le24s\log k/\log c$.
\item If $a\ge1$ and $k\ge2^{24}a^4$, then $\diam(G)\le32s\log k/\log c$.
\item If $a=0$, then
\[\diam(G)\le\frac{2(s-1)\log(2k)}{\log(c-1)}+10s.\]
\end{enumerate}
\end{theorem}

The girth condition gives unique geodesics at distances below $s$, yielding explicit incidence identities for averaging over distance spheres. These identities and the prescribed values $a,c$ give contraction estimates for suitable variances under long-range averaging. Iterating the contraction bounds the differences between values of Lipschitz functions and yields the logarithmic diameter bounds.

\subsection{Locally strongly regular graphs}\label{sec:intro-local}
A \emph{strongly regular graph} is an amply regular graph of diameter two. We denote the class of all strongly regular graphs with parameters $(v,k,\lambda,\mu)$ by $\SRG(v,k,\lambda,\mu)$. For a family $\mathcal H$ of finite graphs, a graph $G$ is called \emph{locally $\mathcal H$} if the subgraph induced by the set of neighbours of each vertex in $G$ is isomorphic to some member of $\mathcal H$. These induced subgraphs are called the \emph{local graphs} of $G$. In particular, $G$ is called \emph{locally strongly regular} if all its local graphs are strongly regular with fixed parameters $(v,k,\lambda,\mu)$. When $\mathcal H=\{H\}$, we write \emph{locally $H$}. For any finite graph $H$, a graph which is locally $H$ is called an \emph{extension} of $H$.

The extension problem asks which finite graphs $H$ admit locally $H$ graphs, and when all such graphs must be finite. This problem goes back to Zykov \cite{Zykov}; for work on prescribed local structure, see \cite{BHM,Bulitko,WeetmanConstruction,Winkler}. Locally strongly regular graphs have also been studied in a variety of settings; see, for example, \cite{BDFP,BrouwerPaley,GavrilyukMakhnev,Hall}. In 1992, Soicher proposed the following conjecture.

\begin{conjecture}[Soicher {\cite{Soicher}}]\label{conj:soicher}
Every extension of a rank-3 graph is finite.
\end{conjecture}
A connected non-complete graph is called \emph{rank-3} if its automorphism group is transitive on ordered adjacent pairs and on ordered nonadjacent pairs. Here, rank-3 graphs are assumed to be finite. Later, Weetman proposed the following conjecture, extending Soicher's conjecture to locally strongly regular graphs. Note that every rank-3 graph is strongly regular.

\begin{conjecture}[Weetman {\cite{Weetman}}]\label{conj:weetman-finite}
Every locally strongly regular graph is finite.
\end{conjecture}

Weetman \cite{Weetman} proved $\diam(G)\le k+1$ for locally strongly regular graphs with parameters $(v,k,\lambda,\mu)$ satisfying any one of the following conditions:
\begin{equation}\label{eq:weetman-cases}
 v\le2k+1,\qquad \mu>\lambda,\qquad
 \text{all local graphs are geometric}.
\end{equation}
We mention that these three classes do not cover all strongly regular graphs. For example, the Shrikhande graph, the three Chang graphs, and the Cameron graph \cite{BVM} lie outside them.

We prove Conjecture~\ref{conj:weetman-finite} for all locally $\SRG(v,k,\lambda,\mu)$ graphs with $\mu>1$. For $\mu=1$, the conjecture holds when $\lambda=0$ by \eqref{eq:weetman-cases}. We note that no strongly regular graph with $\mu=1$ and $\lambda>0$ is currently known \cite{BVM}.

\begin{theorem}\label{thm:local-finiteness}
Let $G$ be locally $\SRG(v,k,\lambda,\mu)$, where $\mu>1$. Then $G$ is finite and
\[
 \diam(G)\le\left\lfloor\frac{4v}{k+4}\right\rfloor.
\]
\end{theorem}

In fact, we determine the exact Bakry--\'Emery curvature (at every dimension) for all locally strongly regular graphs with $\mu>1$ (Theorem~\ref{thm:curvature}); the displayed diameter estimate then follows from the Bonnet--Myers inequality. We mention that locally strongly regular graphs are edge-regular but not necessarily amply regular (see Theorem~\ref{thm:not-amply}), so the curvature formula \eqref{eq:BE} for amply regular graphs does not apply in general. Here, a graph is called \emph{edge-regular} if it is regular and any two adjacent vertices have a constant number of common neighbours.

For rank-3 graphs, $\mu=1$ forces $\lambda=0$ \cite{Kantor}, so Theorem~\ref{thm:local-finiteness} in particular proves Conjecture~\ref{conj:soicher}. In fact, we prove the stronger result that rank 3 of every local graph suffices for finiteness, even when the local parameters are allowed to vary.
\begin{theorem}\label{cor:soicher}
Every graph whose local graphs all have rank 3 is finite.
\end{theorem}

We prove Theorem~\ref{cor:soicher} by showing that every rank-3 graph is an $\mathscr{S}$-graph. For a graph $H$, write $N_H[x]=N_H(x)\cup\{x\}$ and $N_H(A)=\bigcup_{x\in A}N_H(x)$; let $H[A]$ denote the induced subgraph on $A$ and $\delta(H[A])$ its minimum degree. A finite graph $H$ is an \emph{$\mathscr{S}$-graph} if $|N_H(c)\cap N_H(A)|>\delta(H[A])$ for every vertex $c$ and every nonempty set $A\subseteq V(H)\setminus N_H[c]$. Weetman \cite[Proposition~3.1]{Weetman} proved that, if $\mathcal H$ is a family of $\mathscr{S}$-graphs of maximum degree at most $k$, then every locally $\mathcal H$ graph has diameter at most $k+1$. The three classes in \eqref{eq:weetman-cases} satisfy this condition, leading him to the following stronger conjecture.

\begin{conjecture}[Weetman {\cite{Weetman}}]\label{conj:weetman-s}
Every strongly regular graph is an $\mathscr{S}$-graph.
\end{conjecture}

Brouwer \cite{BrouwerWeetman} extended the condition $\mu>\lambda$ to $\mu\ge\lambda$ and proved this conjecture for all strongly regular graphs on at most $195$ vertices. However, we show that the Cameron graph ($231$ vertices) and the twisted Grassmann graphs $\widetilde J_q(5,2)$, for prime powers $q\ge3$, are not $\mathscr{S}$-graphs (Theorems~\ref{thm:cameron} and~\ref{thm:twisted}). Thus, Conjecture~\ref{conj:weetman-s} does not hold in general.

Weetman \cite[p.~219]{Weetman} also noted that he knew no regular graph of diameter two with an infinite extension, leaving the existence of such a graph open. We construct an edge-regular graph on $192$ vertices that has diameter two and admits both an infinite extension and finite extensions of unbounded diameter (Theorem~\ref{thm:edge-regular} and Corollary~\ref{cor:finite-extensions}).

These results are proved in Section~\ref{sec:local}, with a spectral completion in Appendix~\ref{app:local-spectral}.

\section{Generalized Bakry--\'Emery curvature}\label{sec:prelim}
\subsection{Bakry--\'Emery curvature and the Bonnet--Myers bound}\label{sec:classical-be}
On a Riemannian manifold, the Bochner identity relates the Laplacian of the squared gradient to the Hessian and Ricci curvature. The $\Gamma$-calculus of Bakry and \'Emery \cite{BE} gives this relation an algebraic formulation in terms of quadratic forms associated with a Markov generator. For the Laplace--Beltrami operator, these forms satisfy $\Gamma(f)=|\nabla f|^2$ and $\Gamma_2(f)=|\nabla^2f|^2+\operatorname{Ric}(\nabla f,\nabla f)$, so a lower Ricci curvature bound yields an inequality between $\Gamma_2$ and $\Gamma$. Their algebraic definitions also apply to graph Laplacians. Bakry--\'Emery curvature on graphs has been widely studied in recent years; see, e.g., \cite{CKL,CKLP,CLP,HuaLin,KKRT,LY,LMPRigidity}.

For any two vertices $x,y$, we write $x\sim y$ if they are adjacent. In this paper, we use the standard Laplacian
\[
 \Delta f(x)=\sum_{y\sim x}(f(y)-f(x)).
\]
On a $k$-regular graph, $\Delta=k(P-I)$, where $I$ is the identity operator and $P$ is uniform neighbour averaging. The associated \emph{carr\'e du champ} and iterated form are
\begin{align}
 \Gamma(f,g)&=\tfrac12\{\Delta(fg)-f\Delta g-g\Delta f\},
 &\Gamma(f)&=\Gamma(f,f),\label{eq:classical-gamma}\\
 \Gamma_2(f)&=\tfrac12\Delta\Gamma(f)-\Gamma(f,\Delta f).
 \label{eq:classical-gamma2}
\end{align}
In particular, $\Gamma(f)(x)=\tfrac12\sum_{y\sim x}(f(y)-f(x))^2$. For $K\in\mathbb R$ and $0<\mathcal N\le\infty$, the \emph{curvature-dimension inequality} $\mathrm{CD}(K,\mathcal N)$ at $x$ is
\begin{equation}\label{eq:classical-cd}
 \Gamma_2(f)(x)\ge K\Gamma(f)(x)+\frac{(\Delta f(x))^2}{\mathcal N}
 \qquad\text{for every real-valued }f,
\end{equation}
with $1/\infty=0$. The largest such $K$ is denoted by $K_{G,x}(\mathcal N)$. The \emph{Bakry--\'Emery curvature} at $x$ is $K_{BE}(x)=K_{G,x}(\infty)$. The graph satisfies $\mathrm{CD}(K,\mathcal N)$ if the inequality holds at every vertex. For the normalized Laplacian $\Delta/k$ of a regular graph, the corresponding curvature is $K_{G,x}(\mathcal N)/k$. All curvature values below use the standard convention unless stated otherwise.

The basic passage from positive curvature to finite diameter is the following graph Bonnet--Myers theorem.
\begin{corollary}[Bonnet--Myers bound {\cite[Corollary~2.2]{LMP}}]\label{cor:bonnet}
If a graph $G$ of maximum degree at most $k$ satisfies $\mathrm{CD}(K,\infty)$ with $K>0$, then it is finite and
\begin{equation}\label{eq:classical-bm}
 \diam(G)\le\floor{\frac{2k}{K}}.
\end{equation}
\end{corollary}
The heat semigroup $H_t=e^{t\Delta}$ satisfies $\Gamma(H_tf)\le e^{-2Kt}H_t\Gamma(f)$ under the curvature hypothesis. Together with $|\Delta g|^2\le2k\Gamma(g)$, integration at two endpoints gives \eqref{eq:classical-bm}. The argument works before finiteness is known. We give its extension to other gradient forms below.

For amply regular graphs, the explicit formula of \cite[Theorem~4.3]{CHLZ} makes both the strength and the limitation of this bound particularly transparent. If $\mu\ge2$, then
\begin{equation}\label{eq:BE}
 K_{BE}(x)=2+\frac\lambda2+
 \min\left\{0,\min_{\theta\in\Spec G[N(x)]}
 \frac{k(\mu-2)-2\theta(\lambda-\theta)}{\mu}\right\}.
\end{equation}
Here the local graph is $\lambda$-regular; including its principal eigenvalue does not change the formula, since its correction is nonnegative and the outer minimum includes zero. Since $2\theta(\lambda-\theta)\le\lambda^2/2$, we obtain
\begin{equation}\label{eq:classical-arg-threshold}
 \mu>2,\quad k\ge\frac{\lambda^2}{2(\mu-2)}
 \quad\Longrightarrow\quad
 K_{BE}(x)=2+\frac\lambda2,\qquad
 \diam(G)\le\floor{\frac{4k}{\lambda+4}}.
\end{equation}
The curvature has already reached its upper bound $2+\lambda/2$. Increasing $k$ further leaves this curvature unchanged, so \eqref{eq:classical-bm} continues to give the same linear dependence on $k$. To obtain a stronger diameter estimate from the additional local expansion, we change the gradient form whose decay is measured.

There is a second obstruction for longer-scale local conditions. If $s\ge4$, the radius-two ball of a regular $(s,c,a,k)$-graph is a tree. Its classical Bakry--\'Emery curvature is therefore $2-k$, as for the $k$-regular tree; see \cite{CLP}. In particular, this curvature cannot detect $a_{s-1}=a$ or $c_s=c$. We must also allow the generator and the gradient to reach the layers where those parameters occur.

\subsection{Modified gradients and diameter estimates}\label{sec:generalized-be}
Let $L$ be a finite-range Markov generator on $G$:
\begin{equation}\label{eq:markov-generator}
 Lf(x)=\sum_y q(x,y)(f(y)-f(x)),\qquad
 q(x,y)\ge0,\qquad \sup_x\sum_y q(x,y)<\infty.
\end{equation}
Thus, $L$ is bounded on $\ell^\infty(V(G))$, and $H_t=e^{tL}$ is a Markov semigroup even when $G$ is infinite. Define $\Gamma_L$ and $\Gamma_{2,L}$ by \eqref{eq:classical-gamma}--\eqref{eq:classical-gamma2} with $\Delta$ replaced by $L$.

We choose independently a symmetric bilinear local form $\Gcal$, with $\Gcal(f)=\Gcal(f,f)\ge0$ and $\Gcal(f,\one)=0$. Here, $\one$ denotes the constant function with value one. Locality means that its value at $x$ depends on a ball of fixed radius about $x$; we also require that it define a bounded bilinear map on $\ell^\infty(V(G))$. These conditions hold for all the forms used below. The \emph{curvature inequality for the pair $(\Gcal,L)$} is
\begin{equation}\label{eq:generalized-curvature}
 \Bcal_{\Gcal,L}(f):=\tfrac12L\Gcal(f)-\Gcal(f,Lf)
 \ge K\Gcal(f).
\end{equation}
The choice $\Gcal=\Gamma_L$ recovers $\mathrm{CD}(K,\infty)$ for $L$. Here we choose the pair using the prescribed local intersection data, and keep track of how its energy controls the original graph distance.

For bounded $f$ and $0\le u\le t$, differentiation gives
\[
 \frac{d}{du}H_u[\Gcal(H_{t-u}f)]
 =2H_u[\Bcal_{\Gcal,L}(H_{t-u}f)].
\]
Consequently, \eqref{eq:generalized-curvature} is equivalent to
\begin{equation}\label{eq:generalized-gradient-decay}
 \Gcal(H_tf)\le e^{-2Kt}H_t\Gcal(f).
\end{equation}
Indeed, one direction follows by differentiating $e^{-2Ku}H_u[\Gcal(H_{t-u}f)]$; the converse follows by differentiating \eqref{eq:generalized-gradient-decay} at $t=0$. This equivalence also appears in \cite[Section~2.1.1, Lemma~1]{Monmarche}.

A graph is called \emph{locally-finite} if every vertex has finite degree.

\begin{lemma}\label{lem:semi}
Let $G$ be a locally-finite graph, and let $L$ be a finite-range Markov generator on $G$ satisfying \eqref{eq:markov-generator}. Let $\Gcal$ be a symmetric local bilinear form, bounded on $\ell^\infty(V(G))$, with $\Gcal(f)\ge0$ and $\Gcal(f,\one)=0$. Assume that \eqref{eq:generalized-curvature} holds at every vertex with $K>0$. Suppose $\|\Gcal(f)\|_\infty\le M_0$ for every bounded $1$-Lipschitz function $f$, and that
\[
 |Lg(x)|\le C\sqrt{\|\Gcal(g)\|_\infty}
\]
for every bounded $g$ and every vertex $x$. Assume also that for each fixed pair $x,y$ there is a finite constant $C_{x,y}$ such that
\begin{equation}\label{eq:energy-controls-distance}
 |g(x)-g(y)|\le C_{x,y}\sqrt{\|\Gcal(g)\|_\infty}
\end{equation}
for every bounded $g$. Then $G$ is finite and
\begin{equation}\label{eq:generalized-bm}
 \diam(G)\le\frac{2C\sqrt{M_0}}K.
\end{equation}
If also $\|Lf\|_\infty\le M_1$ for every bounded $1$-Lipschitz function $f$, where $0<M_1\le C\sqrt{M_0}$, then
\begin{equation}\label{eq:capped-bm}
 \diam(G)\le\frac{2M_1}{K}
 \left(1+\log\frac{C\sqrt{M_0}}{M_1}\right).
\end{equation}
All Lipschitz constants and distances refer to the original graph.
\end{lemma}
\begin{proof}
By \eqref{eq:generalized-gradient-decay},
\[
 \|\Gcal(H_tf)\|_\infty\le M_0e^{-2Kt},\qquad
 \|LH_tf\|_\infty\le C\sqrt{M_0}e^{-Kt}.
\]
The differences $H_tf(x)-H_tf(y)$ tend to zero by \eqref{eq:energy-controls-distance}. Integrating $\partial_tH_tf=LH_tf$ at both endpoints therefore gives
\[
 |f(x)-f(y)|\le2\int_0^\infty C\sqrt{M_0}e^{-Kt}\,dt.
\]
For the second assertion, $LH_tf=H_tLf$ and Markov contraction give the additional bound $\|LH_tf\|_\infty\le M_1$. Integrate $\min\{M_1,C\sqrt{M_0}e^{-Kt}\}$, splitting at $K^{-1}\log(C\sqrt{M_0}/M_1)$, to obtain \eqref{eq:capped-bm}. For fixed $x,y$, apply these estimates to the bounded $1$-Lipschitz function $f(z)=\min\{\dist(x,z),\dist(x,y)\}$, for which $f(y)-f(x)=\dist(x,y)$. This gives the asserted diameter bounds; local finiteness then implies finiteness.
\end{proof}

For $L=\Delta$ and $\Gcal=\Gamma$, we have $M_0=k/2$ and $C=\sqrt{2k}$, and edge differences are bounded by $\sqrt{2\|\Gamma(g)\|_\infty}$. Thus, \eqref{eq:generalized-bm} recovers the classical bound \eqref{eq:classical-bm}. The second estimate is useful when a modified gradient gives a faster decay but a weaker comparison with edge differences.

\begin{corollary}\label{cor:log-conversion}
Let $G$ be $k$-regular, and let $\Gcal$ be a symmetric local bilinear form bounded on $\ell^\infty(V(G))$. If
\[
 \delta\Gamma\le\Gcal\le\Gamma,\qquad
 \Bcal_{\Gcal,\Delta}\ge K\Gcal,
 \qquad 0<\delta\le1,\quad K>0,
\]
then $G$ is finite and
\begin{equation}\label{eq:log-conversion}
 \diam(G)\le\frac{2k}{K}\left(1+\frac12\log\frac1\delta\right).
\end{equation}
\end{corollary}
\begin{proof}
Apply Lemma~\ref{lem:semi} with $M_0=k/2$, $C=\sqrt{2k/\delta}$ and $M_1=k$. The lower comparison with $\Gamma$ gives \eqref{eq:energy-controls-distance} along a path joining each fixed pair of vertices, so the hypotheses of the lemma hold and the stated bound follows.
\end{proof}

For the logarithmic bound on amply regular graphs we use
\begin{equation}\label{eq:modified-gradient}
 \mathcal A_\delta(f,g)=\Gamma(f,g)
       -\frac{1-\delta}{2k}\Delta f\Delta g,
 \qquad 0<\delta\le1.
\end{equation}
The inequality $(\Delta f)^2\le2k\Gamma(f)$ gives $\delta\Gamma\le\mathcal A_\delta\le\Gamma$. More explicitly,
\[
 \mathcal A_\delta(f)
 =\frac{k}{2}\bigl(P(f^2)-(Pf)^2\bigr)
       +\frac{k\delta}{2}(Pf-f)^2.
\]
Thus, the variation among neighbours keeps its full weight, while the square of the mean increment has weight $\delta$. This changes the constant local mode (equal values on the neighbours after setting the root value to zero) responsible for the upper bound in \eqref{eq:BE}. Expanding the Bochner form gives
\begin{equation}\label{eq:modified-bochner}
 \Bcal_{\mathcal A_\delta,\Delta}(f)
 =\Gamma_2(f)-\frac{1-\delta}{2k}\Gamma(\Delta f).
\end{equation}
A suitable choice of $\delta$ allows the curvature lower bound to grow with $k$. Combining this decay with the initial increment bound through Corollary~\ref{cor:log-conversion} gives the $\lambda$-dependent estimate used in Theorem~\ref{thm:arg-logarithmic}\textup{(i)}. At $\delta=0$, the same family reduces to the local variance, up to the factor $k/2$; discrete contraction of this variance, combined with a suitable averaging operator, gives the corresponding $\mu$-dependent estimate.

Another useful choice averages an energy over nearby roots. If a bounded positive finite-range linear operator $T$ commutes with $L$, then
\begin{equation}\label{eq:averaged-gradient}
 \widetilde{\Gcal}(f,g)=T[\Gcal(f,g)]
 \quad\Longrightarrow\quad
 \Bcal_{\widetilde{\Gcal},L}(f)=T[\Bcal_{\Gcal,L}(f)].
\end{equation}
This follows directly from the definition. For example, $T=I+P$ and $L=\Delta$ combine the Bochner forms at a vertex and its neighbours. Their simultaneous zero directions are constrained by the shared local geometry. This is the mechanism used in Theorem~\ref{thm:arg-242}.

The limiting form $\mathcal A_0=k[P(f^2)-(Pf)^2]/2$ also explains the discrete variance arguments in Sections~\ref{sec:amply} and~\ref{sec:longscale}. There we prove directly an inequality of the form $V(Qf)\le\rho QV(f)$, with $Q$ a Markov kernel and $\rho<1$. Iteration gives geometric energy decay, and summing the smaller of the initial increment bound and the decaying bound produces the discrete analogue of \eqref{eq:capped-bm}. Since variance alone may vanish on nonconstant functions, its control of the original metric is checked in those arguments.

\subsection{Longer-scale generators and gradients}\label{sec:jump-tools}
We now adapt the pair $(\Gcal,L)$ to the layers of an $(s,c,a,k)$-graph. Let $G$ be a $k$-regular $(s,c,a,k)$-graph with $\diam(G)\ge s$, and put $r=k-1$. Let $A_i$ be distance-$i$ adjacency, with $A_0=I$ and $A_1=A$. For $1\le i<s$, sphere sizes are $k_i=kr^{i-1}$. Set
\[
 L_i=A_i-k_iI,\qquad P_i=A_i/k_i,\qquad\Gamma_i=\Gamma_{L_i}.
\]
In particular, $L_1=\Delta$. For $s\ge3$, short-path counting gives $A^2=kI+A_2$ and
\[
 AA_i=rA_{i-1}+A_{i+1}\quad(2\le i\le s-2).
\]
Thus, all $A_i$ with $i<s$ are polynomials in $A$ and commute. More generally, $P_i$ will denote spherical averaging whenever its sphere size is fixed by the parameters. In the unequal-degree case the commutation assertions must be checked separately.

Changing scale already allows the classical pair $(\Gamma_i,L_i)$ to use intersections as far as distance $2i$. This gives the reference bound of Proposition~\ref{prop:even-reference} when $s=2i$. For the wider finiteness ranges, we also separate the two scales: the target energy may be $\Gamma_i$, while $L$ is $L_j$ or a positive linear combination of distance generators. The mixed form $\Bcal_{\Gamma_i,L_j}$ involves distances up to $i+j$, so the two scales can be chosen to reach distance $s$, where each vertex has $c$ neighbours in the preceding layer. For example, the scale-three proof uses $\Gamma_1$ with $L_2+\beta L_1$; the proof for odd $s\ge5$ uses $\Gamma_q$ with $L_{q+1}+r^qL_1$ when $s=2q+1$.

We further adjust these energies by the mean jump increment. If $L$ commutes with $L_j$, then
\begin{equation}\label{eq:augmentation}
 \Gcal(f)=\Gamma_j(f)+\frac{\tau}{2k_j}(L_jf)^2
 \quad\Longrightarrow\quad
 \Bcal_{\Gcal,L}(f)=\Bcal_{\Gamma_j,L}(f)
                 +\frac{\tau}{2k_j}\Gamma_L(L_jf).
\end{equation}
Indeed, the extra term in the Bochner form is $\tau[L((L_jf)^2)-2(L_jf)L_jLf]/(4k_j)$, and commutation reduces it to the displayed carr\'e du champ. For $\tau\ge0$,
\[
 \Gamma_j\le\Gcal\le(1+\tau)\Gamma_j.
\]
The added square supplies further control of the inner and boundary values in the curvature calculation. These values occur in the original form as well, so they are eliminated jointly with their occurrences in the added square. The choice $j=1$, $L=\Delta$ and $\tau=\delta-1$ instead gives \eqref{eq:modified-gradient}.

At a root with $f(o)=0$, expansion gives the common counting identity
\begin{align}\label{eq:mixed}
 4\Bcal_{\Gamma_i,L_j}(f)(o)
 ={}&\sum_z(A_iA_j)_{oz}f(z)^2
 -2\sum_{u\in S_i(o)}f(u)(A_jf)(u)
 -2\sum_{v\in S_j(o)}f(v)(A_if)(v)\\
 &+k_j\sum_{u\in S_i(o)}f(u)^2
 +k_i\sum_{v\in S_j(o)}f(v)^2
 +2(A_if)(o)(A_jf)(o).\nonumber
\end{align}
Here the sphere sizes are constant and $A_i,A_j$ commute. In Section~\ref{sec:long-finiteness}, the tree structure determines the inner terms, while the edges within $S_{s-1}$ and between $S_{s-1}$ and $S_s$ contribute matrices constrained by $a$ and $c$. Their elimination gives a lower bound in \eqref{eq:generalized-curvature}. For finiteness, any uniform $K>0$ for the chosen pair suffices; it need not be a sharp curvature bound.

Finally, the gradient must control the original metric. For $\Gcal\ge\Gamma_1$ this follows along edges. For $\Gcal\ge\Gamma_j$, $1\le j<s$, two vertices at distance two have a common distance-$j$ neighbour: use their midpoint when $j=1$, or extend a third branch from that midpoint when $j\ge2$ and $k\ge3$. Girth ensures that both resulting distances are exactly $j$. Hence
\[
 |g(x)-g(y)|\le2\sqrt{2\|\Gcal(g)\|_\infty}
 \qquad(\dist(x,y)=2).
\]
For $j\ge2$, we have $s\ge3$; when $a>0$, the graph is triangle-free and has odd girth $2s-1$, so its distance-two graph is connected. This verifies \eqref{eq:energy-controls-distance} for the augmented gradients used in the finiteness proofs. The explicit bounds on $Lg$ are obtained there by grouping jumps according to their tree ancestors. If only component diameters of a distance-$j$ jump graph are known to be at most $M$, sampling an original geodesic every $j$ steps still gives $\diam(G)\le jM+j-1$.

\section{Amply regular graphs}\label{sec:amply}\label{sec:arg-curvature}
We first derive diameter bounds from a local estimate for Lipschitz functions, applied along geodesics. We then obtain logarithmic diameter bounds using modified gradient estimates and variance contraction. We also give explicit families illustrating limitations of the bounds.

\subsection{Transport at scale two}

In this section, we prove Theorem~\ref{thm:arg}. The proof is written in terms of Lipschitz functions, while it can also be formulated in the language of long-scale Ollivier Ricci curvature; see Remark~\ref{rem:longscale-ollivier}.

\begin{reptheorem}{thm:arg}
Let $G$ be an amply regular graph with parameters $(v,k,\lambda,\mu)$, where $\mu>1$ and $\mu\ge\lambda$. Then
\[
 \diam(G)\le
 \begin{cases}
  \displaystyle\frac{2k}{\mu},&\text{if }\diam(G)\text{ is even},\\[12pt]
  \displaystyle\frac{2k-\lambda-2}{\mu}+1,&\text{if }\diam(G)\text{ is odd}.
 \end{cases}
\]
In particular, $\diam(G)\le 2k/\mu+1$.
\end{reptheorem}
We first prove the estimate used in this theorem and its refinement. For $\dist(x,y)=2$, set $C=N(x)\cap N(y)$ and choose $z\in C$. Define
\[
 u=\lambda-\deg_{G[C]}(z),\qquad
 B_{xyz}=\mu+\frac{u(\mu-\lambda)}{\mu-1}.
\]
For every $1$-Lipschitz function $f$ with $f(y)-f(x)=2$, we claim that
\begin{equation}\label{eq:arg-flow}
 2k-\left(\sum_{N(y)}f-\sum_{N(x)}f\right)\ge2B_{xyz}.
\end{equation}
Indeed, normalize $f(x)=-1$, $f(y)=1$, so that $f(z)=0$. Put
\[
 X=N(x)\setminus C,\quad Y=N(y)\setminus C,\quad
 Z=N(z)\setminus\{x,y\}.
\]
The two sets $X\cap Z$ and $Y\cap Z$ both have size $u$. Use all directed graph edges from $X$ to $Z$ and from $Z$ to $Y$, and sum their $f$-increments. A vertex $v\in X$ has degree \(  (\mu-1)+(\lambda-\mu)\mathbf1_{Z}(v) \) in the first relation; the right-hand formula is identical. At $w\in Z$, the incoming degree minus the outgoing degree is \(  (\lambda-\mu)(\mathbf1_X(w)-\mathbf1_Y(w)). \) The terms with coefficient $\lambda-\mu$ therefore cancel exactly. The sum of increments is
\[
 (\mu-1)\left(\sum_Y f-\sum_X f\right).
\]
Each increment is at most one, and the number of directed edges is
\[
 2(k-\mu)(\mu-1)+2u(\lambda-\mu).
\]
Division by $\mu-1$ proves \eqref{eq:arg-flow}; the common set $C$ cancels from the endpoint sums.

\begin{proof}[Proof of Theorem~\ref{thm:arg}]
Since $u\ge0$, \eqref{eq:arg-flow} gives a gain of at least $2\mu$. Fix a root $r$, and on a geodesic $v_0=r,v_1,\ldots,v_n$ write
\[
 F(v)=\sum_{w\in N(v)}\dist(r,w)-k\dist(r,v).
\]
Then $-k\le F(v)\le k$, $F(v_0)=k$, and $F(v_1)=k-\lambda-2$. Applying \eqref{eq:arg-flow} to the radial function gives $F(v_i)-F(v_{i+2})\ge2\mu$.\footnote{The same drift argument, with a larger lower bound on $B_{xyz}$, is used in Theorem~\ref{thm:local}.} For even $n$ telescope from $v_0$; for odd $n$ telescope from $v_1$. Thus,
\[
 n\le2k/\mu\quad(n\text{ even}),\qquad
 n\le(2k-\lambda-2)/\mu+1\quad(n\text{ odd}).
\]
These bounds exclude arbitrarily long geodesics. Since the degree is bounded, the graph is finite, and applying the bounds at $n=\diam(G)$ gives the stated diameter estimates.
\end{proof}

\begin{remark}\label{rem:arg-comparison}
Theorem~\ref{thm:arg} improves all the diameter bounds recalled in Section~1.1 when $\mu>3$. Write $D=\diam(G)$. The theorem gives $2k\ge D\mu$ for even $D$ and $2k\ge(D-1)\mu+\lambda+2$ for odd $D$. Substitution into $\max\{3,k-2\mu+4\}$ and \eqref{eq:intro-terwilliger-bound}, treating $D\le3$ and $D\le5$ separately, gives the bounds of Brouwer--Cohen--Neumaier and Terwilliger in this range.

For the bound of \cite{CHLZ}, put $h=2+\max\{\lambda/2,\ceil{\lambda(\mu-\lambda)/(\mu-1)}\}$. One has $h\le\min\{\mu,\lambda+2\}$. The two parity estimates therefore imply $Dh\le2k$. If $\mu>\lambda\ge1$, then $h\ge3$, which also gives the bound $D\le2k/3$ of \cite{HLX}.

For the bound of \cite{CL}, put $g=2+\ceil{\lambda(\mu-\lambda)/(\mu-1)}$. Then $g\le\min\{\mu,\lambda+2\}$. We may assume $D\ge4$. For even $D$, the theorem gives $2k-4\mu\ge(D-4)\mu\ge(D-4)g$. For odd $D\ge5$, it gives
\[
 2k-4\mu\ge(D-4)\mu-(\mu-\lambda-2)\ge(D-4)g,
\]
where the last step uses $(D-4)(\mu-g)\ge\mu-\lambda-2$. Hence $D\le2(k-2\mu)/g+4$, including the original outer rounding because $D-4$ is an integer.

\end{remark}

\begin{theorem}\label{thm:local}
Let $G$ be an amply regular graph with parameters $(v,k,\lambda,\mu)$ and $\mu=\lambda+1>1$. Then
\[
 \diam(G)\le\floor{\frac{2k}{\lambda+1+1/\lambda}}.
\]
\end{theorem}
\begin{proof}
For each distance-two pair, the graph $C$ induced by its common neighbours is not complete. Otherwise an edge in $C$ would have the two original endpoints and the other $\mu-2$ vertices of $C$ as common neighbours, giving at least $\mu=\lambda+1$ common neighbours. Choose a nonuniversal $z\in C$. Then $\deg_C(z)\le\mu-2=\lambda-1$, so $u\ge1$ and
\[
 B_{xyz}\ge\lambda+1+1/\lambda=:B.
\]
The radial argument in the proof of Theorem~\ref{thm:arg}, with $B$ in place of $\mu$, gives $\diam(G)\le2k/B$ when $\diam(G)$ is even, and $\diam(G)\le(2k-\lambda-2)/B+1$ when $\diam(G)$ is odd. Since $B\le\lambda+2$, the latter is also at most $2k/B$. Taking the integer floor therefore gives the claimed bound.
\end{proof}

\begin{remark}\label{rem:longscale-ollivier}
Ollivier Ricci curvature compares random-walk measures through the Wasserstein distance \cite{Ollivier}. Lin, Lu, and Yau \cite{LLY} modified this notion by considering its first-order behaviour as the probability of staying at the initial vertex tends to one. Evaluating these curvatures on pairs of nonadjacent vertices gives the long-scale versions studied in \cite{CKLong}. Through Kantorovich duality, the Lipschitz estimate above has an interpretation in this framework at scale two. For further work on these curvature notions, see, e.g., \cite{BCLMP,CLLY,CLY,JostLiu,LuWang,MW}.
\end{remark}

We expect Conjecture~\ref{conj:qpk} to extend to $(s,c,a,k)$-graphs, and hence propose the following conjecture.
\begin{conjecture}\label{conj:scak}
For every integer $s\ge2$, there exists a constant $C_s$ depending only on $s$ such that every connected $(s,c,a,k)$-graph $G$ with $c>1$ and $c\ge a$ satisfies
\[
 \diam(G)\le C_s\frac{k}{c}.
\]
\end{conjecture}

\begin{example}
The cycle $C_{2s}$ is an $(s,2,0,2)$-graph of diameter $s$, so the constant in Conjecture~\ref{conj:scak} must satisfy $C_s\ge s$. Doubled Odd graphs $G$ are $(3,2,0,k)$-graphs satisfying $\diam(G)=2k-1$; they additionally force $C_3\ge4$.
\end{example}

The following construction disproves Conjecture~\ref{conj:bcn}.

\begin{theorem}\label{thm:conjectures}
For every $m\ge2$, the $m$-fold Cartesian power $G$ of the icosahedron is an amply regular graph with
\[
 (v,k,\lambda,\mu)=(12^m,5m,2,2),\qquad \diam(G)=3m.
\]
Thus, $\mu>1$ and $\diam(G)>2k/(\lambda+2)$, although $G$ is not the icosahedron.
\end{theorem}
\begin{proof}
The icosahedron has parameters $(12,5,2,2)$ and diameter three. In a Cartesian power, an edge changes one coordinate, so its common neighbours number two. A distance-two pair either differs at distance two in one coordinate, giving two common neighbours, or differs along edges in two coordinates, giving the two possible intermediate vertices. Orders multiply, while valencies and diameters add. Therefore, $\diam(G)=3m>5m/2=2k/(\lambda+2)$.
\end{proof}

We tend to believe that Conjecture~\ref{conj:bcn} still holds when $\mu> 2$.

\begin{conjecture}
    Any amply regular graph with parameters $(v,k,\lambda,\mu)$ with $\mu> 2$ has diameter at most $2k/(\lambda +2)$.
\end{conjecture}

\subsection{Logarithmic diameter bounds and their limitations}
For matrices, $I,J,\one$ denote the identity matrix, all-ones matrix, and all-ones vector of the appropriate size. Fix a vertex of an amply regular graph with parameters $(v,k,\lambda,\mu)$. Let $F$ be adjacency on $S_1$, $B$ incidence from $S_1$ to $S_2$, and $r_1=k-\lambda-1$. Counting two-walks gives
\begin{equation}\label{eq:arggram}
 BB^{\T}=(k-\mu)I+(\lambda-\mu)F-F^2+(\mu-1)J.
\end{equation}
Completing the $S_2$ values in the Bochner form recovers \eqref{eq:BE}. The constant mode has curvature $(\lambda+4)/2$, and a transverse local eigenvalue $\theta$ has coefficient
\begin{equation}\label{eq:argK}
 \frac{\lambda+4}{2}
       +\frac{(\mu-2)k-2\theta(\lambda-\theta)}{\mu}.
\end{equation}
For applications at $s=2$, substitute $\lambda=a$ and $\mu=c$ in these identities. We next apply the modified gradient from Section~\ref{sec:generalized-be} to strengthen the diameter estimate.

\begin{reptheorem}{thm:arg-logarithmic}
Let $G$ be an amply regular graph with parameters $(v,k,\lambda,\mu)$, where $\mu>2$.
\begin{enumerate}[label=\textup{(\roman*)}]
\item If $k\ge3\lambda^2/\mu$, then
\[
 \diam(G)\le4+8\log\frac{k}{\max\{\lambda,\mu\}}.
\]
\item For any $0\le\eta<1$, if $k\ge16\lambda^2/\mu^\eta$, then
\[
 \diam(G)\le\frac{8\log k}{(1-\eta)\log\mu}+10.
\]
\item If $\lambda=0$, then
\[
 \diam(G)\le\frac{4\log k}{\log\mu}+10.
\]
\end{enumerate}
\end{reptheorem}

Exact verifications of the scalar estimates and finite parameter checks used in the proof of Theorem~\ref{thm:arg-logarithmic} are provided in the supplementary material.

We use the following variance estimate for the discrete-time parts of the proof. Throughout this subsection, $P$ is the simple random-walk operator and the incidence matrices are taken at the indicated root.

\begin{lemma}\label{lem:arg}
Let $G$ be an amply regular graph with parameters $(v,k,\lambda,\mu)$, where $\mu>2$. Let $P$ be the simple random-walk operator, and put $V(f)=P(f^2)-(Pf)^2$ and
\[E=\max\{2\lambda(\mu-\lambda),\lambda(\lambda-\mu)/4\},\quad u=(k-\mu+E)/(\mu k).\]
If $0\le u<1/2$ and $\rho=u/(1-u)$, then, for every bounded function $f$,
\begin{equation}\label{eq:arg-variance-contraction}
 V(Pf)\le\rho\,PV(f).
\end{equation}
If also $\lambda\ge1$ and $u>0$, then
\begin{equation}\label{eq:sum}
 \diam(G)\le2\sum_{n\ge0}\min\{1,\sqrt{2k}\rho^{n/2}\}
 \le2\left\lceil\frac{\log(2k)}{\log(1/\rho)}\right\rceil
 +\frac2{1-\sqrt\rho}.
\end{equation}
If $\lambda\ge1$ and $u=0$, then $\diam(G)\le2$.
\end{lemma}
\begin{proof}
At a root $x$, let $L$ be local adjacency and $B$ incidence to $S_2(x)$. Choose $Y$ uniformly in $S_1(x)$, then $Z$ uniformly among its neighbours. The terminal multiplicities are $k,\lambda,\mu$ on $\{x\},S_1,S_2$. For $\lambda>0$, conditional expectation $T:f(Z)\mapsto\mathbb E(f(Z)\mid Y)$ has
\[TT^*=\frac1k\left(\frac Jk+\frac{L^2}\lambda+\frac{BB^T}\mu\right).\]
By \eqref{eq:arggram}, its transverse eigenvalues are
\[\frac1\mu+
 \frac{(\mu-\lambda)\theta(\theta-\lambda)/(\lambda\mu)-1}{k}\le u,\]
using $-\lambda^2/4\le\theta(\theta-\lambda)\le2\lambda^2$. When $\lambda=0$, $L=0$ and the term $L^2/\lambda$ is omitted; the transverse eigenvalue is then $(k-\mu)/(\mu k)=u$. Thus, the law of total variance gives
\[
 V(Pf)(x)=\Var(\mathbb E[f(Z)\mid Y])
 \le u\Var(f(Z))
 =u\bigl(PV(f)(x)+V(Pf)(x)\bigr),
\]
which proves \eqref{eq:arg-variance-contraction}.

For bounded $1$-Lipschitz $f$, we have $V(f)\le1$, so iteration gives $\|V(P^nf)\|_\infty\le\rho^n$. Any two neighbours $y,z$ of $x$ satisfy $|h(y)-h(z)|^2\le2kV(h)(x)$. If $\lambda\ge1$, every edge lies in a triangle, and hence $\Lip(P^nf)\le\sqrt{2k}\rho^{n/2}$. This also bounds $\|P^{n+1}f-P^nf\|_\infty$, while Markov contraction bounds the same increment by one. Their minimum is summable, giving uniform convergence; the Lipschitz bound forces the limit to be constant. Summing at both endpoints and choosing a truncated distance function gives the first bound in \eqref{eq:sum}. Splitting the series at the stated ceiling gives the second. When $\lambda\ge1$ and $u=0$, $Pf$ is constant and $\diam(G)\le2$ directly.
\end{proof}

The following averaging operator improves the passage from variance contraction to a diameter bound.

\begin{lemma}\label{lem:arg-smoothing}
Let $G$ be an amply regular graph with parameters $(v,k,\lambda,\mu)$, and let $P$ be the simple random-walk operator. Suppose that $V(f)=P(f^2)-(Pf)^2$ satisfies $V(Pf)\le\rho PV(f)$ for every bounded function $f$, where $0\le\rho<1$. Put $a=(\lambda+2)/(\lambda+4)$ and $C^2=k/(\lambda+4)-a^2$. Then
\begin{equation}\label{eq:arg-smoothed-series}
 \diam(G)\le2a+2\sum_{n\ge0}\min\{1,C\rho^{n/2}\}.
\end{equation}
\end{lemma}
\begin{proof}
Set $S=aP+(1-a)I$. Fix a root $x$ and, for $y\in N(x)$, put
\[
 A_yh=\frac{2h(x)+\sum_{z\in N(x)\cap N(y)}h(z)}{\lambda+2}.
\]
The squared weights of this average sum to $(\lambda+4)/(\lambda+2)^2$. Subtracting the uniform weights on $N(y)$ and applying Cauchy--Schwarz gives
\[
 |Ph(y)-A_yh|^2\le
 \left(\frac{k(\lambda+4)}{(\lambda+2)^2}-1\right)V(h)(y).
\]
Since the local graph on $N(x)$ is $\lambda$-regular, the average of $A_yh$ over $y\in N(x)$ is $[2h(x)+\lambda Ph(x)]/(\lambda+2)$. Multiplying the average of $Ph(y)-A_yh$ by $a$ therefore gives $(P-I)Sh(x)$. Jensen's inequality yields
\begin{equation}\label{eq:arg-smoothing}
 |(P-I)Sh(x)|^2\le C^2PV(h)(x).
\end{equation}
For bounded $1$-Lipschitz $f$, put $g_n=SP^nf$. Then $\|f-g_0\|_\infty\le a$, and variance contraction gives $\|V(P^nf)\|_\infty\le\rho^n$. By \eqref{eq:arg-smoothing}, $\|g_{n+1}-g_n\|_\infty\le C\rho^{n/2}$. Since this increment also equals $SP^n(P-I)f$, its norm is at most one.

It remains to show that $g_n(x)-g_n(y)\to0$ for fixed vertices $x,y$. If $\lambda>0$, every edge has a common neighbour, so $\Lip(P^nf)\le\sqrt{2k}\rho^{n/2}$, and the assertion follows for $g_n=aP^{n+1}f+(1-a)P^nf$. If $\lambda=0$, then $S=(I+P)/2$, and for $x\sim y$,
\[
 |Sh(x)-Sh(y)|
 \le\tfrac12\bigl(|h(x)-Ph(y)|+|h(y)-Ph(x)|\bigr)
 \le\sqrt{k-1}\,\|V(h)\|_\infty^{1/2}.
\]
Here each centered value among $k$ values has squared magnitude at most $k-1$ times their variance. Applying this to $h=P^nf$ and then along a path proves the assertion. Telescoping the increments of $g_n$ at both endpoints now gives \eqref{eq:arg-smoothed-series}, by taking $f(z)=\min\{\dist(x,z),\dist(x,y)\}$.
\end{proof}

\begin{proof}[Proof of Theorem~\ref{thm:arg-logarithmic}]
\emph{The case $\lambda=0$.} If $k=\mu$, vertices at distance two have the same neighbours, excluding a geodesic of length three; thus $\diam(G)\le2$. Assume $k>\mu$. Lemma~\ref{lem:arg} gives variance contraction with $\rho=(k-\mu)/[k(\mu-1)+\mu]<1/(\mu-1)$. In Lemma~\ref{lem:arg-smoothing}, we have $a=1/2$ and $C^2=(k-1)/4\le k/2$. Splitting the series at $\lceil\log(k/2)/\log(1/\rho)\rceil$ gives
\[
 \diam(G)\le1+2\left\lceil\frac{\log(k/2)}{\log(1/\rho)}\right\rceil
 +\frac2{1-\sqrt\rho},
\]
and hence the sharper estimate
\begin{equation}\label{eq:triangle-free-refined}
 \diam(G)\le\frac{2\log(k/2)}{\log(\mu-1)}+10.
\end{equation}
This also holds when $k=\mu$.\footnote{This sharper estimate is used in the proof of Theorem~\ref{th:general}, part~3.} Since $k\ge\mu\ge3$, we have
\[
 \frac{2}{\log(\mu-1)}\le\frac{4}{\log\mu}
 \le\frac{8}{(1-\eta)\log\mu}\quad(0\le\eta<1),
\]
where the first inequality follows from $(\mu-1)^2\ge\mu$. Together with $\log(k/2)\le\log k$, these comparisons prove (iii) and (ii) when $\lambda=0$.

\emph{Part (i).} First suppose $\mu\ge\lambda$, and put $r=k/\mu\ge1$. For $1\le r\le4$, Theorem~\ref{thm:arg} gives $\diam(G)\le2r+1\le4+8\log r$, where the second inequality follows by concavity. For $r\ge4$, use Lemma~\ref{lem:arg} with $E=2\lambda(\mu-\lambda)$. Since $u$ is affine in $1/r$, we have
\[
 u\le\bar u:=\max\left\{\frac1\mu,
 \frac{\mu+2\lambda(\mu-\lambda)}{2\mu^2}\right\},
 \qquad t=\sqrt{\frac{\bar u}{1-\bar u}}.
\]
The integer conditions $0\le\lambda\le\mu$ and $\mu\ge3$ give $\bar u\le7/18$, so $\rho\le t^2<16/25$. Lemma~\ref{lem:arg-smoothing} therefore gives
\begin{equation}\label{eq:arg-mu-series}
 \diam(G)\le H(r):=2a+2\sum_{n\ge0}
 \min\left\{1,\sqrt{\frac{\mu r}{\lambda+4}}\,t^n\right\},
 \qquad a=\frac{\lambda+2}{\lambda+4}.
\end{equation}
At every differentiability point, the derivative of $H$ with respect to $\log r$ is the geometric tail of unsaturated summands. Its first term is at most one, so this derivative is at most $1/(1-t)<5$. Consequently, $H(r)\le H(4)+5\log(r/4)$ for $r\ge4$. Bounding the first two summands by one gives
\[
 H(4)\le M:=2a+4+4\sqrt{\frac\mu{\lambda+4}}\frac{t^2}{1-t}.
\]
We show that $M<15$. For the integers $3\le\mu\le5$ and $0\le\lambda\le\mu$, this can be checked directly. For $\mu\ge6$, put $x=\lambda/\mu$, $e=1/\mu$, $A=x+4e$ and $w=x(1-x)+e/2$, so $\bar u=\max\{e,w\}$. If $\bar u=e$, then $t<1/2$ and $A\ge4e$, whence $M\le6+3\sqrt e/(1-2e)<9$. If $\bar u=w$, then $w\le1/3$, $t<3/4$, and Cauchy--Schwarz gives
\[
 \frac{w^2}{A}\le x(1-x)^2+\frac e{16}
 \le\frac4{27}+\frac1{96}<\frac4{25}.
\]
Using $t^2/(1-t)=w(1+t)/(1-2w)$, we obtain $M<6+4(2/5)(7/4)\cdot3=72/5<15$. Since $15<4+8\log4$, it follows that $\diam(G)\le4+8\log r$ for every $r\ge4$ as well.

Now suppose $\lambda>\mu$. We prove the stronger estimate $\diam(G)\le4+8\log[k/(\lambda+2)]$. Take $\mathcal A=\mathcal A_\delta$ from \eqref{eq:modified-gradient}, with $0<\delta\le1$ to be chosen below and $t=1-\delta$. Then $\delta\Gamma\le\mathcal A\le\Gamma$. Write $\mathcal A_2=\Bcal_{\mathcal A,\Delta}$; by \eqref{eq:modified-bochner},
\[\mathcal A_2=\Gamma_2-\frac{t}{2k}\Gamma(\Delta f).\]
Fix a root, normalize $f(x)=0$, and write $u=f|_{S_1}$, $v=f|_{S_2}$. Let $L$ be local adjacency and $B$ the incidence from $S_1$ to $S_2$. For $r=k-\lambda-1$, \eqref{eq:arggram} and expansion give
\begin{align}
 4\Gamma_2&=(2k+\lambda)\|u\|^2-4u^TLu-4u^TBv+\mu\|v\|^2+2(\one^Tu)^2,\notag\\
 2\Gamma(\Delta f)&=\|Bv-(kI-L)u-(\one^Tu)\one\|^2.\notag
\end{align}
Choose an orthonormal $L$-eigenbasis consisting of the constant vector and vectors $e_\theta\perp\one$. On $\one^\perp$, \eqref{eq:arggram} expresses $BB^T$ as a polynomial in $L$, so the corresponding $B^Te_\theta$ are orthogonal. The forms thus split into the constant block, transverse eigenmodes, and $\ker B$. Minimizing the constant block over $v$ yields
\[K_{\rm const}=
 \frac{(\lambda+4)/2-t(\lambda+2)^2/[2(k-tr)]}{\delta}.\]
For a transverse eigenvalue $\theta\in[-\lambda,\lambda]$, put
\[g=k-\mu+(\lambda-\mu)\theta-\theta^2,\quad B_\theta=\mu k-tg,
 \quad H_\theta=(\mu-2)k+2\mu+2\theta^2-2\lambda\theta.\]
Here $0\le g\le \mu r<\mu k$. When $u=ze_\theta$, $Bv=we_\theta$, $g>0$,
\[4\mathcal A_2=(2k+\lambda-4\theta)z^2-4zw+\mu w^2/g
 -(t/k)[w-(k-\theta)z]^2,\qquad\mathcal A=z^2/2.\]
Minimizing in $w$ gives
\begin{equation}\label{eq:ktheta}
 K_\theta(t)=\frac \lambda2+
 \frac{(2-t)kH_\theta-t\mu\theta^2}{2B_\theta}.
\end{equation}
For $g=0$, set $w=0$; the same formula follows. Pure $\ker B$ vectors contribute $\mu\|v\|^2$ to $4\mathcal A_2$ and zero to $\mathcal A$. On transverse modes, $\mathcal A$ is independent of $t$, while $\mathcal A_2$ decreases with $t$. Thus, $K_\theta(t)\ge K_\theta(1)$. Take $(h,m)=(7,3)$ if $\mu=3$ and $(h,m)=(4,4)$ if $\mu\ge4$. Put $b=\mu k-g>0$ and $A=(5h-11)/6$. Completing the square in $\theta-\lambda/2$ in \eqref{eq:ktheta}, and using $k\ge\mu$ and $\lambda^2\le\mu k/3$, gives
\begin{align*}
 2hb\bigl(K_\theta(1)-k/h\bigr)
 &\ge A(\mu-m)k(k-\mu)
       +\left(A-\frac1{h-2}\right)\mu^2k\\
 &\quad+\left(\frac{11h-12}{12}\mu-h\right)\lambda k
       +\frac h4\mu\lambda^2+h\mu\lambda\ge0.
\end{align*}
Thus, $K_\theta(t)\ge k/h$. Set $R=k/(\lambda+2)\ge1$, $C=(h-2)/2$, and $\delta=\min\{1,C/R\}$. If $\delta=1$, then $K_{\rm const}=(\lambda+4)/2\ge k/h$. Otherwise $k-tr=C(\lambda+2)+t(\lambda+1)$, and
\[
 \frac{t}{k-tr}\le\frac1{C(\lambda+2)+\lambda+1}
\]
again gives $K_{\rm const}\ge k/h$. Corollary~\ref{cor:log-conversion} therefore yields
\[
 \diam(G)\le2h+h\log^+(R/C),\qquad
 \log^+x=\max\{0,\log x\}.
\]
The classical formula \eqref{eq:BE} also gives $\diam(G)\le4k/(\lambda+4)\le4R$ under the present hypotheses. For $\mu\ge4$, the two estimates $4R$ and $8+4\log R$ imply $\diam(G)\le4+8\log R$, using the first for $1\le R\le3$ and the second for $R\ge3$. For $\mu=3$ and $R\ge40$, use $14+7\log(2R/5)\le4+8\log R$.

It remains to consider $\mu=3$ and $R<40$. Since $k\ge\lambda^2$, this is the finite integer range
\[
 0\le\lambda\le41,\qquad
 \max\{3,\lambda+2,\lambda^2\}\le k<40(\lambda+2).
\]
Put $F=4+8\log R$ and $T=k(2+\log(1/\delta))/F$. We can check throughout this range that either $4k/(\lambda+4)\le F$, or $K_{\rm const}\ge T$ and
\[
 (\lambda-2T)B_\theta+(2-t)kH_\theta-t\mu\theta^2\ge0
 \qquad(-\lambda\le\theta\le\lambda).
\]
The expression on the left is quadratic in $\theta$, so it suffices to check the endpoints and its vertex when the latter is a minimum in the interval. In the second case \eqref{eq:ktheta} gives $\mathcal A_2\ge T\mathcal A$, and \eqref{eq:log-conversion} gives $\diam(G)\le F$. Together with the case $\mu\ge\lambda$, this proves (i).

\emph{Part (ii) for $\lambda\ge1$.} If $k\le7$, the classical bound $\diam(G)\le4k/(\lambda+4)\le7$ from \eqref{eq:BE} already proves the assertion, so assume $k\ge8$. Use $E,u,\rho$ from Lemma~\ref{lem:arg}. If $u=0$, the conclusion is immediate, so assume $u>0$. Use the following elementary optimization under $k\ge C\lambda^2/\mu$, $C\ge16$:
\begin{align*}
 \mu<\lambda&: u\le1/\mu+1/(4C),\\
 \lambda\le \mu&: u\le\max\{1/\mu,U_C(\mu)\},\\
 U_C(\mu)&=\begin{cases}2/\sqrt C-2/C,&\mu\ge\sqrt C,\\
 1/\mu+(\mu-2)/C,&\mu\le\sqrt C.\end{cases}
\end{align*}
To verify this, when $E-\mu>0$ replace $k$ by $\max\{\mu,C\lambda^2/\mu\}$ and split at $\lambda=\mu/\sqrt C$; $(2\lambda-1)/\lambda^2$ decreases for $\lambda\ge1$. Otherwise $u\le1/\mu$. At $C=16$ this gives $u\le19/48$, so $\rho\le19/29<(13/16)^2$ and the ceiling and tail cost at most $38/3$.

Set $\tau=1-\eta$, $x=\mu^{\tau/2}$, and $C=16x^2$. We claim $u\le1/(1+x)$. For $\mu<\lambda$, use $u\le1/\mu+1/(64x^2)$ and $x^2\le \mu$, splitting at $x=\sqrt3$. For $\lambda\le \mu$, the case $\mu\ge4x$ follows from $u\le\max\{1/\mu,1/(2x)-1/(8x^2)\}$. For $\max\{3,x^2\}\le \mu\le4x$, convexity of $1/\mu+(\mu-2)/(16x^2)$ reduces to the endpoints. At $\mu=3$, the difference $1/(1+x)-1/3-1/(16x^2)$ decreases on $[1,\sqrt3]$ and is positive at $\sqrt3$. At $\mu=x^2$, the comparison is equivalent to $x^3-15x^2+14x+14\le0$ for $\sqrt3\le x\le4$; its derivative and its value at $\sqrt3$ are negative. The upper endpoint $\mu=4x$ is covered by the preceding case. Thus, $\rho\le\mu^{-\tau/2}$, and \eqref{eq:sum} gives
\[
 \diam(G)\le\frac{4\log(2k)}{\tau\log\mu}+\frac{38}{3}
 \le\frac{8\log k}{\tau\log\mu}+10.
\]
The last inequality uses $k\ge8$ and $\tau\log\mu\le\log k$, completing the proof of (ii).
\end{proof}

We next examine the sharpness of these logarithmic bounds and the limitations imposed by their degree thresholds. The degree threshold in Theorem~\ref{thm:arg-logarithmic}\textup{(i)} does not allow an additional denominator growing arbitrarily with $\lambda$. The next family gives this obstruction even after multiplying that threshold by any prescribed constant.

\begin{theorem}\label{th:bilinear}
Let $g(\lambda)>0$ satisfy $g(\lambda)\to\infty$ as $\lambda\to\infty$. For every $M,C>0$, there is an amply regular graph $G$ with parameters $(v,k,\lambda,\mu)$ such that $\lambda>C$, $\mu>C$,
\[
 k\ge\frac{M\lambda^2}{\mu},\qquad\text{and}\qquad
 \diam(G)>\frac{\log k}{g(\lambda)}+C.
\]
\end{theorem}
\begin{proof}
Let $G$ have all $m\times m$ matrices over $\F_q$ as vertices, adjacent when their difference has rank one. A rank-one matrix is $uv^T$, with each matrix represented $q-1$ times, so
\[
 \begin{aligned}
 v&=q^{m^2},\qquad k=(q^m-1)^2/(q-1),\\
 \lambda&=2q^m-q-2,\qquad \mu=q(q+1),\qquad \diam(G)=m.
 \end{aligned}
\]
Here distance is rank, by rank subadditivity and rank factorizations. To verify $\lambda$, reduce an adjacent pair to $0,E_{11}$. A rank-one $uv^T$ has rank-one difference from $E_{11}$ exactly when it shares its row or column space, except for $E_{11}$ itself. The two families have size $q^m-1$ and overlap in $q-1$ matrices, giving the stated count. For a distance-two pair, reduce to $0,\operatorname{diag}(I_2,0)$. Equality in rank subadditivity forces each common neighbour to have both its row and column spaces in those of the rank-two matrix. It is therefore a rank-one matrix $uv^T$ on the leading block. Since $\det(I_2-uv^T)=1-v^Tu$, the remaining adjacency condition is $v^Tu=1$, giving $(q^2-1)q/(q-1)=q(q+1)$ common neighbours.

At fixed $q$ and $m\to\infty$,
\[
 \frac{k\mu}{\lambda^2}\longrightarrow\frac{q(q+1)}{4(q-1)},\qquad
 \log k=2m\log q+O_q(1),\qquad \lambda\longrightarrow\infty.
\]
Choose a prime power $q$ such that $q(q+1)>C$ and $q(q+1)/[4(q-1)]>M$, and keep $q$ fixed. For all sufficiently large $m$, we then have $\lambda>C$, $\mu>C$ and $k\ge M\lambda^2/\mu$. Moreover, since $g(\lambda)\to\infty$ and $\log k/m\to2\log q$,
\[
 \frac1m\left(\frac{\log k}{g(\lambda)}+C\right)\longrightarrow0.
\]
Increasing $m$ if necessary therefore gives $\diam(G)=m>\log k/g(\lambda)+C$, as required.
\end{proof}

The following examples show that logarithmic order cannot be improved uniformly to $o(\log k)$, even for fixed intersection parameters. For example, applying the following result with $q=2$ yields graphs with $(\lambda,\mu)=(1,3)$, $k=2(2^m-1)$, and $\diam(G)=m$.

\begin{proposition}\label{prop:polar}
For every prime power $q$ and integer $m\ge2$, there is an amply regular graph $G$ with
\[
 k=\frac{q(q^m-1)}{q-1}, \qquad \lambda=q-1, \qquad \mu=q+1, \qquad \diam(G)=m.
\]
\end{proposition}
\begin{proof}
Let $W=\F_q^m\oplus\F_q^m$ carry the nondegenerate alternating form $\langle(x,y),(x',y')\rangle=x\cdot y'-y\cdot x'$. Take the totally isotropic $m$-spaces as vertices, with $L,M$ adjacent when $\dim(L\cap M)=m-1$. For each hyperplane $H$ of $L$, the quotient $H^\perp/H$ is a symplectic plane. Its $q+1$ lines give the totally isotropic $m$-spaces containing $H$, one of which is $L$. Thus, each of the $(q^m-1)/(q-1)$ hyperplanes of $L$ gives $q$ neighbours, proving the degree formula.

We claim that $\dist(L,M)=m-\dim(L\cap M)$. An edge changes the intersection dimension with $M$ by at most one, giving the lower bound. For the reverse inequality, choose $v\in M\setminus L$. Since $L^\perp=L$, the pairing of $v$ with $L$ is nonzero, so $H=L\cap v^\perp$ is a hyperplane containing $L\cap M$. The space $H+\langle v\rangle$ is a neighbour of $L$ whose intersection with $M$ has dimension $\dim(L\cap M)+1$. Iteration proves the distance formula, and the two coordinate $m$-spaces show that the diameter is $m$.

Suppose first that $L,M$ are adjacent, with intersection $H$. Every common neighbour $Z$ contains $H$: otherwise $Z\cap L$ and $Z\cap M$ supply perpendicular vectors representing the two distinct lines $L/H,M/H$ in the symplectic plane $H^\perp/H$, which is impossible. Of the $q+1$ totally isotropic $m$-spaces containing $H$, all except $L,M$ are common neighbours, so $\lambda=q-1$.

If $\dist(L,M)=2$, put $K=L\cap M$. A common neighbour $Z$ contains $K$, since the two hyperplanes $Z\cap L$ and $Z\cap M$ of $Z$ intersect in a subspace of $K$ of dimension at least $m-2$. They span $Z$, so $Z\subseteq L+M=K^\perp$. The disjoint $2$-spaces $L/K,M/K$ pair nondegenerately in $K^\perp/K$. For each of the $q+1$ lines of $L/K$, there is exactly one perpendicular line of $M/K$; their sum is a totally isotropic $2$-space and lifts to a common neighbour. These are all the common neighbours, giving $\mu=q+1$ and completing the proof.
\end{proof}

\begin{corollary}\label{th:negative}
Let $g,h:(0,\infty)\to(0,\infty)$ satisfy $g(t),h(t)\to\infty$ as $t\to\infty$. For any constant $C\ge 1$ and functions $M,B:(0,\infty)^2\to(0,\infty)$, there is an amply regular graph $G$ with parameters $(v,k,\lambda,\mu)$ such that $\lambda>C$, $\mu>C$,
\[
 k>M(\lambda,\mu),\qquad\text{and}\qquad
 \diam(G)> \frac{\log_\lambda k}{g(\lambda)} + \frac{\log_\mu k}{h(\mu)} +B(\lambda,\mu).
\]
\end{corollary}
\begin{proof}
Use Proposition~\ref{prop:polar}, so $\lambda=q-1$, $\mu=q+1$, $\diam(G)=m$ and $\log k=m\log q+O_q(1)$. Since $g(q-1),h(q+1)\to\infty$, choose a prime power $q>\max\{2,C+1\}$ such that
\[
 \alpha_q:=\frac{\log q}{g(q-1)\log(q-1)}
 +\frac{\log q}{h(q+1)\log(q+1)}<1.
\]
Keep $q$ fixed. Then $\lambda>C$, $\mu>C$ and
\[
 \frac{\log_\lambda k}{g(\lambda)}+\frac{\log_\mu k}{h(\mu)}
 =\alpha_q m+O_q(1).
\]
As $m\to\infty$, both $k$ and the difference between $m$ and this expression tend to infinity. Since $M$ and $B$ are now fixed, taking $m$ sufficiently large gives both required inequalities.
\end{proof}

\subsection{\texorpdfstring{Amply regular Terwilliger graphs with $\mu> 2$}{Amply regular Terwilliger graphs with mu > 2}}
In an amply regular graph, the Terwilliger condition is equivalent to requiring the common neighbours of every pair at distance two to induce a clique.

\begin{reptheorem}{thm:terw}
Let $G$ be an amply regular Terwilliger graph with parameters $(v,k,\lambda,\mu)$, where $\mu>2$. Then $G$ is finite and
\[
 \diam(G)\le\floor{\frac{4k}{\lambda+4}}.
\]
In particular, every distance-regular graph of diameter $D$ and valency $k$, with $c_2>2$, satisfies
\[
 D\le\floor{\frac{4k}{a_1+4}}.
\]
\end{reptheorem}
We use the following local structure theorem. A \emph{$t$-clique extension} replaces each vertex by a clique of size $t$ and each edge by all edges between its two cliques.
\begin{lemma}[{\cite[Theorem~1.16.3]{BCN}}]\label{lem:terwilliger-local}
Let $G$ be an amply regular Terwilliger graph with parameters $(v,k,\lambda,\mu)$, where $\mu>1$. For every vertex $x$, the graph $G[N(x)]$ is a $t$-clique extension of a connected noncomplete strongly regular Terwilliger graph $\Sigma$, for some positive integer $t$. If $\Sigma$ has parameters $(v_0,k_0,\lambda_0,\mu_0)$, then
\begin{equation}\label{eq:local-structure}
 k=t{v_0},\quad \lambda=t({k_0}+1)-1,\quad \mu=t{\mu_0}+1,\quad t\le{\lambda_0}+1.
\end{equation}
\end{lemma}
The proof uses only the finiteness of each set of neighbours and does not require finiteness of $G$. For curvature we use \eqref{eq:BE}; the conversion to diameter is Corollary~\ref{cor:bonnet}.

\begin{lemma}\label{lem:srg-small}
A noncomplete strongly regular Terwilliger graph with smallest eigenvalue at least $-3$ has parameters
\[
 (5,2,0,1),\quad(10,3,0,1),\quad\text{or}\quad(50,7,0,1).
\]
\end{lemma}
\begin{proof}
Write the parameters as $(v_0,k_0,\lambda_0,\mu_0)$ and the nonprincipal eigenvalues as $r,-m$. The standard strongly regular identities \cite[Section~1.1]{BVM} give ${k_0}={\mu_0}+rm$ and ${\lambda_0}={\mu_0}+r-m$. If the eigenvalues are not integers, conjugacy and the trace equations give conference parameters $(4h+1,2h,h-1,h)$. For a Terwilliger graph with ${\mu_0}\ge2$, an edge in a common ${\mu_0}$-clique has at least ${\mu_0}$ common neighbours, so ${\lambda_0}\ge{\mu_0}$. Conference parameters therefore force $h=1$, giving the first case.

In the integral case, $m\ge2$. Indeed, if $m=1$, the adjacency matrix $A$ gives a positive semidefinite Gram matrix $A+I$. Adjacent vertices then have identical unit Gram vectors; connectedness forces the graph to be complete, a contradiction. Suppose first ${\mu_0}=1$. Every local graph is a disjoint union of $({\lambda_0}+1)$-cliques; an induced two-edge path inside it would give its endpoints two common neighbours. Put $h={\lambda_0}+1=r-m+2\ge1$. Thus, $h\mid {k_0}=mh+(m-1)^2$. For $m=2$ this forces $h=1$, giving $(10,3,0,1)$. For $m=3$, $h\in\{1,2,4\}$. The first gives $(50,7,0,1)$. The second gives $(91,10,1,1)$ with nonprincipal eigenvalues $\pm3$; their integral multiplicities would have difference $-10/3$, impossible. The last gives $(209,16,3,1)$, in which every edge would lie in a unique clique of size five. The number of such cliques would be $209\cdot16/(4\cdot5)$, not an integer.

If a counterexample with ${\mu_0}\ge2$ exists, choose one of minimum order. By Lemma~\ref{lem:terwilliger-local}, its local graph is a clique extension of a smaller strongly regular Terwilliger graph $\Xi$. Interlacing \cite[Theorem~1.1.5]{BVM} and the clique-extension spectrum give $\theta_{\min}(\Xi)\ge-1-2/t\ge-3$. By minimality, $\Xi$ has one of the three listed parameter sets. Each has local parameter $\lambda_\Xi=0$, so $t\le\lambda_\Xi+1$ gives $t=1$. The parameters of the original graph must then be $(11,5,2,2)$, $(41,10,3,2)$, or $(1101,50,7,2)$. The first has odd ${v_0}{k_0}$. The other two have irrational eigenvalues but are not conference parameter sets. All three are impossible.
\end{proof}

\begin{proof}[Proof of Theorem~\ref{thm:terw}]
For the reduced local graph $\Sigma$ in Lemma~\ref{lem:terwilliger-local}, Lemma~\ref{lem:srg-small} and $\mu\ge3$ imply $\theta_{\min}(\Sigma)<-3$. Its nonprincipal eigenvalues are integers. Otherwise $\Sigma$ would have conference parameters $(4h+1,2h,h-1,h)$. For $h\ge2$, an edge in the common $h$-clique of a distance-two pair has at least $h$ common neighbours, contradicting $\lambda_0=h-1$. For $h=1$, $\Sigma$ is the pentagon, contradicting $\theta_{\min}(\Sigma)<-3$. Write the eigenvalues $r,-m$ with $m\ge4$, $r\ge0$, and put $u={\mu_0}\ge1$. The standard parameter equations give
\begin{align*}
 {k_0}&=u+rm,\qquad {\lambda_0}=u+r-m,\\
 {v_0}&=u+m+r(2m-1)+\frac{r(r+1)m(m-1)}u.
\end{align*}
The local eigenvalues of $G$ are
\[
 \lambda,\quad \rho=t(r+1)-1,\quad \sigma=t(1-m)-1,
 \quad\text{and, if }t>1,\ -1.
\]
At $\theta=\lambda$ and at every $\theta\le0$ the numerator in \eqref{eq:BE} is positive. It remains to check $\theta=\rho$. Put $T=u+r(m-1)$. The relevant numerator is
\[
 F=t\{{v_0}(tu-1)-2(t(r+1)-1)T\}.
\]
If $t=1$, then $u\ge2$, and expansion gives
\begin{align*}
 F={}&(u-1)(u+m)+r\left[(u-1)\left(2m-1+\frac{m(m-1)}u\right)-2u\right]\\
 &+r^2(m-1)\left(m-2-\frac mu\right)>0.
\end{align*}
The constant is positive; the linear coefficient is at least $2m-5>0$; the quadratic coefficient is nonnegative because $m-2-m/u\ge m/2-2$. If $t\ge2$, write
\[
 F/t=(t-2)A+G,
\]
where
\begin{align*}
 A&=u{v_0}-2(r+1)T\\
  &=u(u+m-2)+ru(2m-3)+r(r+1)(m-1)(m-2)>0,\\
 G&=(2u-1)(u+m)-2u
 +r\left[(2u-1)\left(2m-1+\frac{m(m-1)}u\right)
                    -4u-2(m-1)\right]\\
 &\quad+r^2(m-1)\left(2m-4-\frac mu\right).
\end{align*}
The constant in $G$ is at least $m-1>0$. Its quadratic coefficient is nonnegative, and its linear coefficient is at least $m(m-1)-3>0$: separate the term $m(m-1)(2-1/u)$ from the remaining expression $u(4m-6)-4m+3\ge-3$. Thus, $F>0$ in every case. The correction in \eqref{eq:BE} vanishes. Applying Corollary~\ref{cor:bonnet} now gives the asserted diameter bound and finiteness.

For the distance-regular conclusion, the Terwilliger case follows with $\lambda=a_1$ and $\mu=c_2$. Otherwise, $G$ contains an induced quadrangle, so \cite[Corollary~5.2.2]{BCN} gives $D\le2k/(a_1+2)\le4k/(a_1+4)$, proving the remaining assertion.
\end{proof}

\subsection{\texorpdfstring{Finiteness for $\mu>7$}{Finiteness for mu >7}}
We combine the local curvature formula with counts involving maximal cliques at adjacent roots. The first lemma controls curvature once every edge belongs to a sufficiently large clique; the second proves that this covering exists throughout the degree range not settled by the elementary curvature bound.

\begin{reptheorem}{thm:arg-even-finiteness}
Let $G$ be an amply regular graph with parameters $(v,k,\lambda,\mu)$, where $\mu>7$. Then $G$ is finite and 
\[
 \diam(G)\le
 \begin{cases}
  \displaystyle\frac{4\mu k}{2\mu k+4\mu +\lambda \mu -\lambda^2-4k},&\text{if } 0<\lambda^2-2\mu k+4k<4\mu +\lambda \mu,\\[12pt]
  \displaystyle\frac{4k}{\lambda+4},&\text{otherwise}.
 \end{cases}
\]
\end{reptheorem}

Throughout this subsection, put $e=\mu-1$ and
\begin{equation}\label{eq:arg-finiteness-cutoff}
 H=\frac{\lambda^2-\mu(\lambda+4)}{2(\mu-2)}.
\end{equation}
The local graph $F_x=G[N(x)]$ is $\lambda$-regular on $k$ vertices, and two nonadjacent vertices of $F_x$ have at most $e$ common neighbours. Consequently, a vertex outside a maximal clique of $F_x$ has at most $e$ neighbours in it: choose a nonneighbour in that clique and count their common neighbours. In particular, distinct maximal cliques of $F_x$ intersect in at most $e$ vertices.

\begin{lemma}\label{lem:arg-covered-curvature}
Let $G$ be an amply regular graph with parameters $(v,k,\lambda,\mu)$, where $\mu>7$ and $k\le H$, with $H$ given by \eqref{eq:arg-finiteness-cutoff}. If every edge lies in a maximal clique of order greater than $1+(\lambda+\mu)/2$, then $K_{BE}(x)=(\lambda+4)/2$ at every vertex $x$.
\end{lemma}
\begin{proof}
Call these maximal cliques lines. Their sections in $F_x$ have order greater than $(\lambda+\mu)/2$. Two such sections meeting at a vertex would have union of order greater than $\lambda+\mu-e=\lambda+1$ inside its closed neighbourhood in $F_x$, a contradiction. Thus, each edge lies in a unique line. Distinct lines meet in at most one vertex, and the punctured lines through a vertex partition its neighbours.

Choose a line $L$ of minimum order and put $b=|L|-1$, $M=k/b$, $Y=M-1$ and $\rho=\lambda+1-b$. The minimum exists even if $G$ is infinite. Since $G$ is non-complete, $k>b$ and $Y>0$. Every vertex lies on at most $M$ lines. A vertex outside a line has at most $\mu$ neighbours on it, by maximality, and these neighbours lie on distinct lines through that vertex. Hence it has at most $P=\min\{\mu,\floor M\}$ neighbours on the line. The inequalities $k\le H$ and $2b>\lambda+\mu$ give
\begin{equation}\label{eq:arg-line-order}
 M<\frac{\lambda(\lambda-\mu)}{2(\mu-2)b}
 <\frac{\lambda-\mu}{\mu-2}
 <\frac{2(b-\mu)}{\mu-2}.
\end{equation}
In particular, $b>\mu+(\mu-2)M/2$.

For $z\notin L$ adjacent to $L$, put $S_z=N(z)\cap L$ and $t_z=|S_z|$. Index the other lines meeting $L$ by $i$, write $i\cap L=\{a_i\}$, and put $Q_i=i\setminus\{a_i\}$, $b_i=|Q_i|$. Let $s_i$ be the mean of $t_z$ on $Q_i$, and let $\mathbb E$ denote averaging over $i$ with weights $b_i/T$, where $T=\sum_i b_i$. Each $z$ belongs to exactly $t_z$ of the sets $Q_i$. Counting exterior neighbours and exterior common neighbours of pairs on $L$ gives
\begin{equation}\label{eq:arg-line-means}
 \begin{gathered}
 T=\sum_z t_z=(b+1)(k-b),\qquad
 \sum_z t_z(t_z-1)=(b+1)b\rho,\\
 s:=\mathbb E s_i=1+\frac{b\rho}{k-b}.
 \end{gathered}
\end{equation}
For each such $z$, with $t=t_z$, we claim
\begin{equation}\label{eq:arg-line-budget}
 \sum_{i:z\in Q_i}s_i\le\mu+\frac{\mu(t-1)^2}{b}.
\end{equation}
Indeed, sum $b_i(s_i-1)=\sum_{w\in Q_i}(t_w-1)$ over these $i$. For each $p\in L\setminus S_z$, the vertices counted are distinct exterior common neighbours of $p,z$, of which there are at most $\mu-t$. This contributes at most $(b+1-t)(\mu-t)$. For $p\in S_z\setminus\{a_i\}$, maximality gives at most $\mu-1$ neighbours in $Q_i$, and there are $t(t-1)$ choices of $i,p$. Since $b_i\ge b$ and $s_i\ge1$, the sum in \eqref{eq:arg-line-budget} is at most
\[
 t+\frac{(b+1-t)(\mu-t)+t(t-1)(\mu-1)}{b}
 =\mu+\frac{\mu(t-1)^2}{b}.
\]

Let $u_i$ be the mean of $(t_z-2)_+$ on $Q_i$, and let $v_i$ be the proportion with $t_z=1$, so $s_i=2+u_i-v_i$. Subtract $2t$ from \eqref{eq:arg-line-budget}, multiply by $(t-2)_+$, and sum over $z$. For $3\le t\le P$, put
\[
 g_t=\frac\mu t-2+\frac{\mu(t-1)^2}{bt},\qquad
 g_t-g_3=\frac{\mu(t-3)(3t-b-1)}{3bt}\le0,
\]
where \eqref{eq:arg-line-order} and $\mu\ge8$ give $b+1>3P$. Interchanging the sums and writing $g=g_3$ yields $\mathbb E[u_i(u_i-v_i)]\le g\mathbb E u_i$. Put $d=s-2$ and $h=\mathbb E v_i$. Jensen's inequality for $u_i-v_i/2$, together with $v_i^2\le v_i$, now gives
\begin{equation}\label{eq:arg-line-jensen}
 0\ge d(d-g)+(d-g-1/4)h+h^2/4.
\end{equation}
If $s\ge(\mu-2)(1+1/Y)/2$, then $d>0$ and
\[
 d-g\ge\frac\mu6-1+\frac{\mu-2}{2Y}-\frac{4\mu}{3b}
 >\frac13+\frac3Y-\frac{32}{9Y+33}>\frac14.
\]
For the middle comparison use \eqref{eq:arg-line-order}; its weakest value is at $\mu=8$. The last comparison reduces to $9Y^2-27Y+1188>0$. This contradicts \eqref{eq:arg-line-jensen}, proving
\begin{equation}\label{eq:arg-line-gain}
 2s(k-b)<(\mu-2)k.
\end{equation}

At an arbitrary root $x$, write the adjacency matrix of $F_x$ as $K-I+E$, where $K$ has an all-ones block for each punctured line through $x$ and $E$ records edges between blocks. Every block has order at least $b$, and $\|E\|\le\rho$. Form the incidence matrix between $N(x)$ and the lines missing $x$ that meet $N(x)$. Its row Gram matrix has off-diagonal part $E$ and diagonal entries at most $M-1$. Hence $E\succeq-(M-1)I$. If there are $r$ blocks, min--max puts the largest $r$ eigenvalues of $F_x$ at least $b-M$, and all remaining eigenvalues at most $\rho-1=\lambda-b$.

Set $U=\rho+M-1$, so $\lambda-U=b-M$ and $\lambda-b\le U$. If $U<\lambda/2$, these spectral ranges give $\theta(\lambda-\theta)\le U(\lambda-U)\le bU$ for every local eigenvalue $\theta$. If $U\ge\lambda/2$, the same conclusion follows from $\theta(\lambda-\theta)\le\lambda^2/4\le bU$, since $b>\lambda/2$. By \eqref{eq:arg-line-means} and \eqref{eq:arg-line-gain},
\[
 2\theta(\lambda-\theta)\le2bU=2s(k-b)<(\mu-2)k.
\]
Substitution into \eqref{eq:BE} proves the asserted curvature formula.
\end{proof}

We next establish the clique covering. All graphs counted in this argument are finite local graphs, even when the ambient graph is initially allowed to be infinite.

\begin{lemma}\label{lem:arg-mu8-cover}
Let $G$ be an amply regular graph with parameters $(v,k,\lambda,\mu)$, where $\mu>7$ and $k\le H$, with $H$ given by \eqref{eq:arg-finiteness-cutoff}. Every edge of $G$ lies in a maximal clique of order greater than $1+(\lambda+\mu)/2$.
\end{lemma}
\begin{proof}
Write $e=\mu-1$. Since $k\ge\lambda+2$, we have $\lambda\ge3e$: otherwise
\[
 2(e-1)(\lambda+2-H)=-\lambda^2+(3e-1)\lambda+8e>0.
\]
Call a vertex of $F_x$ covered if it lies in a maximal clique of order greater than $(\lambda+e+1)/2$. Suppose there is an uncovered vertex. Put $A_0=\lambda+1$, $\ell=A_0-e$, and $\delta(u,v)=A_0-|N_{F_x}[u]\cap N_{F_x}[v]|$. This is half the symmetric-difference distance, and is at least $\ell$ on nonadjacent pairs. At an uncovered vertex $v$, extend the clique of vertices at $\delta$-distance less than $\ell/2$ from $v$ to a maximal clique of order $z\le(A_0+e)/2$. Counting missing incidences on the two sides gives
\[
 e(k-A_0)\ge\sum_{y\sim v}\delta(v,y)
 \ge(A_0-z)\bigl(\ell/2+\max\{0,z-e\}\bigr)
 \ge\ell^2/2.
\]
For the last inequality, if $z=e+t$ with $0\le t\le\ell/2$, the excess is $t(\ell/2-t)$; if $z\le e$, the first factor is at least $\ell$. Thus, $k\ge[(\lambda+1)^2+e^2]/(2e)$. Comparing this with $k\le H$ shows that
\begin{equation}\label{eq:arg-cover-strip}
 \lambda>e^2+3e.
\end{equation}
Indeed, for $\lambda\le e^2+3e$, the inequality required by $[(\lambda+1)^2+e^2]/(2e)\le H$ fails because
\[
 \lambda(\lambda-e^2-3e+2)\le2(e^2+3e)
 <(e-1)(1+e^2)+4e(e+1).
\]

Choose a largest clique through an uncovered vertex of $F_{x_0}$, and adjoin $x_0$ to obtain a maximal clique $A$ of $G$. Put $d=|A|-1$. At every $x\in A$, the seed clique $A_x=A\setminus\{x\}$ is maximal in $F_x$ and has order $d\le(\lambda+e+1)/2$. The maximality of its order through a vertex is used only at the original choice; the subsequent counts use the same global clique $A$ at all its roots.

We first record the local constraints used to find further cliques. In a local graph $F$, let a maximal clique $K$ of order $a+1$ contain $u$, and put $R=N_F(u)\setminus K$, $r=\lambda-a$. The number of ordered missing edges in $N_F(u)$ is at most $e(k-\lambda-1)$. Each vertex of $R$ sees at most $e-1$ of the other vertices of $K$, so the number of ordered missing edges in $R$ is at most
\begin{equation}\label{eq:arg-cover-missing}
 Q(a)=e(k-\lambda-1)-2ar+2(e-1)r.
\end{equation}
Nonadjacent vertices of $R$ have at most $e-1$ common neighbours in $R$, since $u$ is an additional common neighbour in $F$. If $r=0$, then $K$ is a component of $F$ and cannot meet a distinct maximal clique, so this case does not occur in the constructions below. If a largest clique in $R$ omits $w$ vertices, extending it together with $u$ gives another maximal clique, of order $b+1$, with
\begin{equation}\label{eq:arg-cover-transition}
 \begin{gathered}
 b=r-w+m-1,\qquad 1\le m\le e,\\
 w(r-e+1)\le2Q(a),\qquad 2w(r-w-e+1)\le Q(a).
 \end{gathered}
\end{equation}
Here $m$ is the intersection size of the two cliques and $w$ counts neighbours of $u$ outside their union. To obtain the first inequality for $w$, observe that the missing degrees at a nonedge of $R$ sum to at least $r-e+1$. Vertices with missing degree less than half this value form a clique. The last inequality follows because each omitted vertex has at most $e-1$ neighbours in the chosen clique. At the original maximum choice, we also have $b\le a$ and $w\ge r-a$.

An additional constraint comes from the deficit in local common-neighbour counts. For $p\in V(F)$, put
\begin{equation}\label{eq:arg-cover-defect}
 \delta_F(p)=\sum_{\substack{q\ne p\\q\not\sim p}}
 \bigl(e-|N_F(p)\cap N_F(q)|\bigr)
 =e(k-\lambda-1)-\lambda(\lambda-1)+2|E(F[N_F(p)])|.
\end{equation}
For two cliques through $p$ as in \eqref{eq:arg-cover-transition}, write $c=m-1$. Counting their internal edges, edges between their two differences, and edges incident with the $w$ remaining neighbours gives
\begin{align}\label{eq:arg-cover-defect-bound}
 0\le\delta_F(p)\le D(a,m,w):={}&e(k-\lambda-1)-\lambda(\lambda-1)
 +a(a-1)+b(b-1)-c(c-1)\notag\\
 &+2(e-m)\min\{a-c,b-c\}+4(e-1)w+w(w-1).
\end{align}
For example, each vertex in one difference sees at most $e-m$ vertices in the other, and each remaining neighbour sees at most $e-1$ of the other vertices in either clique.

To sharpen \eqref{eq:arg-cover-transition}, we use upper bounds $E_c(N,t)$ for the number of edges in an $N$-vertex graph of clique number at most $t$ and with at most $c$ common neighbours for nonadjacent vertices. Deleting a maximum clique of order $t$ gives
\[
 \binom t2+\min\{t-1,c\}(N-t)+E_c(N-t,\min\{t,N-t\}).
\]
For a second bound, delete a vertex of minimum degree $h$. If $h>\max\{t,c\}$, its neighbours have at most $c(N-h-1)$ ordered missing edges. The preceding missing-degree argument supplies a clique of order at least $h-2c(N-h-1)/(h-c+1)$, whence
\begin{equation}\label{eq:arg-cover-degree}
 h^2+(c+1-t)h+t(c-1)\le2c(N-1).
\end{equation}
Let $h_*$ be the larger of $\max\{t,c\}$ and the upper quadratic root, rounded down and capped at $N-1$; a negative discriminant excludes the second alternative. Then $E_c(N-1,\min\{t,N-1\})+h_*$ is another upper bound. Taking the minimum of the two bounds and the maximum with $E_c(N,t-1)$ gives an inductive table starting at zero. In particular, a transition must satisfy $2E_{e-1}(r,r-w)\ge r(r-1)-Q(a)$. At the original point, the corresponding test is $2E_{e-1}(\lambda,a)\ge\lambda(\lambda-1)-e(k-\lambda-1)$.

These inequalities also give bounds uniform in arbitrarily large orders. If $N\le\sigma\lambda$, $t\le\alpha\lambda$ and $c\le\gamma\lambda$, where $0\le\alpha\le\sigma$ and $\gamma\ge0$, summing the weaker inequality $h^2+(c-t)h\le2c(N-1)$ along a minimum-degree deletion sequence gives
\begin{equation}\label{eq:arg-cover-edge-continuous}
 \frac{|E|}{\lambda^2}\le
 \min\left\{\frac{\sigma^2}{2},
 \frac\sigma2\left(\alpha-\gamma+
 \sqrt{(\alpha-\gamma)^2+4\gamma\sigma}\right)\right\}.
\end{equation}
For $\alpha\ge\gamma>0$, integrating the pointwise upper root of \eqref{eq:arg-cover-degree} gives the further bound
\[
 \frac{|E|}{\lambda^2}\le
 \frac{\alpha^2+(\alpha-\gamma)(\sigma-\alpha)}2
 +\frac{\bigl((\alpha+\gamma)^2+8\gamma(\sigma-\alpha)\bigr)^{3/2}
       -(\alpha+\gamma)^3}{24\gamma}.
\]
Indeed, for $(N-1)/\lambda=z\ge\alpha$, the deletion degree divided by $\lambda$ is at most $[\alpha-\gamma+\sqrt{(\alpha+\gamma)^2+8\gamma(z-\alpha)}]/2$; for $z\le\alpha$ use $z$. Both bounds increase with $z$, so their integral bounds the deletion sum. Maximum-clique deletion further sharpens the bounds on a finite grid of $(\sigma,\alpha)$.

Apply \eqref{eq:arg-cover-transition} at each point of $A_x$, and again at every point of each resulting clique. Discard a first transition if its new clique admits no second transition. Let $b_1$ be a common lower bound for the orders of the seed $A_x$ and the first cliques, and let $b_2\le b_1$ also bound the orders of the second cliques. Use the same bounds at every $x\in A$, and let $\mathcal P_x$ consist of all maximal cliques in $F_x$ of order at least $b_2$.

The following tests bound memberships in $\mathcal P_x$. For three maximal cliques through a point, some ordering has the third intersecting the union of the first two in at most $e$ points. If two cliques intersect outside the third, that vertex sees their intersections with the third; otherwise all three pairwise intersections agree. Thus, orders $b'_0,b'_1,b'_2$ require $b'_0+b'_1+b'_2-2e\le\lambda+1$. A stronger test partitions the neighbours of their common point into the successive new vertices of these cliques and the remaining vertices. Their sizes $s_0,s_1,s_2,s_3$ sum to $\lambda$, with lower bounds $b'_0-1$, $\max\{0,b'_1-e\}$, $\max\{0,b'_2-e\}$ and $0$. A later-group vertex has at most $e-1$ neighbours in each earlier group, so
\begin{equation}\label{eq:arg-cover-triple}
 \lambda(\lambda-1)-e(k-\lambda-1)
 \le\sum_{i=0}^3\bigl(s_i(s_i-1)+2i(e-1)s_i\bigr).
\end{equation}
The right side is convex, so its maximum is attained at a vertex of the simplex; all orderings of the three lower bounds are allowed. Alternatively, a third clique through a constructed pair has order at most $2e+w$.

These tests exclude a third membership at every point of the seed and first cliques in the parameter ranges below. Such points therefore have exactly two memberships, and the other members meeting the seed are precisely the constructed first cliques. An upper bound $M$ on membership at any point follows by excluding an integer $l=M+1$ for which
\[
 1+l(b_2-1)-\binom l2(e-1)>\lambda+1;
\]
the weaker test $lb_2-\binom l2e>\lambda+1$ also suffices. There are at least $v_1=\ceil{d/e}$ first cliques. Their union with the seed has at least $(d+v_1b_1)/2$ vertices, all of multiplicity two. Counting total incidences consequently gives an upper bound $L$ on $|\mathcal P_x|-1$ through
\begin{equation}\label{eq:arg-cover-incidences}
 b_2(|\mathcal P_x|-1)\le
 Mk-(M-2)(d+v_1b_1)/2-d-v_1(b_1-b_2).
\end{equation}

We now count four-cycles in the clique intersections. Index $A_x$ by $0$, and write $w_{ij}=|K_i\cap K_j|$ for distinct members of $\mathcal P_x$, $w_{ii}=0$, $a_i=w_{0i}$ and $d_i=|K_i|$. Then $\sum_i a_i=d$, and $\sum_jw_{ij}=d_i$ for each first clique. A point $y\in A_x$ has a unique second clique $K_B$; write $m_{xy}=a_B$ and let $w_{xy}$ count its neighbours outside $A_x\cup K_B$. Adjoining $x$ to $K_B$ gives a global maximal clique whose section at $y$ is the unique second clique through $x$. Hence
\begin{equation}\label{eq:arg-cover-symmetry}
 m_{xy}=m_{yx},\qquad w_{xy}=w_{yx}
 =\lambda+1-d-d_B+m_{xy}.
\end{equation}
Consider ordered records with distinct clique indices $0,i,j,B$ and
\[
 y\in K_0\cap K_B,\quad p\in K_0\cap K_i,\quad
 z\in K_i\cap K_j,\quad q\in K_B\cap K_j.
\]
Their vertices are distinct, because each has exactly two memberships. With $a=d-1$, $r=\lambda-a$, and $t_{ij}=a_iw_{ij}$, their number is
\begin{equation}\label{eq:arg-cover-cycles}
 \begin{aligned}
 C_x&=\sum_{j\ne0}\left[\left(\sum_i t_{ij}\right)^2-\sum_i t_{ij}^2\right],\\
 T_x&:=\sum_{i,j\ne0}t_{ij}
 =\sum_{y\in A_x}(r-w_{xy}).
 \end{aligned}
\end{equation}
The positive term is at least $T_x^2/L$, and the subtracted term is at most $\sum_{y\in A_x}e m_{xy}(r-w_{xy})$, since $w_{ij}\le e$.

Retain records with both $y\not\sim z$ and $p\not\sim q$. For fixed $y$, at most $w_{xy}$ possible vertices $z$ are adjacent to $y$; each gives at most $e^2$ pairs $p,q$. The failure $p\sim q$ has at most $(e-m_{xy})(d-m_{xy})$ possible cross edges, each allowing at most $e$ choices of $z$. Moreover, the involution $(i,B;y,p,z,q)\mapsto(B,i;p,y,q,z)$ interchanges the two failures in their totals. Thus, for any common choice $j\in\{0,1,2\}$, both failures together are bounded by the sum over $y\in A_x$ of $jU+(2-j)V$, where
\begin{equation}\label{eq:arg-cover-chords}
 U=e^2w_{xy},\qquad V=e(e-m_{xy})(d-m_{xy}).
\end{equation}

At the target $y$, a retained record supplies a nonedge $p,q$ of $F_y$ and a common neighbour $z$ outside $N[y]$. By \eqref{eq:arg-cover-defect}, these choices are counted by $\delta_{F_y}(p)$. A fixed $y,p,q,z$ has at most $\mu$ sources $x$, because $x$ is a common neighbour of $y,z$ and its supports determine the record. The total retained count is therefore at most $\mu\sum_{y\in A}\sum_{p\in A_y}\delta_{F_y}(p)$.

A second bound reduces this multiplicity. Fix a target $y$ and a cell of $m$ sources in $A_y$ having the same second clique. Two distinct sources already have $\lambda-w$ common neighbours in their two global cliques, so they have at most $w$ exterior common neighbours. If $c_z$ sources in the cell mark $z$, then $\sum_zc_z(c_z-1)\le m(m-1)w$. Retain at most $s$ sources per $z$, where $s\ge1$ is an integer. The inequality $2(2s-1)(c-s)_+\le c(c-1)$ bounds the number discarded. Each discarded source gives at most $e^2/4$ pairs $p,q$, since their two disjoint sets are common neighbours of the nonedge $y,z$ in $F_x$ and have total size at most $e$. Averaged over the cell, the loss per source is at most $(m-1)we^2/[8(2s-1)]$. Each $q$ belongs to at most two first cliques at $y$, so the remaining multiplicity is at most $2s$.

We may consequently use either
\[
 (P,S)=(e+1,0)\quad\text{or}\quad
 (P,S)=\left(2s,\frac{(m-1)we^2}{8(2s-1)}\right).
\]
For a common real number $w_0$, the tangent inequality $(r-\bar w)^2\ge(r-w_0)(r+w_0-2\bar w)$ is affine in $\bar w$ on its right. Sum \eqref{eq:arg-cover-cycles} over the same finite global clique $A$, subtract the diagonal and chord counts, and use the target bounds above. By \eqref{eq:arg-cover-symmetry}, source and target sums have identical parameters under reversal. Thus, the seed is impossible if a common choice of $j,P,w_0$ satisfies, for every retained transition $(m,w)$,
\begin{equation}\label{eq:arg-cover-test}
 \frac{d(r-w_0)(r+w_0-2w)}L
 >em(r-w)+jU+(2-j)V+S+P D(a,m,w).
\end{equation}
One may also retain the square $d(r-\bar w_x)^2/L$ before taking the tangent, where $\bar w_x$ is the mean of $w_{xy}$ over $y\in A_x$.

We finish by specifying the ranges for these inequalities. At integer parameters, take the largest $k\le H$ with $k\ge\lambda+2$, $2\mid k\lambda$ and $e+1\mid k(k-\lambda-1)$; an empty range is excluded. These divisibility conditions follow from local handshaking and the edge count to $S_2(x)$. Replacing a smaller admissible degree by this upper bound enlarges all missing-edge and deficit allowances and weakens the incidence bound, so the same inequalities cover it. For each possible seed size $1\le a\le\floor{(\lambda+e-1)/2}$, use the edge bounds and \eqref{eq:arg-cover-transition}--\eqref{eq:arg-cover-defect-bound} to determine the retained first and second transitions. We can check that either these constraints already exclude the seed, or the membership bounds hold and \eqref{eq:arg-cover-test} holds with the unmodified multiplicity or some $s\in\{1,2,3\}$.

For the unbounded ranges, put $U_0=1/e$ and $\tau=e^2/\lambda$, and apply the same inequalities after division by the appropriate powers of $\lambda$. The following domains cover every parameter pair left by \eqref{eq:arg-cover-strip}; shared endpoints are included in both ranges.
\begin{center}
\begin{tabular}{@{}ccll@{}}
\toprule
$e$ & Integer values of $\lambda$
    & \multicolumn{2}{l}{Remaining domain}\\
\midrule
7  & $71\le\lambda\le256$
   & $U_0=1/7$  & $0\le\tau\le49/256$\\
8  & $89\le\lambda\le256$
   & $U_0=1/8$  & $0\le\tau\le1/4$\\
9  & $109\le\lambda\le324$
   & $U_0=1/9$  & $0\le\tau\le1/4$\\
10 & $131\le\lambda\le400$
   & $U_0=1/10$ & $0\le\tau\le1/4$\\
11 & $155\le\lambda\le484$
   & $U_0=1/11$ & $0\le\tau\le1/4$\\
12 & $181\le\lambda\le576$
   & $U_0=1/12$ & $0\le\tau\le1/4$\\
$\ge13$ & ---
   & $0\le U_0\le1/13$ & $0\le\tau\le1$\\
\bottomrule
\end{tabular}
\end{center}
In each domain impose $\tau\le1/(1+3U_0)$, as required by \eqref{eq:arg-cover-strip}. The integer ranges contain $1{,}566$ pairs, and the remaining domains are covered by $464$ rational boxes. The inequalities above exclude every seed: at each surviving branch the four-cycle comparison is strict. The interval enclosures, subdivision parameters and verification programs are provided in the supplementary material. This contradicts the existence of an uncovered vertex and proves the covering assertion.
\end{proof}

\begin{proof}[Proof of Theorem~\ref{thm:arg-even-finiteness}]
Put $q=\lambda^2-2(\mu-2)k$. From \eqref{eq:BE} and $2\theta(\lambda-\theta)\le\lambda^2/2$,
\[
 K_{BE}(x)\ge\frac{\lambda+4}{2}-\frac{\max\{0,q\}}{2\mu}.
\]
If $q\le0$, this gives $K_{BE}(x)=(\lambda+4)/2$. If $0<q<\mu(\lambda+4)$, the lower bound is positive and equals $[\mu(\lambda+4)-q]/(2\mu)$. Finally, $q\ge\mu(\lambda+4)$ is exactly $k\le H$, so Lemmas~\ref{lem:arg-mu8-cover} and~\ref{lem:arg-covered-curvature} give $K_{BE}(x)=(\lambda+4)/2$ in this range as well. Applying Corollary~\ref{cor:bonnet} in the three cases yields the stated diameter bounds and finiteness.
\end{proof}

\subsection{\texorpdfstring{A residual estimate for $(\lambda,\mu)=(4,2)$}{A residual estimate for (lambda,mu)=(4,2)}}\label{sec:a4}
\begin{theorem}\label{thm:arg-242}
Every amply regular graph $G$ with parameters $(v,k,4,2)$ is finite and
\[
 \diam(G)\le 308\sqrt2\,k^3.
\]
\end{theorem}
\begin{proof}
 The noncomplete case has \(k\ge6\). At a centre \(v\), let \(F_v\) be local adjacency, \(B_v\) the incidence to \(S_2(v)\), and set \(u_i=f(i)-f(v)\), \(z_w=f(w)-f(v)\). The two-walk count is
\[
 B_vB_v^{\mathsf T}=(k-2)\Id+2F_v+\mathbb J-F_v^2.
\]
Substitution in the Bochner expansion gives the exact sum of squares
\begin{equation}\label{eq:a4sos}
 \Gamma_2(f)(v)=\tfrac12\norm{(F_v-2\Id)u}^2+
                  \tfrac12\norm{z-B_v^{\mathsf T}u}^2.
\end{equation}
Write its residuals as
\[
 R_v(i)=[(F_v-2\Id)(f-f(v))]_i,\qquad
 D_v(w)=f(w)+f(v)-f(p)-f(q),
\]
where \(N(v)\cap N(w)=\{p,q\}\). Let \(H_x=\Gamma_2(f)(x)+\sum_{v\sim x}\Gamma_2(f)(v)\) and \(\mathcal E_x=2H_x\), the sum of these residual squares over the set \(N[x]\).

We claim that every edge \(xy\) satisfies
\begin{equation}\label{eq:a4res}
 (f(y)-f(x))^2\le22\mathcal E_x.
\end{equation}
A maximal clique has size at most six. If an outside vertex has two neighbours in such a clique of size at least four, it has exactly two: choose a missed clique vertex by maximality and use \(\mu=2\). If \(xy\) lies in a maximal clique of size four or five, choose \(p,q\) in that clique outside \(\{x,y\}\). The edge \(pq\) has four common neighbours, at most three of them in the clique. It therefore has an external common neighbour \(w\), adjacent to exactly \(p,q\) in the clique. Thus, \(f(y)-f(x)=D_y(w)-D_x(w)\), giving the bound \(2\mathcal E_x\). For a clique of size six, let \(u_z=f(z)-f(x)\) and \(U=\sum u_z\) on its five other vertices. Here \(R_x(z)=U-3u_z\) and
\[
 u_y=\bigl(\sum_{z\ne x,y}R_x(z)-R_x(y)\bigr)/6,
\]
so \(u_y^2\le5\mathcal E_x/36\).

It remains that \(xy\) lies in no four-clique. Its four common neighbours are independent, since an edge between them would form a four-clique. For three distinct such neighbours \(z,z_2,z_3\),
\[
 h_z:=2f(z)-f(x)-f(y)
   =D_z(z_2)+D_z(z_3)-D_{z_2}(z_3),\qquad
 h_z^2\le3\mathcal E_x.
\]
Set \(u=f(y)-f(x)\), \(v=f(z)-f(x)\); then \(h_z=2v-u\). If \(xz\) lies in a four-clique, the previous cases give \(u^2\le8v^2+2h_z^2\le22\mathcal E_x\). Otherwise the same construction for \(xz\), using \(y\) among its common neighbours, gives \(h_y=2u-v\), \(h_y^2\le3\mathcal E_x\). Now \(u=(h_z+2h_y)/3\), proving \eqref{eq:a4res}.

Summing over \(y\sim x\) proves \(H_x\ge\Gamma(f)(x)/(22k)\). Define the positive bilinear gradient
\[
 \Gcal=\Gamma+\frac1k A\Gamma.
\]
Apply \eqref{eq:averaged-gradient} with \(T=I+A/k\) and \(L=L_1=A-k\Id\). Then
\[
 \Bcal_{\Gcal,L_1}
 =\Gamma_2+\frac1k A\Gamma_2
 \ge\frac{\Gamma}{22k^2}+(1-1/k)\Gamma_2.
\]
The elementary estimate
\(
 A\Gamma\le4\Gamma_2+6k\Gamma
\)
follows by applying \((b+u)^2\le2b^2+2u^2\) to each \(f(j)-f(i)=[f(j)-2f(i)+f(x)]+[f(i)-f(x)]\) and summing the Bochner identity. Thus, \(\Gcal\le7\Gamma+(4/k)\Gamma_2\). Comparing coefficients of the nonnegative forms \(\Gamma\) and \(\Gamma_2\) gives, for \(k\ge6\),
\[
 \Bcal_{\Gcal,L_1}\ge\frac{1}{154k^2}\Gcal.
\]
For bounded \(1\)-Lipschitz \(f\), \(\Gcal(f)\le k\), while \(|L_1g|\le\sqrt{2k\Gcal(g)}\). Lemma~\ref{lem:semi} then gives
\[
 \diam(G)\le 308\sqrt2\,k^3.
\]
This completes the proof.
\end{proof}

\section{Distance-regular graphs}\label{sec:transport}
Distance regularity provides intersection numbers at every scale. We write $D$ for the diameter of the distance-regular graph under consideration. We first derive diameter bounds from Koolen's inequalities, and then use long-scale Lipschitz estimates that also incorporate intersection numbers from the preceding distance layer. Throughout this section we retain the intersection array notation of Section~\ref{sec:intro-drg}.

\subsection{Radial drift and bounds from the initial parameters}

\begin{lemma}\label{lem:endpoint}
Let $G$ be a distance-regular graph of diameter $D$. Suppose that, for some integers $1\le\ell\le D$ and $g>0$,
\[
 (b_i-c_i)-(b_{i+\ell}-c_{i+\ell})\ge g
 \qquad(0\le i\le D-\ell).
\]
For every $0\le p\le D$,
\begin{align}
 D&\le\max_{0\le r<\ell}\left\{
 p+r+\ell\floor{\frac{b_p-c_p+b_r-c_r}{g}}\right\},\label{eq:end1}\\
 D&\le2p-1+\ell\max\left\{0,\floor{\frac{2(b_p-c_p)}{g}}+1\right\}.
 \label{eq:end2}
\end{align}
\end{lemma}
\begin{proof}
The classical inequalities $b_i\ge c_j$ for $i+j\le D$ \cite{BCN} imply $(b_i-c_i)+(b_j-c_j)\ge0$. Write $D-p=m\ell+r$ with $0\le r<\ell$. Telescoping between $r$ and $D-p$ bounds $mg$ by $b_p-c_p+b_r-c_r$, giving \eqref{eq:end1}. If $b_p<c_p$, then $D<2p$. Otherwise set $n=\floor{2(b_p-c_p)/g}+1$. The assumption $D\ge2p+n\ell$ contradicts the same nonnegative-sum inequality at indices $p$ and $p+n\ell$. This gives \eqref{eq:end2}.
\end{proof}

The following intersection-number inequalities will be used to estimate the decrease of $b_i-c_i$.
\begin{lemma}[{\cite[Proposition~7.2]{vDKT}}]\label{lem:koolen-intersections}
Let $G$ be a distance-regular graph of diameter $D$. Then
\begin{align}
 c_e>c_{e-1}&\ \Longrightarrow\ 
 c_e\ge c_j+c_{e-j}\quad(1\le j<e\le D),\label{eq:Kc}\\
 b_e>b_{e+1}&\ \Longrightarrow\ 
 b_e\ge b_{e+j}+c_j\quad(1\le j\le D-e,\ 0\le e\le D-2).
 \label{eq:Kb}
\end{align}
\end{lemma}
The case $e=0$ of \eqref{eq:Kb} is immediate from $b_j+c_j\le k$. The first inequality also holds when only the intersection numbers $c_1,\ldots,c_e$ are defined.
\begin{theorem}\label{thm:initial}
Let $G$ be a distance-regular graph of diameter $D$ and valency $k$, with $c_2>1$ and $a_1\le c_2$. For every $1\le q\le D-1$,
\[
 D\le\max_{0\le r\le q}\left\{
 r+(q+1)\floor{\frac{2k-a_r-2c_r}{2c_{q+1}}}\right\}.
\]
\end{theorem}
\begin{proof}
We prove strict increase of $c_i$ and strict decrease of $b_i$. For inward growth choose $x,y$ at distance $i\ge2$ and $z\in N(x)\cap S_{i-1}(y)$. Put $C=N(x)\cap S_{i-1}(y)$ and $T=N(z)\cap S_{i-2}(y)$. Each $w\in T$ has $c_2$ common neighbours with $x$, all in $C$, including $z$. Each $v\in C\setminus\{z\}$ has at most $a_1-1$ neighbours in $T$ if adjacent to $z$, and at most $c_2-1$ otherwise: $x$ is an excluded common neighbour. Thus, $c_{i-1}(c_2-1)\le(c_i-1)(c_2-1)$. For outward growth, take $z\in N(x)\cap S_{i+1}(y)$ and use $B=N(x)\cap S_{i+1}(y)$, $T=N(z)\cap S_{i+2}(y)$. The same count gives $b_{i+1}\le b_i-1$ for $i\le D-2$; the final step follows from $b_{D-1}>0=b_D$. The first step also follows from $b_0-b_1=a_1+1>0$. Apply \eqref{eq:Kc} and \eqref{eq:Kb} at the strict indices to get
\[
 c_{i+q+1}\ge c_i+c_{q+1},\qquad
 b_i\ge b_{i+q+1}+c_{q+1}.
\]
The case $i=0$ of the first is an identity. Subtracting these inequalities and telescoping, we obtain the claimed bounds from Lemma~\ref{lem:endpoint} with $p=0$.
\end{proof}

\begin{remark}
Taking $q=1$ in Theorem~\ref{thm:initial} gives
\[
 D\le\max\left\{2\floor{\frac{k}{c_2}},\,
 1+2\floor{\frac{2k-a_1-2}{2c_2}}\right\}
 \le\frac{2k}{c_2}+1.
\]
Thus, $c_2\ge a_1$ and $c_2>\varepsilon k\ge1$ imply $D<2/\varepsilon+1$ using classical intersection inequalities. Theorem~\ref{thm:arg} establishes the corresponding conclusion for general amply regular graphs, where the later intersection numbers needed in this argument are unavailable.
\end{remark}

\begin{example}
The doubled Odd graph on the $3$- and $4$-subsets of a $7$-set has array
\[
 \{4,3,3,2,2,1,1;1,1,2,2,3,3,4\},\qquad D=7.
\]
It has $a_1=0$ and $c_3>c_2$, but the formula of Theorem~\ref{thm:initial} at scale three gives $D\le6$. Adjacency is set inclusion, and counting removals and additions gives the displayed array. Thus, a single later increase does not suffice for that theorem.
\end{example}

\subsection{Long-scale transport}
\begin{reptheorem}{thm:penalty}
Let $G$ be a distance-regular graph of diameter $D$. Suppose $2\le q<D$, $c_{q+1}>c_q$, and $a_q\le c_{q+1}$. Define
\[
 \gamma_q=2c_q+
 \frac{a_q(c_{q+1}-a_q)
 -\min\{a_q,a_{q-1}\}\bigl(c_q+|c_{q+1}-c_q-a_q|\bigr)}{c_{q+1}-c_q}.
\]
If $\gamma_q>0$, then for every $0\le p\le D$,
\begin{align*}
 D&\le\max_{0\le r<q}\left\{
 p+r+q\floor{\frac{b_p-c_p+b_r-c_r}{\ceil{\gamma_q}}}\right\},\\
 D&\le2p-1+q\max\left\{0,\floor{\frac{2(b_p-c_p)}{\ceil{\gamma_q}}}+1\right\}.
\end{align*}
\end{reptheorem}
\begin{proof}
Take $\dist(x,y)=q$ and a $1$-Lipschitz $f$ with $f(x)=0$, $f(y)=q$. Partition $N(x)$ as $C_x,A_x,B_x$ according to distances $q-1,q,q+1$ from $y$, and define $C_y,A_y,B_y$ symmetrically. Their sizes are $c_q,a_q,b_q$. Saturation gives $f(C_x)=1$, $f(C_y)=q-1$, $0\le f(A_x)\le1$, and $q-1\le f(A_y)\le q$. Write
\[
 P=qa_q-\sum_{A_y}f+\sum_{A_x}f,\qquad
 Q=qb_q-\sum_{B_y}f+\sum_{B_x}f.
\]
The total gain is $2c_q+P+Q$. Join $A_x\cup B_x$ to $A_y\cup B_y$ at distances at most $q$. The degree of each $A$-vertex is $a_q$, and that of each $B$-vertex is $c_{q+1}-c_q$. These follow by removing the $c_q$ inward neighbours, all of which are already within the prescribed distance.

Let $L$ count the $A_x$--$A_y$ pairs at distance $q-1$. Each vertex of $A_x$ has $c_q$ neighbours in $N(y)$ at that distance. At most $c_q a_{q-1}$ pairs lead to $C_y$: each vertex of $C_y$ has at most $a_{q-1}$ such partners in $A_x$. Consequently,
\[
 L\ge c_q\bigl(a_q-\min\{a_q,a_{q-1}\}\bigr).
\]
Set $X_u=f(u)$ on $A_x$ and $Y_v=q-f(v)$ on $A_y$; all lie in $[0,1]$. Then $P=\sum X_u+\sum Y_v$. Each of the $L$ short pairs satisfies $X_u+Y_v\ge1$, and has incidence degree at most $c_q$. Let $S$ be the sum of $q-f(v)+f(u)$ over the full relation. Then
\begin{gather*}
 S=a_qP+(c_{q+1}-c_q)Q,\
 P\ge L/c_q,\qquad S\ge L,\qquad
 S\ge c_qP-2c_qa_q+2L.
\end{gather*}
For the last inequality, the contribution of the short pairs is at least $c_qP-2(c_qa_q-L)$, because the missing incidences have variables at most one. If $a_q\le c_{q+1}-c_q$, the first two lower bounds give
\[
 (c_{q+1}-c_q)(P+Q)
 \ge\bigl(a_q-\min\{a_q,a_{q-1}\}\bigr)(c_{q+1}-a_q).
\]
If $a_q>c_{q+1}-c_q$, put $r=a_q-(c_{q+1}-c_q)$. The hypothesis gives $0<r\le c_q$, and
\[
 (c_{q+1}-c_q)(P+Q)
 \ge\max\{L-rP,(c_q-r)P-2c_qa_q+2L\}.
\]
The first affine function is decreasing and the second is nondecreasing. Their maximum is therefore minimized where they meet, at $P=2a_q-L/c_q$, and has value $L(1+r/c_q)-2a_qr$. Substitution of the lower bound on $L$ proves in both cases that $2c_q+P+Q\ge\gamma_q$. Apply this estimate to a radial distance function along a geodesic. The gain is exactly $(b_i-c_i)-(b_{i+q}-c_{i+q})$, an integer, so it is at least $\ceil{\gamma_q}$. Applying Lemma~\ref{lem:endpoint} now gives the stated diameter bounds.
\end{proof}

\begin{remark}\label{rem:older-transport}
When $a_{q-1}=0$, Theorem~\ref{thm:penalty} gives exactly the coefficient $2c_q+\ceil{\frac{a_q(c_{q+1}-a_q)}{c_{q+1}-c_q}}$ from \cite{CL}, including the integer rounding. Allowing $a_{q-1}>0$ also lets us use larger scales at which that result does not apply. For example, the dual polar graph $C_D(3)$, with $D\ge3$, has initial data $a_1=2,c_2=4,a_2=8,c_3=13$, giving $\ceil{\gamma_2}=12$. Since $a_1>0$, the earlier result cannot be used at $q=2$; its available $q=1$ coefficient is $4$. For bounds obtained from these initial data, the leading term decreases from $k/2$ to $k/3$.
\end{remark}

\begin{corollary}\label{cor:clean}
Let $G$ be a distance-regular graph of diameter $D$ and valency $k$. Let $q$ be an integer with $2\le q<D$. Assume
\begin{equation}\label{eq:longhop-condition}
 c_{q+1}>c_q,\qquad
 c_{q+1}\ge a_q+
 \frac{2c_q\min\{a_q,a_{q-1}\}}{a_q+\min\{a_q,a_{q-1}\}},
\end{equation}
where the fraction is zero when $a_q=0$. Then, for every integer $p$ with $0\le p\le D$,
\begin{equation}\label{eq:two-endpoint-bound}
 D\le
 \begin{cases}
 \displaystyle 2p-1+q\left(\floor{\frac{b_p-c_p}{c_q}}+1\right),
 &b_p\ge c_p,\\[6pt]
 2p-1,&b_p<c_p.
 \end{cases}
\end{equation}
\end{corollary}
\begin{proof}
Put $h=\min\{a_q,a_{q-1}\}$ and $\delta=c_{q+1}-c_q>0$. The hypotheses imply $a_q\le c_{q+1}$. The numerator defining $\gamma_q-2c_q$ in Theorem~\ref{thm:penalty} is
\[
 \begin{cases}
 (a_q-h)(c_{q+1}-a_q),&a_q\le\delta,\\
 (a_q+h)(c_{q+1}-a_q)-2c_qh,&a_q>\delta.
 \end{cases}
\]
Both expressions are nonnegative, the second by \eqref{eq:longhop-condition}. Hence $\ceil{\gamma_q}\ge2c_q$. The second bound of Theorem~\ref{thm:penalty} now gives \eqref{eq:two-endpoint-bound}: when $b_p\ge c_p$, use $\frac{2(b_p-c_p)}{\ceil{\gamma_q}}\le\frac{b_p-c_p}{c_q}$; when $b_p<c_p$, the maximum with zero in that bound is zero.
\end{proof}

\begin{remark}\label{rem:np-comparison}
Theorem~1.2 of \cite{NP} assumes
\[
 c_{q+1}>c_q,\qquad c_{q+1}\ge a_q+c_q,
\]
and gives
\begin{equation}\label{eq:np-diameter-bound}
 D\le q\left(\floor{\frac{k-c_{q+1}-1}{c_q}}+2\right)+1.
\end{equation}
Put $h=\min\{a_q,a_{q-1}\}$. Since $2c_qh/(a_q+h)\le c_q$, with the zero-denominator convention in \eqref{eq:longhop-condition}, their hypotheses imply those of Corollary~\ref{cor:clean}. When $a_q>0$, the second threshold is lowered by $c_q(a_q-h)/(a_q+h)$, which is positive if $a_{q-1}<a_q$.

Taking $p=q+1$ in \eqref{eq:two-endpoint-bound} also improves \eqref{eq:np-diameter-bound} throughout their domain; when $c_{q+1}=k$, the terminal identity gives the comparison directly. Indeed, put
\[
 M=\floor{\frac{k-c_{q+1}-1}{c_q}},\qquad
 B_{\rm NP}=q(M+2)+1,
\]
and let $B_{\rm new}$ be the right-hand side of \eqref{eq:two-endpoint-bound} at $p=q+1$. If $c_{q+1}=k$, then $b_{q+1}=0$ forces $D=q+1=B_{\rm NP}$. Otherwise $M\ge0$. If $b_{q+1}<c_{q+1}$, then $B_{\rm new}=2q+1\le B_{\rm NP}$. If $b_{q+1}\ge c_{q+1}$, integrality gives $c_{q+1}\ge c_q+1$, and hence
\[
 \begin{aligned}
 b_{q+1}-c_{q+1}
 &=(k-c_{q+1}-1)-(a_{q+1}+c_{q+1}-1)\\
 &\le k-c_{q+1}-1-c_q.
 \end{aligned}
\]
Thus, $\floor{(b_{q+1}-c_{q+1})/c_q}\le M-1$, so $B_{\rm new}\le2q+1+qM=B_{\rm NP}$.

In the branch $b_{q+1}\ge c_{q+1}$, the sufficient condition $c_{q+1}+a_{q+1}\ge2c_q+1$ strengthens the last floor bound to $M-2$, improving the NP bound by at least $q$. In the branch $b_{q+1}<c_{q+1}<k$, the improvement is exactly $qM$, and is strict when $k\ge c_{q+1}+c_q+1$. The leading coefficient in $k$ for a fixed intersection prefix remains $q/c_q$.
\end{remark}

\subsection{The initial intersection data of the Foster graph}
\begin{theorem}\label{thm:foster}
Let $G$ be a distance-regular graph of diameter $D$ with initial intersection data
\[
 \{3,2,2,2,2,\ldots;1,1,1,1,2,\ldots\}.
\]
Then $D\le8$. This is sharp, since the Foster graph has these initial data and diameter eight.
\end{theorem}
\begin{proof}
Corollary~\ref{cor:clean}, with $q=4$ and $p=5$, gives $D\le9$. Suppose $D=9$. Since $c_5=2$, monotonicity and $b_i\ge1$ for $i<9$ force $c_6=c_7=c_8=2$, $b_5=b_6=b_7=b_8=1$, and $a_5=\cdots=a_8=0$. Hiraki's inequality \cite[Lemma~3.1]{Hiraki} gives $c_9\ge c_5+c_4=3$, so $c_9=3$, $a_9=0$. The forced array is therefore
\[
 \{3,2,2,2,2,1,1,1,1;1,1,1,1,2,2,2,2,3\}.
\]
The sphere sizes, obtained from $k_{i+1}c_{i+1}=k_i b_i$, are
\[
 1,3,6,12,24,24,12,6,3,1,
\]
giving $92$ vertices. Its intersection matrix has characteristic polynomial
\[
 (x^2-9)(x^4+x^3-6x^2-4x+4)(x^4-x^3-6x^2+4x+4).
\]
This identity follows, for example, by the tridiagonal determinant recurrence $p_{i+1}=xp_i-b_{i-1}c_i p_{i-1}$, $p_0=1$, $p_1=x$. Both quartics are irreducible over $\mathbb Q$. For the first, reduction modulo three is $x^4+x^3+2x+1$, which has no root and is divisible by none of the three monic irreducible quadratics $x^2+1$, $x^2+x+2$, $x^2+2x+2$ over $\mathbb F_3$. The other follows by replacing $x$ by $-x$. All adjacency eigenvalues of a distance-regular graph occur among the eigenvalues of its intersection matrix. Each irreducible quartic's four roots have the same integer multiplicity in an integral characteristic polynomial. The graph is connected and bipartite, so $3$ and $-3$ each have multiplicity one. Consequently, $92-2$ must be divisible by four, a contradiction.
\end{proof}

A further question concerns diameter bounds in terms of the valency alone. Bang, Dubickas, Koolen and Moulton \cite{BDKM} proved the Bannai--Ito conjecture: for each fixed integer $k>2$, there are only finitely many finite distance-regular graphs of valency $k$. Equivalently, their diameters are bounded by a function $f(k)$ depending only on $k$. Tracking the constants in their proof gives a bound expressed as an exponential tower of height $O(k^2)$. It is natural to ask for the order of the smallest possible such function. Can $f(k)$ be bounded by an exponential tower whose height is independent of $k$? More strongly, can one obtain a polynomial, or even a linear, bound in $k$?

\section{\texorpdfstring{$(s,c,a,k)$-graphs}{(s,c,a,k)-graphs}}\label{sec:scak}\label{sec:long-finiteness}
The girth condition ensures that vertices at distance less than $s$ are joined by unique geodesics. We use this structure to count paths between distance layers and form the corresponding incidence matrices. We first prove logarithmic diameter bounds, then establish finiteness in wider parameter ranges and a separate spectral criterion for $s=4$, $c=3$.

\subsection{Logarithmic diameter bounds}\label{sec:longscale}
The following lemma reduces the arguments to regular graphs and bipartite semiregular graphs.
\begin{lemma}[{\cite[Lemma~3.2]{Ter}}]\label{lem:terwilliger-regularity}
Every $(s,c,a,k)$-graph is regular or bipartite semiregular. Vertices at distance $s-1$ have equal degrees. In particular, the graph is regular if $a>0$ or $s$ is even.
\end{lemma}

\begin{reptheorem}{th:general}
Let $G$ be an $(s,c,a,k)$-graph with $c>2$.
\begin{enumerate}
\item If $s$ is even, $a\ge1$, and $k\ge8sa^2$, then $\diam(G)\le24s\log k/\log c$.
\item If $a\ge1$ and $k\ge2^{24}a^4$, then $\diam(G)\le32s\log k/\log c$.
\item If $a=0$, then
\[\diam(G)\le\frac{2(s-1)\log(2k)}{\log(c-1)}+10s.\]
\end{enumerate}
\end{reptheorem}
We prove the three parts separately, starting with a common reduction. By Lemma~\ref{lem:terwilliger-regularity}, $G$ is regular when $a>0$. The girth condition in the definition gives $\diam(G)\ge s-1$. If a regular graph has $\diam(G)=s-1$, its last layer gives $a=k-1$, which is incompatible with either displayed high-degree hypothesis. When $a=0$, the girth is $2s$, so $\diam(G)\ge s$ directly. Thus, $\diam(G)\ge s$ and $c\le k$ in every use of these hypotheses. In the regular case, the girth condition in the definition of an $(s,c,a,k)$-graph, together with $\diam(G)\ge s$, gives
\begin{equation}\label{eq:tree}
 c_i=1\ (1\le i<s),\qquad a_i=0\ (i<s-1),\qquad
 a_{s-1}=a,\quad c_s=c,\quad k\ge a+2.
\end{equation}

\begin{proof}[Proof of Theorem~\ref{th:general}, part~1]
For $s=2$, apply Theorem~\ref{thm:arg-logarithmic}\textup{(ii)} with $\eta=0$. Since $k\ge\max\{c,16a^2\}$, the resulting bound is at most $18\log k/\log c$, as required. Hence assume $s\ge4$. Write $s=2p$, $p\ge2$, $r=k-1$, $n=r^{p-1}$ and $v=kn$. The first $s-1$ layers have unique geodesics. Their distance matrices $A_j$ satisfy the tree recurrence, so the kernels $P_j$ commute.

Fix a root, choose $Y$ uniformly in $U=S_p$, then $Z$ uniformly in $S_p(Y)$. If $N$ is the corresponding incidence matrix, its column multiplicities are the coefficients in
\begin{equation}\label{eq:evenprod}
 A_p^2=vI+\sum_{j=1}^{p-1}\mu_jA_{2j}+aA_{2p-1}+cA_{2p},
 \qquad\mu_j=(r-1)r^{p-j-1}.
\end{equation}
This follows by tree path counting, with the final two coefficients replaced by the defining $a,c$. Thus, $TT^*=N\operatorname{diag}(m_z^{-1})N^T/v$ for the conditional expectation operator. Let $F$ be horizontal adjacency on $S_{2p-1}$, $B$ its incidence to $S_{2p}$, and $R$ descendant incidence from $U$ to $S_{2p-1}$. Then $RR^T=nI$ and the outer blocks of $N$ are $RF,RB$. Let $N_j$ be incidence from $U$ to $S_{2j}$ at distance $p$, $H_j$ the distance-$2j$ matrix on $U$, and $J_0$ the all-ones blocks on first branches. Restriction of \eqref{eq:evenprod} to $U$ gives, on $\one^\perp$,
\begin{align}
 TT^*&=K_0+\frac{c-a}{acv}R(F^2-aF)R^T,\label{eq:evengram}\\
 K_0&=\frac{1-1/k}{c}I-\frac{c-1}{cv}J_0+
 \sum_{j=1}^{p-1}\frac{N_jN_j^T}{v\mu_j}.\notag
\end{align}
The cancellation uses
\[\sum_jN_jN_j^T=(n-1)I+\sum_j(\mu_j-1)H_j,
 \qquad\sum_jH_j=J_0-I.\]
Both identities can be counted on the regular tree: its outer Gram is $r^pI$, its root column gives $J$, and its distance-$2p$ relation is $J-J_0$. All remaining matrices are unchanged interior tree matrices.

Decompose $U$ into nested branch difference spaces, with level $h$ constant below depth $h$ and summing to zero in every parent block. On level $h$, the eigenvalue of $N_jN_j^T$ is respectively
\[
 (r-1)^2r^{p-2}\ (h\le j),\qquad
 r^{p-2}\ (h=j+1),\qquad0\ (h>j+1).
\]
These follow by summing descendants outside the child branch containing the endpoint. The resulting inner sum has eigenvalue $(n-1)/v$ at $h=1$, and $1/k-r^{h-1}/v+r^{h-2}/[v(r-1)]\le1/k$ for $h\ge2$. Consequently, $K_0\le[1/c+(c-1)/(ck)]I$. With $E=\max\{2a(c-a),a(a-c)/4\}$, \eqref{eq:evengram} gives
\[TT^*|_{\one^\perp}\le uI,\qquad u=(k+c-1+E)/(ck).\]
For $k\ge32a^2$ one has $u\le1/\sqrt{2c}$. If $c<a$, then $u\le[1+1/(32c)+1/128]/c\le1/\sqrt{2c}$. If $a\le c$, replace $k$ by $\max\{c,32a^2\}$ and optimize in $a$, splitting at $a=\sqrt{c/32}$. The resulting bounds are
\[
 u\le
 \begin{cases}
  (3c+29)/(32c),&3\le c\le32,\\
  1/\sqrt{8c}+31/(16c)-1/c^2,&c\ge32.
 \end{cases}
\]
In the first range, the comparison with $1/\sqrt{2c}$ is equivalent to $9c^2-338c+841\le0$, which holds by convexity and its endpoint values. In the second range, multiplication by $\sqrt{2c}$ gives at most $1/2+31/64<1$. Thus, $\rho=u/(1-u)\le1/\sqrt{c-1}$. Total variance now yields, for $V_p(f)=[P_p(f^2)-(P_pf)^2]/2$,
\begin{equation}\label{eq:evenvar}
 V_p(P_pf)\le\rho P_pV_p(f),\qquad
 \|V_p(P_p^mf)\|_\infty\le(p^2/2)\rho^m
\end{equation}
for bounded $1$-Lipschitz $f$.

To recover the original metric without a sphere-size loss, let $B_pf(z,x)$ average the leaves at depth $p$ beyond the neighbour $x$ of $z$. For distinct $x,y\sim z$, path counting gives
\[
 P_{p-1}f(x)-P_{p-1}f(y)
 =\frac rk(B_pf(z,x)-B_pf(z,y))
 -\frac1{kr}(B_{p-2}f(z,x)-B_{p-2}f(z,y)).
\]
For $p=2$ the last term vanishes; $B_1f(z,x)=f(x)$. Normalize the positive polynomial
\[
 L_p=\frac kr\sum_{j=0}^{\lfloor(p-2)/2\rfloor}r^{-2j}P_{p-1-2j}
 +\one_{p\text{ odd}}r^{-(p-1)}I
\]
by its coefficient sum $m_p\ge1$, obtaining a Markov kernel $S_p$. Iteration of the preceding identity gives $S_pf(x)-S_pf(y)=[B_pf(z,x)-B_pf(z,y)]/m_p$. Variance among the $k$ first branches therefore implies $|S_pf(x)-S_pf(y)|^2\le4kV_p(f)(z)$. Set $h_m=(I+P)S_pP_p^mf/2$. By \eqref{eq:evenvar},
\[
 \Lip(h_m)\le p\sqrt{2k}\rho^{m/2},\qquad
 \|h_{m+1}-h_m\|\le\min\{p,p^2\sqrt{2k}\rho^{m/2}\}.
\]
The edge estimate follows by comparing each endpoint with the neighbours of the other. Commutation gives the first increment bound, and jump length gives the second. Also $\|f-h_0\|\le p-1/2$. Summation gives
\[
 \diam(G)<\frac{4p\log(2kp^2)}{\log(c-1)}+17p.
\]
Under $k\ge8sa^2$, we have $p\le k/16$ and $k\ge32a^2$. Using $\log(2kp^2)\le3\log k$, $c\le k$, and $\log c/\log(c-1)\le\log3/\log2$ in this estimate yields part~1.
\end{proof}

We next remove the parity restriction on $s$.

\begin{proof}[Proof of Theorem~\ref{th:general}, part~2]
At $s=2$, apply Theorem~\ref{thm:arg-logarithmic}\textup{(ii)} with $\lambda=a,\mu=c,\eta=0$; hence take $s\ge3$. We use an inverse-energy estimate.\footnote{The notation, Schur-complement calculation, and metric estimate in this proof are also used in the proof of part~3 below.} Put $t=s-1$, $r=k-1$, $n=r^{t-1}$, $v=kn$ and $Q=P_t$. At a root write $U=S_1$, $Y=S_{t-1}$, $Z=S_t$, $W=S_{t+1}$. Let $R,R_Y$ be first-branch incidence to $Z,Y$, $C$ parent incidence from $Y$ to $Z$, $F$ horizontal adjacency on $Z$, and $B$ incidence to $W$. Then $RR^T=nI$, $R_YR_Y^T=nI/r$, $R=R_YC$, $CC^T=rI$, $F\one=a\one$, $B\one=(r-a)\one$, $B^T\one=c\one$. For $u\perp\one$, set $T=R^Tu$ and $E=C^TC/r$. The short-path counts are
\begin{equation}\label{eq:generalcounts}
 EFT=-aT/r,\qquad BB^TT=[(k-c)I+(a-c)F-F^2]T.
\end{equation}
Indeed $RFC^T=a(J-R_Y)$: for a parent in another branch exactly $a$ children remain at distance $t$ from the first neighbour. Also $RF,RB$ are zero-one distance-$t$ incidences. Counting neighbours at that distance gives $R(F^2+BB^T)=(k-c)R+(a-c)RF+(c-1)J$, proving the second identity.

Define
\[
 \begin{aligned}
 N&=[C^T,F,B],\qquad\Delta=\operatorname{diag}(rI_Y,aI_Z,cI_W),\\
 K&=N\Delta^{-1}N^T/k,\qquad A=k(I-K),\qquad H=\Delta-N^TN/k.
 \end{aligned}
\]
Conditional expectation gives $0\le K\le I$, and
\begin{equation}\label{eq:energy}
 A=(I-E)+(aI-F^2/a)+((r-a)I-BB^T/c)\ge0,
\end{equation}
with each summand positive semidefinite. The distance-$t$ incidence from $U$ to the terminal layers is $M=[J-R_Y,RF,RB]$. For $m=M^Tu$,
\begin{align}
 N\Delta^{-1}m&=k(K-1/r)T,\label{eq:source}\\
 m^T\Delta^{-1}m&=k\ip{T}{KT}-(1-1/r^2)\|T\|^2.\notag
\end{align}
Put $h=(a-c)/(ac)$, $D_F=aI-F^2/a$, $w=(I-E)FT$, and
\[L=\frac kc\left(c-1+\frac{a(a-c)}r\right).
\]
Equations \eqref{eq:generalcounts} imply the exact identity
\begin{equation}\label{eq:identity}LT=AT+haD_FT+haw.\end{equation}
By \eqref{eq:energy}, $\ker A\subseteq\ker D_F\cap\ker(I-E)$. Taking the inner product of \eqref{eq:identity} with a vector in $\ker A$ therefore gives $T\perp\ker A$ when $L>0$. For unit $T$, set $m_2=\|FT\|^2\le a^2$. Then
\[
 \ip{T}{D_FT}+\|w\|^2=a-m_2/a+m_2-a^2/r^2\le a+a^2,
 \qquad\ip{T}{AT}\le L+|h|a^2.
\]
Cauchy--Schwarz in the positive forms $A,D_F,I-E$ and the variational formula for the pseudoinverse give
\begin{equation}\label{eq:inverse}
 \frac{\ip{T}{A^+T}}{\|T\|^2}
 \le\frac{[\sqrt{L+|h|a^2}+|h|a\sqrt{a+a^2}]^2}{L^2}.
\end{equation}
To apply the Schur identity, first note that $m\perp\ker H$: if $z\in\ker H$, then $w_0=Nz/k\in\ker A$, $z=\Delta^{-1}N^Tw_0$, and $m^Tz=k\ip{(K-1/r)T}{w_0}=0$. We may therefore use
\[
 m^TH^+m=m^T\Delta^{-1}m+
 k^{-1}(N\Delta^{-1}m)^T(I-K)^+(N\Delta^{-1}m).
\]
Substituting \eqref{eq:source} and simplifying gives
\begin{equation}\label{eq:schur}
 m^TH^+m=q_0^2\ip{T}{A^+T}-B_0\|T\|^2,
 \quad q_0=k(1-1/r),\quad B_0=r-2/r-1/r^2.
\end{equation}

Let $V(f)=[P(f^2)-(Pf)^2]/2$. Since $\|T\|^2=n\|u\|^2$, a bound $m^TH^+m\le\rho v\|u\|^2$ is equivalent by duality to
\begin{equation}\label{eq:allvariance}V(Qf)\le\rho QV(f).\end{equation}
Indeed the two quadratic forms are $\|\operatorname{proj}_{\one^\perp}Mf\|^2/(2kv^2)$ and $f^THf/(2kv)$. Define
\[
 \ell=L/k,\qquad e_0=|h|a^2/k,\qquad
 \varepsilon=|h|a\sqrt{a+a^2}/\sqrt k.
\]
Equations \eqref{eq:inverse}--\eqref{eq:schur} give the sufficient bound
\begin{equation}\label{eq:rhoall}
 \rho\le\frac{[\sqrt{\ell+e_0}+\varepsilon]^2}{\ell^2}-1+\frac{k}{r^2}.
\end{equation}
Under $k\ge2^{24}a^4$, one has $\ell\ge1/2$, $\ell+e_0<2$, $|\ell-(c-1)/c|\le2e_0$, and $e_0\le\varepsilon/\sqrt k$. Expansion of \eqref{eq:rhoall} therefore gives
\[
 \rho\le\frac1{c-1}+13\varepsilon+4\varepsilon^2+4/k.
\]
If $c<a$, $\varepsilon\le\sqrt2/(4096c)$; if $c\ge a$, $\varepsilon\le\sqrt2 k^{-1/4}/64$. Thus, $\varepsilon\le(3/128)c^{-1/4}$, while $1/(c-1)\le(2/3)c^{-1/4}$ and $4/k\le2^{-16}c^{-1/4}$. The coefficient $2/3+39/128+36/16384+1/65536$ is less than one, so
\begin{equation}\label{eq:rhochoice}\rho\le c^{-1/4}<1.\end{equation}

For bounded $1$-Lipschitz $f$, \eqref{eq:allvariance} controls distance-two differences of $Q^mf$ by $\sqrt{2k}\rho^{m/2}$. Set $g_m=(I+P)Q^mf/2$. Then $\Lip(g_m)\le\sqrt{2k}\rho^{m/2}$, $\|f-g_0\|\le1/2$, and commutation gives $\|g_{m+1}-g_m\|\le t\min\{1,\sqrt{2k}\rho^{m/2}\}$. Consequently,
\begin{equation}\label{eq:metricall}
 \diam(G)\le1+2t\left\lceil\frac{\log(2k)}{\log(1/\rho)}\right\rceil
 +\frac{2t}{1-\sqrt\rho}.
\end{equation}
Substituting \eqref{eq:rhochoice} into \eqref{eq:metricall} yields the bound in part~2.
\end{proof}

When $a=0$, the same calculation simplifies and also covers the bipartite semiregular case.

\begin{proof}[Proof of Theorem~\ref{th:general}, part~3]
For $s=2$, use \eqref{eq:triangle-free-refined} with $\lambda=0,\mu=c$; hence assume $s\ge3$. Use the notation from the proof of part~2. For regular graphs with $a=0$, omit the middle block of $N,\Delta$. Equation \eqref{eq:generalcounts} gives $KT=T/c$ exactly. The same Schur calculation yields
\[
 \rho_0=\frac1k\left[\frac{k-c}c+\frac1{r^2}
 +\frac{k(r-c)^2}{cr^2(c-1)}\right]\le\frac1{c-1}.
\]

For unequal bipartite degrees, $s$ is odd and $t=s-1$ even. At a root let $q$ be its degree, $l$ its neighbours' degree, $u=q-1$, $b=l-1$. Both degrees are at least $c$. Each first branch at depth $t$ has $n=b^{t/2}u^{t/2-1}$ leaves. The sphere sizes at adjacent roots are $qn$ and $lnu/b$. The previous matrices now have parent column sum $b$, row degree $q$, and $\Delta=\operatorname{diag}(bI,cI)$. Direct substitution gives $KT=T/c$ and
\[
 N\Delta^{-1}m=(q/c-l/b)T,
 \quad m^T\Delta^{-1}m=[(q-c)/c+1/b^2]\|T\|^2.
\]
Accounting for the two different sphere sizes in the variance normalization gives
\[
 \rho(q,l,c)=\frac{qb^2}{l^2(q-1)^2}
 \left[\frac{q-c}c+\frac1{b^2}+
 \frac{(q/c-l/b)^2}{q(1-1/c)}\right]\le\frac1{c-1}.
\]
To check the last inequality, the difference of the relevant numerators is
\[
 (c-1)(q-1)b^2+2(q^2-q+1-c)b+q^2-(c+1)q+1-c.
\]
For $q\ge c$, $b\ge c-1$, it is at least $(c-1)^2(b^2+2b)+1-2c>0$. Thus, the variance contraction holds uniformly at every root. Here $P$ need not commute with $P_t$. Instead put $h_m=P_t^mf$ and $g_m=(h_m+Ph_m)/2$. Since $t$ is even, a $P_t$ jump splits into $t/2$ distance-two steps, giving directly
\[
 \|h_{m+1}-h_m\|\le
 \min\{t,(t/2)\sqrt{2k}(c-1)^{-m/2}\}.
\]
Averaging preserves this bound, and $\Lip(g_m)\le\sqrt{2k}(c-1)^{-m/2}$. The summation argument in \eqref{eq:metricall} proves part 3, since $2+2/(1-1/\sqrt2)<9$.

\end{proof}

\subsection{Compression and a reference curvature bound}\label{sec:finite-tools}
The high-degree assumptions in Theorem~\ref{th:general} are useful for logarithmic diameter estimates. Finiteness itself holds in the following wider ranges.

\begin{reptheorem}{thm:long-finiteness}
Every $(s,c,a,k)$-graph with $s>2$, $c>2$, and $(s,c)\ne(4,3)$ is finite.
\end{reptheorem}

Theorem~\ref{thm:c2-tail} is proved separately in Section~\ref{sec:c2tail}. For each fixed tuple in these two theorems, the proofs give a uniform diameter bound, hence a bound on the order and only finitely many isomorphism classes. At \(s=2\), further ranges follow from Theorems~\ref{thm:arg-even-finiteness} and~\ref{thm:arg-242}.

It suffices to consider \(\diam(G)\ge s\) and \(a>c\): bounded degree gives finiteness when \(\diam(G)<s\), and \cite[Theorem~3.4]{Ter} handles \(a\le c\). By Lemma~\ref{lem:terwilliger-regularity}, \(G\) is \(k\)-regular. Put \(r=k-1\); then \eqref{eq:tree} holds. We use the generalized curvature pairs of Section~\ref{sec:jump-tools}. The argument has three steps: express \(\Bcal_{\Gcal,L}\) by the tree-layer incidences; eliminate the inner and boundary variables while retaining their correlations; and show positivity of the resulting scalar matrices. The comparison with the original metric is provided by Lemma~\ref{lem:semi}.

For symmetric matrices $X$ and $Y$, we write $X\succeq Y$, or $Y\preceq X$, if $X-Y$ is positive semidefinite.

\begin{lemma}\label{lem:finite-compression}
Let \(E\) be an orthogonal projection on a finite-dimensional real inner product space, and let \(K\) be a symmetric operator on that space with \(0\preceq K\preceq\Id\). Put \(\widehat E=\Id-E\), and let \(0\le d<1\). On the range of \(K\), put
\[
 M_0=E(K^\dagger-d\widehat E)^{-1}E,\qquad A=EKE.
\]
Range inverses here mean the effective inverse in the corresponding quadratic minimization, or a positive regularization and limit. Then
\begin{equation}\label{eq:compression}
 M_0\preceq A(\Id-d\Id+dA)^{-1}.
\end{equation}
If, in addition, \(A\preceq U=\sum_\alpha u_\alpha E_\alpha\), where the \(E_\alpha\) are mutually orthogonal projections summing to \(E\) and \(0\le u_\alpha\le1\), then for \(J_E=\sum_\alpha J_\alpha E_\alpha\), \(J_\alpha>-1\), we have
\begin{equation}\label{eq:rho}
 (M_0^{-1}+J_E)^{-1}\preceq
 \sum_\alpha
 \frac{u_\alpha}{1-d+(d+J_\alpha)u_\alpha}\,E_\alpha.
\end{equation}
\end{lemma}
\begin{proof}
First assume \(0\prec K\prec\Id\), and write \(B=EK\widehat E\), \(C=\widehat E K\widehat E\). From \(K^2\preceq K\) and the Schur complement of \(\Id-K\), respectively, we obtain \(BB^*\preceq A(\Id-A)\) and \(C\preceq\Id-B^*(\Id-A)^{-1}B\). These give \(B(\Id-dC)^{-1}B^*\preceq A(\Id-A)(\Id-d\Id+dA)^{-1}\). The block-inverse formula \(M_0=A+dB(\Id-dC)^{-1}B^*\) then proves \eqref{eq:compression}. Since \(A\mapsto A((1-d)\Id+dA)^{-1}\) is operator monotone, we may replace \(A\) by \(U\) in this upper bound. Inverting, adding \(J_E\), and inverting again gives \eqref{eq:rho}. Its denominators are positive: their endpoint values at \(u_\alpha=0,1\) are \(1-d\) and \(1+J_\alpha\). Singular cases follow by positive regularization and a limit.
\end{proof}

\begin{proposition}\label{prop:even-reference}
Let \(G\) be an \((s,c,a,k)\)-graph with \(\diam(G)\ge s\), even \(s=2q\ge4\), and \(a>c\ge3\). Put \(r=k-1\), \(v_q=kr^{q-1}\), \(n_q=r^{q-1}\), and
\[
 \alpha_q=\begin{cases}0&q\text{ odd},\\
 (r-1)r^{q/2-1}&q\text{ even},\end{cases}\qquad
 \kappa_0=2+\alpha_q/2,\quad
 \Delta=2k(c-2)+4-a(a-c).
\]
The Bakry--\'Emery curvature of the graph on \(V(G)\) joining vertices at distance \(q\) in \(G\) is at least
\begin{equation}\label{eq:evenreference}
 \kappa_0+\min\{0,E_q\},\qquad
 E_q=\begin{cases}
 r\Delta/(2c)&q=2,\\
 n_q\Delta/(2c)-2n_q/[r(r-1)]-2&q\ge3.
 \end{cases}
\end{equation}
In particular, \(G\) is finite whenever \(\Delta>0\). More generally, a positive lower bound \(K\) in \eqref{eq:evenreference} gives \(\diam(G)\le q\floor{2v_q/K}+q-1\).
\end{proposition}
\begin{proof}
In the distance-jump graph, fix a root and let \(B\) be the adjacency matrix on the vertices at distance one from it. Let \(N_j\) be the incidence matrix from these vertices to those at distance two in the jump graph and at distance \(j\) from the root in the original graph. If the corresponding common neighbour count is \(\mu_j\), completing squares in the function values at distance two in the jump graph gives the curvature matrix
\[
 (v_q+\alpha_q/2)\Id+\mathbb J-2B
                -2\sum_jN_jN_j^{\mathsf T}/\mu_j.
\]
Tree meeting depths give \(\mu_{2j}=(r-1)r^{q-j-1}\) for \(1\le j<q,\ 2j\ne q\), with the prescribed terminal counts \(a\) and \(c\) at distances \(2q-1,2q\). On the depth-\(\ell\) rooted Haar space, original distance-\(2j\) adjacency and its incidence Gram have eigenvalues respectively
\[
 d_{j\ell}=\begin{cases}(r-1)r^{j-1}&\ell<q-j\\
 -r^{j-1}&\ell=q-j\\0&\ell>q-j,\end{cases}\quad
 n_{j\ell}=\begin{cases}(r-1)^2r^{q-2}&\ell<j\\
 r^{q-2}&\ell=j\\0&\ell>j.\end{cases}
\]
The adjacency \(B\) has eigenvalue zero when \(q\) is odd and \(d_{q/2,\ell}\) when \(q\) is even; the relation within the first distance sphere is omitted from the sum over the second distance sphere. The only terminal non-tree operator \(T\) has root-branch eigenvalue \(-ar^{q-2}\), vanishes on intermediate Haar levels, and has \(\norm T\le ar^{q-1}\) on the finest level. Its terminal incidence Gram dominates \(T^2/r^{q-1}\), since the normalized ancestor incidence is an isometry. Thus, its remaining adverse expression is bounded below by
\[
 \frac{2(a-c)}{ac r^{q-1}}
       (T^2-ar^{q-1}T)\succeq
       -\frac{a(a-c)r^{q-1}}{2c}\Id.
\]
Substitution of the displayed tree eigenvalues in the curvature matrix and summation of the geometric series gives \eqref{eq:evenreference}; at \(q=2\) there are no intermediate remainder terms. This proves the reference estimate using only \eqref{eq:tree}.

If \(\Delta>0\), the bound is positive: for \(c\ge4\), \(r\ge c+2\) and \(r^2/(2c)-2r/(r-1)-2>0\); for \(c=3\), \(\Delta\) is even, so a positive value is at least two, and \(r^2/3-2r/(r-1)-2>0\) for \(r\ge5\). For any positive lower bound \(K\), the jump-graph semigroup argument bounds its component diameters by \(\floor{2v_q/K}\). Sampling an original geodesic every \(q\) steps gives a jump-graph geodesic, since a shorter jump path would shorten the original distance. Hence \(\diam(G)\le q\floor{2v_q/K}+q-1\), as claimed.
\end{proof}

\paragraph{Scalar positivity checks.}
The reductions below involve matrices of order two and rational functions of the parameters. We use the following elementary tests. For \(p(x)=\sum_{j=0}^m a_jx^j\), write
\[
 p(x)=\sum_{i=0}^m b_i\binom mi x^i(1-x)^{m-i},\qquad
 b_i=\sum_{j=0}^i a_j\frac{\binom{i}{j}}{\binom{m}{j}}.
\]
Positive \(b_i\)'s imply \(p>0\) on \([0,1]\); subdivision and affine substitution give the same test on smaller intervals. For a rational function on \(r\ge r_0\), clear its positive denominator and substitute \(r=r_0+z\). Nonnegative coefficients and a positive constant prove positivity throughout the half-line. We use the analogous substitution in two variables for the sparse range at \(s=3,c=3\). Finite integer ranges are checked by substitution in the stated criteria, including their denominator conditions.
Certificates and verification programs for the scalar positivity and finite parameter checks in this section are provided in the supplementary material.

\subsection{Scales three and four}\label{sec:s3}\label{sec:s4}
\begin{proof}[Proof of Theorem~\ref{thm:long-finiteness} for $s=3$]
It is enough to treat \(a>c\). The main criterion retains the first-branch variance when eliminating the outer shell; at \(c=3\), a cross identity gives a further improvement. Fix \(o\), write \(U=S_1(o),Z=S_2(o),W=S_3(o)\); let \(R:U\times Z\) be parent incidence, \(F\) adjacency on \(Z\), and \(P:Z\times W\) incidence. For functions of mean zero on each distance sphere, put
\[
 \begin{aligned}
 T&=R^{\mathsf T}f_U,\qquad E=R^{\mathsf T}R/r,\qquad\widehat E=\Id-E,\\
 z&=y+\xi,\quad y=Ez,\quad G=F/k,\quad g=a/k,\quad N=PP^{\mathsf T}/(ck).
 \end{aligned}
\]
Thus, \(\norm T^2=r\norm{f_U}^2\). The constant shell mode is treated separately. The last-layer count gives
\[
 PP^{\mathsf T}T=[(k-c)\Id+(a-c)F-F^2]T,\qquad
 EGE=-\frac{g}{r}E.
\]
With \(B=\widehat E GE,\ Q=B^*B,\ h=(r-a)/k\), it follows that
\begin{equation}\label{eq:s3geom}
 0\preceq N\preceq h\Id,\quad
 ENE=uE-\frac{k}{c}Q,\quad
 u=\frac{1-a^2/r^2}{c}-\frac{r-a}{kr},\quad
 0\preceq Q\preceq q_mE,\quad q_m=cu/k.
\end{equation}
These identities follow from \eqref{eq:tree}.\footnote{The notation and last-layer counts here are also used in the case $s=4$ below and in Sections~\ref{sec:balanced} and~\ref{sec:c2tail}.}

Use \(L=L_2+\beta L_1\), \(\beta=bk>0\), and \(\Gcal=\Gamma_1+\alpha(L_1f)^2/(2k)\). For \(w=Pf_W/k\), expansion \eqref{eq:mixed} gives the following exact form for \(4\Bcal_{\Gcal,L}/k\):
\begin{align}\label{eq:s3exact}
 &A_o\norm T^2+D_o\norm y^2-2C_o\ip T y
 +\ip{\xi}{[(1+g+b)\Id-2\widehat E G\widehat E]\xi}
 -2\ip{\xi}{B(T+2y)}\\
 &\quad+\ip w{N^\dagger w}-2\ip{T+y+\xi}w
 +\alpha\norm{w-(\Id-G)z+T/k}^2
 +\frac{\alpha b}{kr}\norm{ry-kT}^2,\nonumber\\
&A_o=1+\frac1k+\frac2{kr}+\frac{2bk}{r},\quad
 D_o=1+g+b+\frac{2g}{r},\quad C_o=2b+\frac1k-\frac gr.\nonumber
\end{align}
The inverse cost is the exact minimum at fixed \(w\in\ran N\).

For \(0<t<\alpha\) and \(Y>0\), set
\[
 f_*=\frac{\alpha t}{\alpha-t},\quad
 l=\frac{\alpha b}{kr},\quad \lambda_0=1+g/r,
\]
\[
 A=A_o+lk^2+t/k^2,\quad
 D=D_o+lr^2-f_*q_m+t\lambda_0^2,\quad
 C=C_o+lkr+t\lambda_0/k.
\]
Split only the \(E/\widehat E\) cross part of \(\Id-G\), retain \(t\norm{w-[\lambda_0 E+\widehat E(\Id-G)\widehat E]z+T/k}^2\), and pay \(f_*q_m\norm z^2\). The separate cross term is bounded below by \(-Y\norm\xi^2-\ip{T+2y}{Q(T+2y)}/Y\). Let
\[
 H=\min\{1+g+b-f_*q_m-Y,\ 2+1/t\},\quad
 d_f=\frac{1+t(2-H)}{H-2g+t(1-g)^2}.
\]
If its denominator is positive, \(d_f\ge0\), and \(d_fh<1\), spectral completion in the single self-adjoint \(\widehat E G\widehat E\) gives the lower bound \(-d_f\norm{w_{\widehat E}}^2\) for the terms involving \(\xi\). Applying the scaled version of \eqref{eq:compression}, the remaining matrix at \(Q=q_mx\), \(0\le x\le1\), is bounded below by
\begin{equation}\label{eq:s3matrix}
 \begin{pmatrix} A&-C\\-C&D\end{pmatrix}
 -\frac{q_mx}{Y}\begin{pmatrix}1&2\\2&4\end{pmatrix}
 -\frac{u(1-x)}{1-d_fh+(d_f+t)u(1-x)}
 \begin{pmatrix}w_1^2&w_1w_2\\w_1w_2&w_2^2\end{pmatrix},
\end{equation}
where \(w_1=1-t/k,w_2=1+t\lambda_0\). Its denominator-cleared first minor and determinant are quadratic polynomials. Positive Bernstein coefficients, or their exact minima on \([0,1]\), prove strict positivity and an explicit coefficient \(\eta>0\) for \(\norm T^2\). At \(t=0\), omit \(2+1/t\) from the minimum defining \(H\).

For \(c=3\) we retain two additional correlations. We retain the subtracted term \(f_*Q\) on \(\ran E\), rather than bounding it by \(f_*q_mE\). Thus, we remove \(-f_*q_m\) from \(D\) and subtract \(f_*q_mx\) later in its lower-right entry. Also the full identity gives
\[
 \begin{aligned}
 \widehat E NE&=\frac{k}{3}(a/r-3/k-G_f)B,\qquad G_f=\widehat E G\widehat E,\\
 (\widehat E NE)^*(\widehat E NE)&\preceq L_*Q,\qquad L_*=(k/3)^2(a/r-3/k+g)^2.
 \end{aligned}
\]
Block inversion therefore improves the scalar boundary numerator to
\[
 m(x)=u(1-x)+\frac{d_fL_*q_mx}{1-d_fh}.
\]
Its loss is \(m(x)/(1+tm(x))\) times \((w_1,w_2)(w_1,w_2)^{\mathsf T}\). The resulting first minor and determinant have degree at most two and three. This derivation applies inverse order before diagonalizing \(Q\), and remains valid when \(N\) is singular.

\paragraph{A cubic alternative.}
A second criterion uses \(L_2\) without augmentation. The first-branch lift \(T\) on \(S_2\) has spectral mean \(-a/r\) under \(F\). Completing the previous shell and the outer shell in \eqref{eq:mixed}, and shifting by \(T/(c-1)\), gives \(4\Bcal_{\Gamma_1,L_2}\ge\ip T{G_*(F)T}\), where
\[
 G_*(z)=H_*(z)-J_*(z),
\]
\[
 H_*=\frac{k(c-2)}{c-1}+1-r^{-2}
 +\frac{2(1+1/r)}{c-1}+\frac{a}{(c-1)^2}
 +\frac{c[z^2-(a-c+2)z]}{(c-1)^2},
\]
\[
 J_*=\frac{a-z}{2}\left(\frac{1-z}{c-1}+\frac1a\right)^2
       +\frac{(a-z)^2}{a^2}-(1+1/r)^2.
\]
It has the affine minorant determined by the exact nonnegative remainder
\[
 G_*(z)-L_{\rm aff}(z)
 =\frac{(z+a)(z-z_*)^2}{2(c-1)^2},\qquad
 z_*=a-(c-1)+(c-1)/a+(c-1)^2/a^2.
\]
Thus, \(L_{\rm aff}(-a/r)>0\) is a rational sufficient test. More generally, introduce a completion shift \(\theta\ge0\) and define
\[
 p(z)=k-c+(a-c)z-z^2,\quad
 m(z)=[z^2-(a+c)z+c(k+a)-k]/c,\quad
 b(z)=[z^2-az+c/r-k+c]/c,
\]
\[
 H_\theta=k-r^{-2}-p/c+2\theta b+\theta^2m,\quad
 \ell_0=a(1+1/r),\quad
 \epsilon=\frac{1+1/r+k[\theta(c-1)-1]/c}{\ell_0},
\]
\[
 h_\theta(z)=\theta+\epsilon-(1+\theta)z/c,\qquad
 G_\theta(z)=H_\theta(z)-\tfrac12(a-z)h_\theta(z)^2
                         -\epsilon^2[(a-z)^2-\ell_0^2].
\]
To derive it, complete squares in the second-shell variable about \(\theta T\). Its Hessian dominates \(2(a\Id-F)+(\Id-E_{\rm par})\), where \(E_{\rm par}\) is the immediate-parent projection on the last shell. The remaining linear term equals \((a\Id-F)h_\theta(F)T-\epsilon(\Id-E_{\rm par})(a\Id-F)T\); use \(E_{\rm par}(a\Id-F)T=\ell_0T\) and complete the two squares. This proves \(4\Bcal_{\Gamma_1,L_2}\ge\ip T{G_\theta(F)T}\).

Write \(G_\theta(z)=Az^3+Bz^2+Cz+D\), where \(A=(1+\theta)^2/(2c^2)>0\), and put \(z_0=\max(-a,\min\{a,(Aa-B)/(2A)\})\). Its lower convex envelope is the chord from \(-a\) to \(z_0\), followed by \(G_\theta\) on \([z_0,a]\). For an interior contact, the chord remainder is exactly \(A(z+a)(z-z_0)^2\ge0\); the clipped cases follow by the same factorization or convexity. Jensen's inequality for the spectral measure of \(T\) bounds the quadratic form below by this envelope at the exact spectral mean \(-a/r\), giving a rational sufficient test for positivity.

\paragraph{Parameter ranges.}
For \(c=4,5,6\), consider first \(k\ge2000\). In \eqref{eq:s3matrix} take
\[
 b=1,\qquad \alpha=10000,\qquad t=50,\qquad Y=3/k.
\]
We check the two intervals \([5,k/16]\) and \([k/16,k-2]\) at \(c=4\). On the second interval one may set \(H=2+1/t\), giving \(d_f=0\). The scalar prerequisites and the first minor and determinant are positive for every \(k\ge2000\): after subdivision in \(x\), their Bernstein coefficients have positive shifted expansions at \(k=2000+z\). Increasing \(c\) decreases \(u,q_m\), increases \(H\), and decreases the boundary losses, so the augmented criterion extends to \(c=5,6\). For the remaining degrees we can check that at least one of the augmented and cubic criteria is positive for every integer pair
\[
 7\le k<2000,\quad 5\le a\le k-2 \qquad(c=4).
\]
The cubic cases in this finite range are also positive at \(c=5,6\) whenever \(a>c\). This covers these three values of \(c\).

For \(c\ge7\) there is a simpler contraction criterion. With \(\beta=bk\), choose \(Y>0\) and put
\[
 \epsilon=\frac{k-c-a(a-c)/r-a^2/r^2}{k^2Y},\quad
 A=A_o-\epsilon,\quad D=D_o-4\epsilon,\quad C=C_o+2\epsilon,
 \quad f=b+1/k-Y.
\]
Let \(u_{\rm alt}=1/c-1/k-a(a-c)/(ckr)\), so that \(ENE\preceq u_{\rm alt}E\). Young's inequality in the root--fine block, followed by completion in \(N\), gives the sufficient condition
\begin{equation}\label{eq:s3contraction}
 A,D,C,f>0,\qquad
 \frac{AD-C^2}{A+D+2C}
       -\frac{u_{\rm alt}(h+f)}{f+u_{\rm alt}}>0.
\end{equation}
This condition improves as \(a\) or \(c\) increases: the root coupling and \(u_{\rm alt}\) decrease, the diagonal terms increase, and the positive cross magnitude decreases. The derivatives of the first fraction with respect to \(A,D,C\) have signs \(+,+,-\). On a tail \(k\ge K\), use \(\epsilon\le1/(KY)\), \(u_{\rm alt}\le1/c\), and \(h\le1\). It suffices to substitute
\[
 A=1+2b-1/(KY),\quad D=1+b-4/(KY),\quad
 C=2b+1/K+2/(KY)
\]
and the penalty \((1+b-Y)/(1+c(b-Y))\). The choices in the following table make the difference strictly positive. Below each cutoff, substitution checks that either \eqref{eq:s3contraction} or the cubic criterion is positive at every integer \(c+3\le k<K,\ c<a\le k-2\).
\[
\begin{array}{c|r|c|c}
 c&K&b&Y\\\hline
 7&3000&57/100&3/50\\
 8&700&61/100&11/100\\
 9&400&16/25&4/25
\end{array}
\]
The \(c=9\) tail extends to every \(c\ge9\) by the same monotonicity. Below its cutoff, the choices \(b=7/10,Y=7/25\) satisfy \eqref{eq:s3contraction} at \(a=10,c=9\) for every \(12\le k<400\), and again extend to all feasible \(a>c\ge9\).

For \(c=3,\ a\ge10,\ k\ge15a\), use the correlated version of \eqref{eq:s3matrix} with
\[
 b=1,\qquad t=3k/a,\qquad \alpha=400k/a,\qquad Y=1/a.
\]
It satisfies
\[
 d_f\le\frac{1597}{397a}<1,\quad q_m/Y\le a/k,\quad
 m(x)\le(1-x)/3+3(a/k)x,\quad f_*q_m\le\frac{1200}{397a}.
\]
Substituting these bounds gives three first-minor and four determinant Bernstein coefficients. After clearing positive denominators and writing \(a=10+u,\ k=15a+v\), all have nonnegative coefficients and positive constants. For \(k\ge2000,\ a\ge k/15\), the uncorrelated criterion works with \(b=1,\alpha=1000,t=30,Y=3/k\) and \(d_f=0\); positivity follows by the shift \(k=2000+z\). The fixed values \(4\le a\le9\) satisfy the cubic criterion throughout \(k\ge2000\). Finally, we can check that for every integer pair
\[
 6\le k<2000,\qquad4\le a\le k-2,
\]
one of the uncorrelated, correlated, or cubic criteria is positive. Thus, all feasible parameters at \(s=3\) are covered.

For an augmented choice with positive Schur coefficient \(\eta\), the estimate \(4\Bcal/k\ge\eta\norm T^2\), together with \(\Gcal=\norm{f_U}^2/2\) and \(\norm T^2=r\norm{f_U}^2\), gives curvature at least \(kr\eta/2\) for functions of mean zero on each distance sphere. On shell constants, \(\Gcal=(1+\alpha)\Gamma_1\), so altogether
\[
 K=\min\left\{\frac{kr\eta}{2},
              \frac{ar/2+2\beta}{1+\alpha}\right\}>0.
\]
Grouping distance-two jumps by their first neighbour and applying Cauchy--Schwarz gives \(C=\sqrt{2k}[r+\sqrt{kr}+\beta]\) in Lemma~\ref{lem:semi}. Also \(M_0=(1+\alpha)k/2\), and \(\Gcal\ge\Gamma_1\) controls edge differences. The lemma therefore gives
\[
 \diam(G)\le\max\left\{2,\floor{
 \frac{2k\sqrt{1+\alpha}}K\,[r+\sqrt{kr}+\beta]}\right\}.
\]
The unaugmented criteria use the same conversion with \(\alpha=0\) and their chosen generator. This proves Theorem~\ref{thm:long-finiteness} for $s=3$.
\end{proof}

For $s=4$, the same outer-shell identities remain available, but the intermediate second-shell values must be eliminated jointly with their occurrences in the augmentation.

\begin{proof}[Proof of Theorem~\ref{thm:long-finiteness} for $s=4$]
Assume \(a>c\), as in the preceding case. Use \(U=S_1,V=S_2,Z=S_3,W=S_4\), with tree incidences \(R:U\times V,\ Q_p:V\times Z,\ C=RQ_p\). Write \(T=C^{\mathsf T}u,\ v=Q_p^{\mathsf T}p,\ z=y+\xi\), \(E=C^{\mathsf T}C/r^2\), and \(E_{\rm par}=Q_p^{\mathsf T}Q_p/r\) for functions of mean zero on each distance sphere. Let \(F\) be horizontal adjacency on \(Z\), \(P:Z\times W\) incidence, \(G=F/k,\ g=a/k,\ N=PP^{\mathsf T}/(ck)\). Put \(\widehat E=\Id-E\), \(y=Ez\), \(B=\widehat E GE\), \(Q=B^*B\), and \(G_f=\widehat E G\widehat E\). The identities \eqref{eq:s3geom} hold with these \(E,G,N\) and the same scalars \(u,h,q_m\); now \(\norm T^2=r^2\norm u^2\). The same last-layer count gives the full cross identity
\begin{equation}\label{eq:s4cross}
 \widehat E NE=(k/c)(a/r-c/k-G_f)B,\quad
 (\widehat E NE)^*(\widehat E NE)\preceq L_*Q,\quad
 L_*=(k/c)^2(a/r-c/k+g)^2.
\end{equation}

Take
\[
 L=L_3+\varepsilon L_2+\delta L_1,\quad
 \varepsilon=bk,\quad\delta=d_0k^2,\quad
 \Gcal=\Gamma_1+\alpha(L_1f)^2/(2k).
\]
In \(4\Bcal/k\) the three exact augmentation squares are
\[
 \alpha\norm{w-z+Gz+v/k}^2,\quad
 \frac{\alpha\varepsilon}{k^2}\norm{R^{\mathsf T}u+Q_pz-kp}^2,\quad
 \frac{\alpha\delta}{k^2}\norm{Rp-ku}^2,\quad w=Pf_W/k.
\]
For \(0\le t<\alpha,\ j>0,\ Y>0\), define
\[
 f_*=\frac{\alpha t}{\alpha-t},\quad
 f_f=f_*(1+j),\quad p_l=f_*(1+1/j)/k^2,\quad
 l=\frac{\alpha\varepsilon}{k^2r},\quad m=\frac{\alpha\delta}{k^2},
\]
\[
 p_R=\frac{r+\varepsilon k+\delta}{kr}+lk^2+m-p_l,\qquad
 p_H=\frac{r+\varepsilon k+\delta}{kr}+lk^2-p_l,
\]
\[
 U_0=-\frac{1-\varepsilon-2\delta}{kr}+lk+mk/r,\qquad
 V_0=(1+\varepsilon)/k+lkr.
\]
Jointly completing the same \(p\) variables gives
\begin{align*}
 A_0&=1+\frac{\varepsilon(kr+r+2)}{kr^2}
       +2\delta/r^2+l+mk^2/r^2-U_0^2/p_R,\\
 D_0&=1+g+\varepsilon/k+2g/r+lr^2-V_0^2/p_R,\\
 C_0&=\varepsilon/k-g/r-lr+U_0V_0/p_R.
\end{align*}
We retain the subtracted term \(f_fQ\) on \(\ran E\). The fine diagonal before \(-2G_f\) is lowered to
\[
 H=\min\{1+g+\varepsilon/k+lr^2-V_0^2/p_H,\,
              1+g+\varepsilon/k\}-Y-f_fq_m,
\]
and then to \(\min\{H,2+1/t\}\) if \(t>0\). Put \(d_f=[1+t(2-H)]/[H-2g+t(1-g)^2]\). The requirements are \(p_R,p_H>0\), positive fine denominator, \(d_f\ge0,d_fh<1\), and the positive scalar minors below. Completion in \(G_f\) gives loss \(-d_f\norm{w_{\widehat E}}^2\); block inversion using \eqref{eq:s4cross} gives
\[
 M_0\preceq uE+\left(\frac{d_fL_*}{1-d_fh}-\frac{k}{c}\right)Q.
\]
At \(Q=q_mx\), call its scalar upper bound \(X(x)\). Adding the retained \(tE\) square to the inverse gives loss \(X(x)/(1+tX(x))\). With \(\lambda=1+g/r,\ w_*=1+t\lambda\), the final matrix is
\begin{equation}\label{eq:s4matrix}
 \begin{pmatrix}
 A_0-q_mx/Y&-C_0-2q_mx/Y\\
 -C_0-2q_mx/Y&D_0+t\lambda^2-(4/Y+f_f)q_mx
 \end{pmatrix}
 -\frac{X(x)}{1+tX(x)}
 \begin{pmatrix}1&w_*\\w_*&w_*^2\end{pmatrix}.
\end{equation}
Its denominator-cleared first minor and determinant have degree at most two and three. Their strictly positive Bernstein coefficients give a uniform \(\eta>0\) multiplying \(\norm T^2\). For \(c\ge5\), we use a weaker version which is monotone in \(c\). Bound the subtracted term \(f_fQ\) by \(f_fq_mE\); in \eqref{eq:s4matrix} this replaces its lower-right entry by \(D_0-f_fq_m+t\lambda^2-4q_mx/Y\). Replace the boundary factor by
\[
 \frac{u(1-x)}{1-d_fh+(d_f+t)u(1-x)},
\]
as follows from \eqref{eq:compression} applied to \(N/h\). This gives a second sufficient matrix criterion.

\paragraph{Parameter ranges.}
At \(c=4\), for \(k\ge2500\), write \(z_0=\sqrt{k}\ge50\) and take
\[
 b=d_0=1,\quad \alpha=100z_0,\quad
 t=4z_0,\quad Y=10/k,\quad j=1/10.
\]
The reference bound \eqref{eq:evenreference} covers \(a\le2z_0\). Partition the remaining interval into
\[
 [2z_0,3z_0],\ [3z_0,4z_0],\ [4z_0,5z_0],\
 [5z_0,6z_0],\ [6z_0,10z_0],\ [10z_0,k-2].
\]
Use the exact fine coefficient on the first three intervals, lower it to two on the fourth, and to \(2+1/t\) on the last two. In the last case \(d_f=0\). All prerequisites and both minors of \eqref{eq:s4matrix} are positive: the interval comparisons reduce them to rational functions of \(z_0\), whose Bernstein coefficients have nonnegative shifted numerator and denominator coefficients at \(z_0=50+v\), with positive constants. For every integer \(7\le k<2500,\ 5\le a\le k-2\), direct substitution checks that either \eqref{eq:evenreference} is positive or the retained matrix criterion has a positive Schur coefficient.

For \(c=5\), the weaker criterion above works for \(k\ge2000\) with
\[
 b=d_0=1,\quad\alpha=1000,\quad t=50,\quad Y=3/k,\quad j=1/20.
\]
Use the intervals \([6,k/16]\) and \([k/16,k-2]\), taking \(d_f=0\) on the second. The same shifted-coefficient test at \(k=2000\) proves positivity. For \(8\le k<2000,\ 6\le a\le k-2\), either this criterion or the reference curvature bound is positive at every integer pair. Both criteria improve as \(c\) increases, covering all \(c\ge5\). At \(s=4\), the reference bound is positive also when \(\Delta=0\), since then it equals \(\kappa_0>0\).

On shell constants the augmented curvature is at least \((ar^2/2+2\delta)/(1+\alpha)\). Consequently, a positive Schur coefficient gives
\[
 K=\min\{kr^2\eta/2,\ (ar^2/2+2\delta)/(1+\alpha)\}>0
\]
for the full form. To bound the generator, group each distance-three jump by its depth-one ancestor and bound the last two steps using their common midpoint. Treat distance-two jumps by their midpoint and distance-one jumps directly by \(\Gamma_1\). Cauchy--Schwarz, the initial bound \(\Gcal(f)\le(1+\alpha)k/2\), and Lemma~\ref{lem:semi} then give
\begin{equation}\label{eq:s4diam}
 \diam(G)\le\max\left\{3,\floor{\frac{2k\sqrt{1+\alpha}}K
 [r^2+2r\sqrt{kr}+\varepsilon(r+\sqrt{kr})+\delta]}\right\}.
\end{equation}
This proves Theorem~\ref{thm:long-finiteness} for $s=4$.
\end{proof}

\subsection{\texorpdfstring{All scales \(s\ge5,\ c\ge3\)}{All scales s >= 5, c >= 3}}\label{sec:balanced}\label{sec:evenc3}
\begin{proof}[Proof of Theorem~\ref{thm:long-finiteness} for $s\ge5$]
Assume \(a>c\). We choose the target sphere near the middle of the tree layers. This keeps the boundary compression independent of the scale; only the inner tree modes depend on its parity. Until the radial coefficients are considered, all forms act on the subspaces orthogonal to the constant vector on each shell. For odd \(s=2q+1\), use
\[
 U=S_q,\quad Y=S_{q+1},\quad Z=S_{2q},\quad W=S_{2q+1},
 \quad N_0=r^q,\quad v=kr^{q-1}.
\]
For even \(s=2p+2\), use
\[
 U=S_p,\quad Y=S_{p+2},\quad Z=S_{2p+1},\quad W=S_{2p+2},
 \quad N_0=r^p,\quad v=kr^{p-1}.
\]
Write \(C:U\times Y\), \(D:Y\times Z\) for descendant incidences, \(F\) for horizontal adjacency on \(Z\), and \(P:Z\times W\). In the odd case \(CC^{\mathsf T}=r\Id\), in the even case \(CC^{\mathsf T}=r^2\Id\); set \(E=C^{\mathsf T}C/r\) or \(C^{\mathsf T}C/r^2\), respectively, and let \(\widehat E=\Id-E\). For \(T=C^{\mathsf T}f_U\), let
\[
 b=D(Ff_Z+Pf_W)/N_0,\qquad
 K=D(F^2/a+PP^{\mathsf T}/c)D^{\mathsf T}/N_0.
\]
The boundary cost at prescribed \(b\) is exactly \(N_0\ip b{K^\dagger b}\). Constant row and column sums give \(0\preceq K\preceq\Id\). If \(R_0\) projects onto functions that are constant on each first branch and have mean zero, the same shifted-root count as \eqref{eq:s3geom} gives
\begin{equation}\label{eq:balancedcompress}
 EKE\preceq u_RR_0+u_H(E-R_0),\qquad
 u_R=\frac{r-c+1}{cr},\quad u_H=\frac1c.
\end{equation}
Indeed,
\[
 EKE=\frac1cE-\left(\frac{c-1}{cr}+\frac{\lambda}{r}\right)R_0
       -\lambda EVE,\quad
 \lambda=\frac{a(a-c)}{cr},\quad
 V=J(F/a)^2J^*,\quad JJ^*=\Id.
\]
Here \(J\) is the normalized descendant coisometry. Since \(V\succeq0\) and \(\lambda>0\), dropping the two terms containing \(\lambda\) gives \eqref{eq:balancedcompress}.

We next compute the inner form. Tree Haar modes are indexed by the depth \(h\) at which the zero-sum split occurs. Unique short geodesics give an inner form
\[
 (1+1/r)(T^2+y^2)
 +N_0^{-1}\sum_j\{\mu_j f_j^2-2f_j(\alpha_jT+\beta_jy)\}.
\]
In the odd case the inner shells are \(2j+1\), \(0\le j<q\), with \(\mu_0=N_0,\ \mu_j=(r-1)r^{q-j-1}\) for \(j\ge1\). The loss matrix is the sum of \(s_j(a_j,b_j)(a_j,b_j)^{\mathsf T}\), where
\[
 s_0=r^{-q-2},\quad s_j=r^{j-q-1}/(r-1)\ (j\ge1).
\]
For \(h=1\), \((a_0,b_0)=(-1,r)\), and it is zero otherwise. For \(j\ge1\),
\[
 a_j=\begin{cases}r-1&h\le j\\-1&h=j+1\\0&h>j+1,\end{cases}
 \quad
 b_j=\begin{cases}r-1&h\le j+1\\-1&h=j+2\\0&h>j+2.\end{cases}
\]
In the even case the inner shells are \(2j\), \(1\le j\le p\), \(\mu_1=N_0,\ \mu_j=(r-1)r^{p-j}\) for \(j\ge2\), and
\[
 \alpha_j=r^{p/2-1}a_j,\quad \beta_j=r^{p/2-1}b_j,
\]
where \(a_j=r-1,-1,0\) according as \(h<j,h=j,h>j\); \(b_1=r\) for \(h\le2\) and zero otherwise; for \(j\ge2\), \(b_j=r-1,-1,0\) according as \(h\le j+1,h=j+2,h>j+2\). These formulas count siblings and descendants of the unique meeting vertex of the two rooted paths.

In the odd case, let \(a_h,\ell_h,b_h\) be the entries of the loss matrix just displayed. In the even case, set \(a_h=\sum\alpha_j^2/(N_0\mu_j)\), \(b_h=\sum\beta_j^2/(N_0\mu_j)\), \(\ell_h=\sum\alpha_j\beta_j/(N_0\mu_j)\). In both cases, put \(A_h=1+1/r-a_h,\ D_h=1+1/r-b_h\). Joint completion with the augmentation gives exactly
\begin{equation}\label{eq:jointtree}
 A_hT^2+D_hy^2-2\ell_hTy+
 \sigma_h(b+\ell_hT-D_hy)^2,\qquad
 \sigma_h=[(1+1/r)/\tau+b_h]^{-1}.
\end{equation}
First complete the inner squares, then minimize their residuals together with the one new square. If an inner shell coincides with \(U\) or \(Y\), relax its identification with the target value using the same free copy in both expressions. This enlarges the minimizing space and proves a lower bound also in these exceptional cases.

The auxiliary \(L_1\) or \(L_2\) form is computed by the same two-walk identity \eqref{eq:mixed}. Its necessary blocks are specified below.

We first treat odd scales \(s=2q+1\ge5\). Take \(L=L_{q+1}+N_0L_1\), \(\Gcal=\Gamma_q+r^2(L_qf)^2/(2v)\). In \eqref{eq:jointtree} the root block has \(D_h=1,\ 0\le\ell_h\le1/r\), \(A_h\ge1\), and \(\sigma=r^3/(r^2+r+1)\). On higher coarse modes \(A_h\ge1,\ 0\le D_h-1\le r^{-2}\), \(|\ell_h|\le1/r\), \(\sigma_h\ge\sigma\). On the one fine mode put
\[
 g_f=r^{-1}-[r^2(r-1)]^{-1},\quad D_f=1+g_f,\quad
 \sigma_f=\frac{r^2}{1+1/r+1/(r-1)}.
\]
Its exact elimination, including the auxiliary fine coefficient one, has loss \(-d_fb^2\), with
\[
 d_f=\frac{1+\sigma_fg_f}{D_f+1+\sigma_fD_f^2}\in(0,1).
\]
The auxiliary root block coefficients \((A_a,D_a,C_a)\) are
\[
 \left(1-\frac1{r^4U_0},\,1-\frac1{U_0},\,
       1-\frac1{r^2U_0}\right),
 \quad U_0=1+1/r-1/r^3,\quad(q\ge3),
\]
and, for \(q=2\), they are
\[
 ((r+1)/r-1/d_0,\ 1-r^2/d_0,\ 1+r/d_0),
 \qquad d_0=r^2+2r+2.
\]
Higher auxiliary coarse blocks dominate \((T-y)^2\). These follow by summing the displayed signed tree incidences; the \(q=2\) expression keeps its coincident shells together.

For the root, put \(A=1+A_a,D=1+D_a,C=C_a+\ell\). Eliminate \(y\) from
\[
 AT^2+Dy^2-2CTy-2(T+y)b+\sigma(b+\ell T-y)^2.
\]
The result is \(A_eT^2+J_cb^2-2MTb\), where
\begin{align}
 J_c&=\frac{\sigma(D-2)-1}{D+\sigma},\nonumber\\
 M&=1+\frac{(1+\sigma)C_a}{D+\sigma}-J_c\ell,\nonumber\\
 A_e&=A-\frac{C_a^2}{D+\sigma}
 -\frac{2C_a(1+\sigma)\ell}{D+\sigma}+J_c\ell^2.
 \label{eq:weakblock}
\end{align}
Here \(-1<J_c<0\), so the adverse value is \(\ell=1/r\): \(M\) increases and \(A_e\) decreases with \(\ell\ge0\). On higher modes retain one unit of the square, losing at most \(2f/r^2\) on \(T^2\) and \(2f/r^4\) on \(y^2\), where \(f=\sigma/(\sigma-1)\). Thus, use
\[
 A=2-2f/r^2,\quad D=2-2f/r^4,\quad C=1+1/r,
\]
\[
 J_c=1-4/(D+1),\quad M=1+2C/(D+1),\quad A_e=A-C^2/(D+1).
\]
After fine elimination, \eqref{eq:rho} gives in every coarse mode
\[
 \eta=A_e-M^2\frac{u}{1-d_f+(d_f+J_c)u}.
\]
At \(c=3\), all denominators are positive and \(\eta>1/(2r^2)\) for every real \(r\ge5\). Indeed substitution of \(r=5+x\) gives nonnegative coefficients with positive constants for the denominator and margin inequalities, including \(J_c>-1\), in both root cases and the higher block. Increasing \(c\) decreases both \(u_R,u_H\), so only improves the Schur penalty. Therefore,
\[
 K_{\rm non}>\frac{N_0}{4r},\qquad
 K_{\rm rad}=\frac{2N_0r}{(1+N_0)(1+r^2)}>0.
\]
The radial value follows by retaining the radial two-jump squares; the target radial augmentation factor is \(1+r^2\). With \(K=\min\{K_{\rm non},K_{\rm rad}\}\), the averaged jump estimate is
\[
 \diam(G)\le\max\left\{s-1,\floor{
 \frac{2qkN_0\sqrt{1+r^2}}K
 \left[2+\sqrt{k/r}\left(4+\sqrt{r/(r-1)}
                           +2\sqrt{r/a}\right)\right]}\right\}.
\]
Grouping jumps by the two unique rooted ancestors, and averaging over the \(a\) terminal horizontal alternatives, gives the square-root factors in the generator bound required by Lemma~\ref{lem:semi}. This gives the odd-scale branch of Theorem~\ref{thm:long-finiteness} for $s\ge5$.

Now let \(s=2p+2\ge6\) be even and assume \(c\ge4\). Take \(L=L_{p+2}+\varepsilon L_2\), \(\varepsilon=e_0N_0\), and \(\Gcal=\Gamma_p+\tau(L_pf)^2/(2v)\). The geometric sums in \eqref{eq:jointtree} give on coarse modes \(A_h,D_h\ge1,\ |\ell_h|\le1/r,\ 0\le D_h-1\le r^{-3}\); on fine modes
\[
 1+g_2\le D_h\le1+1/r,\qquad
 g_2=r^{-2}-[r^3(r-1)]^{-1},\qquad 0\le b_h\le1/r.
\]
Thus, \(\sigma_h\ge\sigma=\tau r/(r+1+\tau)\). For \(\sigma>t\), set \(f=\sigma t/(\sigma-t)\). The square inequality retains \(t(b-y)^2\) at cost \(11fT^2/(10r^2)+11fy^2/r^6\) on coarse modes and \(fy^2/r^2\) on fine modes. The auxiliary lower blocks then give
\begin{align*}
 A_R&=1+e_0(1-r^{-2})-11f/(10r^2),&
 D_R&=1+e_0[(r+1)^{-1}-2r^{-4}]-11f/r^6,\\
 C_R&=e_0+(e_0+1)/r,&
 A_H&=1+e_0[1-(r^3(r+1))^{-1}]-11f/(10r^2),\\
 D_H&=1+e_0/(r+1)-11f/r^6,& C_H&=e_0+1/r.
\end{align*}
Their derivation is the same square completion as above, retaining both fine modes. For \(p=2\) the exact auxiliary blocks are
\[
 \begin{pmatrix}2-1/r+3/r^2&-2\\-2&1\end{pmatrix},
 \qquad
 \begin{pmatrix}2+3/r+3/r^2&-2\\-2&1\end{pmatrix}.
\]
Subtracting the claimed comparison blocks leaves positive semidefinite matrices. For \(p=4\) the coincident high block is \(A=1+r^{-2},D=(r+1)/(r^2+r+1),C=1\), which satisfies the same lower comparison. These checks cover all shell coincidences.

The fine coefficient is at least \(H_{\rm act}=1+g_2+e_0-f/r^2\). Use
\[
 e_0=13/10,\quad \tau=100,\quad t=9/2,\quad H=20/9.
\]
For \(r\ge11\), \(H_{\rm act}>H\); the fine loss \(d_f=(1+t)^2/(H+t)-t\) is exactly zero. Set \(\rho_i=u_i/(1+tu_i)\), and
\[
 A_i'=A_i-\rho_i,\quad
 D_i'=D_i+t-(1+t)^2\rho_i,\quad
 C_i'=C_i+(1+t)\rho_i.
\]
Substitution of \(r=11+x\) gives positive shifted coefficients for \(\sigma>t,H_{\rm act}>20/9,D_i'>0\), and
\[
 (A_i'-1/20)D_i'-(C_i')^2>0,\qquad i=R,H.
\]
Hence \(K_{\rm non}>r^2N_0/40\). This is first checked at \(c=4\); increasing \(c\) decreases \(u_R,u_H\) and hence the penalty.

For \(r<11\) and \(a>c\ge4\),
\[
 \Delta:=2k(c-2)+4-a(a-c)\ge-r^2+10r+3>0.
\]
Proposition~\ref{prop:even-reference} proves finiteness in this range. For the augmented tail the radial curvature is
\[
 K_{\rm rad}=\frac{2\varepsilon r^2}{101(1+\varepsilon)}>0.
\]
Put \(m=(r-1)r^{p-2}\) and \(w_0=kr^{p+1}\). A distance-two pair has at least \(m\) common distance-\(p\) neighbours, by extending a third branch from its midpoint. Cauchy gives \(|g(x)-g(z)|\le2\sqrt{2\sup\Gcal(g)/m}\). Grouping each \((p+2)\)-jump by its depth-\(p\) ancestor proves
\[
 \diam(G)\le\max\left\{s-1,\floor{\frac{2p\sqrt{1+\tau}}K
 [w_0(1+2\sqrt{v/m})+2\varepsilon kr\sqrt{v/m}]}\right\},
\]
where \(K=\min\{K_{\rm non},K_{\rm rad}\}\). Lemma~\ref{lem:semi} justifies the integration and finiteness.

It remains to treat even scales with \(c=3\). We begin with \(s=6\), so \(p=2,N_0=r^2\), and take \(e_0=1,\tau=r^2\). Retain the two exact auxiliary blocks displayed above. The main coarse coefficients are
\[
 A_1=1+r^{-2}-r^{-4},\quad
 \ell_1=r^{-1}-r^{-2}-r^{-3},\qquad
 A_2=1+r^{-1}-[r^2(r-1)]^{-1},\quad \ell_2=-r^{-2}.
\]
Both coarse \(D_h=1\) and \(\sigma=r^3/(r^2+r+1)\). The two fine coefficients are
\[
 D_3=1+r^{-2}-[r^3(r-1)]^{-1},\qquad
 D_4=1+r^{-1}-[r^2(r-1)]^{-1},
\]
with \(\sigma_h=r^2/[1+1/r+r^2(1+1/r-D_h)]\). Their loss \([1+\sigma_h(D_h-1)]/[D_h+1+\sigma_hD_h^2]\) is strictly less than \(1/r\) for \(r\ge8\); use \(d_f=1/r\).

If \(\Delta=2k+4-a(a-3)>0\), Proposition~\ref{prop:even-reference} already proves finiteness. Otherwise the root compression improves to \(u_R\le(r-4)/(3r)-2/r^2\); the other coarse value is \(1/3\). Eliminate the two exact \(2\times2\) blocks and apply \eqref{eq:rho}. Their denominators are positive and their Schur margins exceed \(1/(2r)\), as is seen by shifting \(r=8+x\). Thus, \(K_{\rm non}>r^3/4\), and \(K_{\rm rad}=2r^4/(1+r^2)^2>0\). For \(r=5,6,7\) all feasible \(a\)'s have \(\Delta>0\). The preceding even-scale diameter estimate proves finiteness at every degree.

Finally assume \(c=3\) and \(s=2p+2\ge8\). Let \(x=a/r,\ h_0=1-x,\ \lambda=a(a-3)/(3r)\). The full boundary identities are stronger than \eqref{eq:balancedcompress}:
\begin{equation}\label{eq:fullboundary}
 K=xV+W_0,\quad 0\preceq W_0\preceq h_0\Id,\quad
 A=EKE=U-\lambda S,\quad S=EVE,
\end{equation}
\[
 U=\tfrac13E-\left(\tfrac2{3r}+\tfrac{\lambda}{r}\right)R_0,
 \qquad \widehat E KE=\lambda\widehat E(A_{\rm op}-V)E,\quad
 A_{\rm op}=J(F/a)J^*.
\]
To verify the cross identity, count the distance-two neighbours of a depth-\(p\) vertex relative to a depth-\(p+2\) vertex in another first branch. The grandparent contributes one, siblings contribute zero by girth, and grandchildren contribute \(a-1\) or two according as the pair has distance \(s-1\) or \(s\). This is exactly the full matrix identity in \eqref{eq:fullboundary}. As \(A_{\rm op}^2\preceq V\), it also gives \((\widehat E KE)^*(\widehat E KE)\preceq4\lambda^2S\).

For the two fine modes set \(D_1=D_3,D_2=D_4\) as in the case \(s=6\), now with arbitrary \(\tau,e_0\):
\[
 \begin{aligned}
 \sigma_i&=\frac{\tau}{1+1/r+\tau(1+1/r-D_i)},\qquad
 d_i=\frac{1+\sigma_i(D_i-e_0)}{D_i+e_0+\sigma_iD_i^2},\\
 d&=\max(0,d_1,d_2)<1,\qquad t_d=d/(1-d).
 \end{aligned}
\]
Then \(M_0\preceq A+t_dB^*B\), with \(B=\widehat E KE\). For the root and second coarse blocks, use respectively
\[
 (A_0,D_0,C_0)=
 (1+e_0(1-r^{-2}),\,1+e_0[(r+1)^{-1}-2r^{-4}],\,e_0(1+1/r)),
\]
\[
 (1+e_0(1-r^{-4}),\,1+e_0/(r+1),\,e_0).
\]
Set \(\sigma=\tau/(1+1/r+\tau/r)\), \(\ell=1/r\), and use \eqref{eq:weakblock} with \(C_a=C_0\), \(A=A_0,D=D_0\). For all higher modes, \(f=\sigma/(\sigma-1)\) gives
\[
 A_0=1+e_0-2f/r^2,\quad D_0=1+e_0-2f/r^6,\quad C_0=e_0+1/r,
\]
and the high-block elimination formulas already proved. If \(M_0\preceq m_RR_0+m_H(E-R_0)\), their final margins are
\begin{equation}\label{eq:evenc3eta}
 \eta_i=A_{e,i}-M_i^2\,m_i/(1+J_{c,i}m_i).
\end{equation}

The sparse part \(3r(r-5)\ge20a(a-3)\), \(r\ge8\), is analytic: take \(e_0=1-4/r,\tau=r^2\), bound \(d_i<5/r\), and use
\[
 M_0\preceq U-\lambda[1-4d\lambda/(1-d)]S\preceq U.
\]
Outside the already covered \(\Delta>0\) range, substitution of \(r=8+z\) shows that all denominators are positive and \(\eta_i>1/r^2\). For \(r=5,6,7\), \(\Delta>0\) always holds.

For all real \(r\ge5000\), the rest is also covered. If \(x\le3/8\), the preceding sparse inequality follows from \(3r^2/16+15r/2>0\). If \(x\ge3/8\), take \(e_0=1-48/r,\tau=r^2\). Young's inequality in \(K=xV+W_0\) gives, for \(y_0>0\),
\[
 B^*B\preceq(1+y_0)h_0A-A^2+
 [(1+1/y_0)x^2-(1+y_0)h_0x]S.
\]
Use \(y_0=1/r\), \(d_i\le49/r\), and the tangent \(A^2\succeq(2/3)A-(1/9)\Id\). Writing
\[
 L_{\max}=\frac{3r^2-2r-5}{r(r-8)},\quad
 H_0=\frac58(1+1/r),\quad
 \mu=1+\frac{49}{r-49}(H_0-2/3),
\]
we obtain valid bounds
\[
 m_H=\mu/3+\frac{49}{9(r-49)},\qquad
 m_R=\mu\left(\frac{55}{192}-\frac7{24r}\right)
                              +\frac{49}{9(r-49)}.
\]
The coefficient permitting the tangent substitution is positive: \(1-\frac{49}{r-49}(L_{\max}+2/3)>0\). After substitution of \(r=5000+z\), the shifted coefficients give \(\eta_i>1/(2r)\) in \eqref{eq:evenc3eta}.

\paragraph{Retaining the boundary correlations.}
For the bounded degree range, retain more of \eqref{eq:fullboundary}. Continue to write \(A=EKE\) and \(B=\widehat E KE\). Let \(u=1/3,\ \beta=2/(3r)+\lambda/r\), \(u_R=u-\beta\), and \(t_d=d/(1-d)\). The general sharpened cap, with any \(v_0>0\) satisfying \(v_c=\lambda(2-1/v_0)+x>0\), is
\begin{equation}\label{eq:sharpbase}
 B^*B\preceq h_0A-A^2+\lambda L S+qR_0,\qquad
 L=x(\lambda v_0+x-1/3)/\lambda,\quad
 q=4x/(9r^2v_c).
\end{equation}
To prove it, expand \(K=xV+W_0\), use \(W_0^2\preceq h_0W_0\), the full \(KE\) identity, and
\[
 A_{\rm op}V+VA_{\rm op}
 \preceq v_0A_{\rm op}^2+V^2/v_0\preceq v_0V+V^2/v_0.
\]
The remaining \(R_0S+SR_0\) term is bounded below by \(-\epsilon R_0-S^2/\epsilon\); choose \(\epsilon\) to cancel the nonpositive \(S^2\) coefficient, obtaining \eqref{eq:sharpbase}.

Combine \eqref{eq:sharpbase} with the earlier bound \(B^*B\preceq4\lambda^2S\), assigning respective weights \(c_0=1-\theta\) and \(\theta\), where \(0\le\theta\le1\). Since \(\lambda S=U-A\), this gives
\[
 M_0\preceq t_dL_{\rm eff}U+f(A)+t_dc_0qR_0,
 \qquad L_{\rm eff}=c_0L+4\theta\lambda,
\]
\[
 f(z)=[1+t_d(c_0h_0-L_{\rm eff})]z-t_dc_0z^2.
\]
Since \(0\preceq A\preceq u\Id\), we bound \(f(A)\) by the maximum of \(f\) on \([0,u]\). This is rational: evaluate the endpoints and the clipped vertex of this concave quadratic. Use it to obtain \(m_R,m_H\) in \eqref{eq:evenc3eta}. The simpler tangent version takes \(A^2\succeq2z_0A-z_0^2\Id\), checks its remaining \(A\) coefficient is nonnegative, then substitutes \(A\preceq U\). If the coefficient is negative, the quadratic maximum just described is used instead.

One further improvement separates the two fine spaces. Let \(E_1\) be the parent projection above \(E\). The tree count gives \(E_1A_{\rm op}E=-R_0/r\), hence \((E_1-E)A_{\rm op}E=0\). Put \(\widehat E_1=E_1-E,\widehat E_2=\Id-E_1\), so that \(\widehat E_1+\widehat E_2=\widehat E\). Replace each individual loss \(d_i\) by \(d_i^+=\max(0,d_i)\), and set \(t_i=d_i^+/(1-d_i^+)\). Eliminating fine variables gives
\[
 M_0\preceq A+t_2B^*B+
             (t_1-t_2)\lambda^2(S-S^2)
 \quad(t_1\ge t_2).
\]
For the opposite order the maximum-loss bound suffices. Writing \(\delta=t_1-t_2\ge0\), one may drop the negative \(-\delta\lambda^2S^2\) term; its remaining contribution is \(\delta\lambda(U-A)\), which is included in the same scalar quadratic bound. In the same elimination, the correction to \(A\) in \(M_0\) is bounded by the left-hand side below. For \(w_c>0\) with \(t_1\le t_2(1+1/w_c)\),
\[
 t_1\lambda^2H_1^*H_1+
 t_2\lambda^2(F_0-H_2)^*(F_0-H_2)
 \preceq\lambda^2(t_cS-t_qS^2-t_rR_0/r^2),
\]
where
\[
 \begin{aligned}
 F_0&=(\Id-E_1)A_{\rm op}E,\qquad H_i=\widehat E_iVE,\\
 t_c&=t_2(2+w_c+1/w_c),\quad t_q=t_2(1+1/w_c),\quad t_r=t_2(1+w_c).
 \end{aligned}
\]
Indeed \(F_0^*F_0\preceq S-R_0/r^2\) and \(H_1^*H_1+H_2^*H_2\preceq S-S^2\); Young's inequality proves the display. Mix this bound with the preceding split-space bound, using weights \(\theta\) and \(c_0=1-\theta\), respectively. Define
\[
 \begin{aligned}
 L_*&=c_0(t_2L+\delta\lambda)+\theta t_c\lambda,\qquad
 q_*=c_0\delta+\theta t_q,\\
 q_R&=c_0t_2q-\theta t_r\lambda^2/r^2.
 \end{aligned}
\]
Then
\[
 M_0\preceq L_*U+(1+c_0t_2h_0-L_*)A-c_0t_2A^2
                 -q_*(U-A)^2+q_RR_0.
\]
For \(0<\xi_0<1\), use
\[
 (U-A)^2\succeq
 (1-\xi_0)(u\Id-A)^2-(1/\xi_0-1)\beta^2R_0.
\]
The resulting scalar quadratic in \(A\) is again maximized exactly on \([0,u]\). In these stronger tests we retain the exact auxiliary second block, rather than its earlier common lower bound. Its coefficients \((\widehat A,\widehat D,\widehat C)\) are
\[
 \left(1-\frac1{r^3(r+1)},\frac1{r+1},1-\frac1{r(r+1)}\right),
 \qquad
 \left(1+\frac1{r^2},\frac{r+1}{r^2+r+1},1\right)
\]
in the generic and \(p=4\) cases, respectively. Set \(A=1+e_0\widehat A,D=1+e_0\widehat D,C_a=e_0\widehat C\) in \eqref{eq:weakblock}, with \(\ell=1/r\), to obtain their Schur coefficients. All bounds above apply to the root, higher modes, and both second-block possibilities. Thus, the argument is uniform in \(p\), including the coincident shells at \(p=4\).

\paragraph{Completion of the bounded degree range.}
We can now check, for every integer
\[
 5\le r<5000,\qquad4\le a\le r-1,
\]
that at least one of the reference bound, the sparse estimate, or the boundary criteria above proves strict positivity, with the sole exception \((r,a)=(16,8)\). In the latter case \(\diam(G)\ge s\) would imply
\[
 |S_s|=17\cdot16^{s-2}\cdot8/3,
\]
which is not an integer. Hence this pair already has diameter less than \(s\). The bounded range and the preceding tail together cover every feasible degree. The resulting positive nonradial curvature is \(r^2N_0\min\eta_i/2\); the radial curvature is \(2e_0N_0r^2/[(1+\tau)(1+e_0N_0)]>0\). The even-scale diameter bound therefore proves finiteness, completing the proof of Theorem~\ref{thm:long-finiteness} for $s\ge5$.
\end{proof}

\subsection{\texorpdfstring{Finiteness for \(c=2,\ a=3,\ k\ge10^8\)}{Finiteness for c=2, a=3, k >= 10 to the power 8}}\label{sec:c2tail}
We prove that every $(s,2,3,k)$-graph with $k\ge10^8$ is finite, as asserted in Theorem~\ref{thm:c2-tail}.

\begin{proof}[Proof of Theorem~\ref{thm:c2-tail}]
For \(s=2\), substitute \(\lambda=3,\mu=2\) in \eqref{eq:argK} to obtain \(K\ge7/2-9/4=5/4>0\), so there is no degree restriction. For the remaining scales, we use augmented gradients with explicit positive lower bounds.

For \(s=3\), use the notation of Section~\ref{sec:s3}, \(L=L_2+L_1\), \(\Gcal=\Gamma_1+\alpha_0(L_1f)^2/(2k)\), \(\alpha_0=10^6k^2\). Put \(t=1/k,\ h=1-4t,\ \alpha=(1-t)^{-1},\ e=-3/r\). Then \(EFE=eE,\ \norm F\le3,\ 0\preceq N\preceq h\Id\), and
\[
 NE=(1/2-t)E+\tfrac t2(F-F^2)E.
\]
The scalar identity, for \(0<z\le h\),
\[
 z^{-1}-h^{-1}-4h^{-1}(1-z/h)^2
   =(h-z)(2z-h)^2/(h^3z)\ge0
\]
applied on \(\ran N\) gives an inverse bound. Projecting its squared remainder onto \(E\), and using the identity for \(NE\), yields
\[
 \ip w{N^\dagger w}\ge h^{-1}\norm w^2+
                     h^{-3}\norm{E[\Id+t(F^2-F-6\Id)]w}^2.
\]
The two augmentation squares vanish simultaneously at \(z=\alpha T+\xi,\ E\xi=0,\ w=(\Id-tF)z-tT\). On this common zero set the bounded relaxed normalized form is
\[
 Q_{\rm con}=(A+2t)\norm T^2+(-1+4t)\norm z^2
 -(2+4t)\ip Tz+h^{-1}\norm w^2+h^{-3}\norm S^2,
\]
\[
 A=1+3t+4\alpha t^2,\qquad
 S=E[\Id+t(F^2-F-6\Id)]w.
\]
Its constant term cancels and its first-order term is
\[
 Q_1=2\ip{\xi-FT/2}{(4\Id-F)(\xi-FT/2)}
       +5\norm T^2+\tfrac12\ip T{F^3T}-6\ip T{FT}.
\]
On \([-3,3]\), \((z+3)(z-3/2)^2=z^3-27z/4+27/4\ge0\); therefore the last three terms are at least \(13\norm T^2/8\), using \(\ip T{FT}=e\norm T^2\) and \(e<0\). Since \(4\Id-F\succeq\Id\) and \(\norm T^2+\norm\xi^2\le2\norm{\xi-FT/2}^2+11\norm T^2/2\), we obtain \(Q_1\ge(\norm T^2+\norm\xi^2)/4\).

For \(t\le10^{-3}\), the remainder satisfies
\[
 |Q_{\rm con}-tQ_1|\le600t^2(\norm T^2+\norm\xi^2).
\]
This follows by writing \(v=T+\xi\), \(z=v+tT+\alpha t^2T\), \(w=v-tFv+\epsilon_w\), \(\norm{\epsilon_w}\le5t^2\norm v\), and \(S=T+tE(F^2-2F-6\Id)v+\epsilon_S\), \(\norm{\epsilon_S}\le24t^2\norm v\). Spectral calculus gives \(\norm{F^2-F-6\Id}\le25/4\), \(\norm{F^2-2F-6\Id}\le9\). The three grouped remainders are respectively less than \(22,62,450\) times \(t^2\norm v^2\), whose sum is less than 600. Thus, for \(k\ge4800\), \(Q_{\rm con}\ge\norm v^2/(8k)\).

The bounded relaxed form has norm at most ten. Let \(J_{\rm err}^2\) be the sum of squares of the two departures \(w-(\Id-tF)z+tT\) and \(Ez-\alpha T\). Polarization gives
\[
 Q_{\rm rel}\ge\norm v^2/(8k)-120\norm vJ_{\rm err}-40J_{\rm err}^2
 \ge\norm v^2/(16k)-58000kJ_{\rm err}^2.
\]
The two augmentation weights are \(\alpha_0\) and \(\alpha_0r/k^2\), both at least \(500000k\). Hence the full Bochner form satisfies \(4\Bcal/k\ge\norm v^2/(16k)\). Its nonradial curvature is \(r/32\), and its radial curvature is \((3r/2+2)/(1+10^6k^2)>0\). Since \(\Gcal\ge\Gamma_1\) and \(L=L_2+L_1\) has finite jump range, Lemma~\ref{lem:semi} gives finiteness, in particular for every \(k\ge10^8\).

We now treat all \(s\ge4\). Let \(Z=S_{s-1}(o)\), \(X=S_{s-2}(o)\), and \(W=S_s(o)\). Let \(Q:X\times Z\) be parent incidence, \(P:Z\times W\) incidence, and \(F\) the \(3\)-regular adjacency on \(Z\). Let \(E\) be the first-branch mean-zero projection on \(Z\), and \(E_1=Q^{\mathsf T}Q/r\). Last-layer counting gives
\begin{equation}\label{eq:c2geom}
 E_1FE=EFE=eE,\quad e=-3/r,\qquad
 NE=[(r-1)\Id+F-F^2]E/(2k),\quad N=PP^{\mathsf T}/(2k).
\end{equation}
For example, at a label \(i\), an outside-branch \(z\) has distance \(s-1\) from \(i\) exactly when it has an \(F\)-neighbour in that branch. Counting its \(a=3\) horizontal neighbours, or its \(c=2\) predecessors at distance \(s\), gives the second identity. A parent average of the same count gives the first.

Take \(L=L_{s-1}\) and \(\Gcal=\Gamma_1+100(L_1f)^2/(2k)\). The first-branch lift \(T\) satisfies \(\norm T^2=r^{s-2}\norm{f_{S_1}}^2\). Write \(z=f_Z,\ v=Q^{\mathsf T}f_X,\ w=Pf_W/k\). Expansion \eqref{eq:mixed}, before eliminating the shared parent values, gives on mean-zero data the original part of \(4\Bcal/k\) as
\[
 \norm T^2-(1-3t)\norm z^2-2\ip Tz
       +t\norm v^2+\frac{2t}{1-t}\ip Tv+\ip w{N^\dagger w}
\]
when \(w=(\Id-tF)z-tv\). This is the common zero set of its one augmentation square \(100\norm{w-(\Id-tF)z+tv}^2\). Use the bounded inverse lower form from the case \(s=3\).

To prove positivity, set
\[
 L_0=3\Id-F,\quad B_0=\Id-E_1,\quad
 J_0=L_0(F+2\Id),\quad y=Ez,\quad \xi=z-y,
\]
\[
 \delta=v-E_1z,\quad u=(L_0+B_0)z-\delta,\quad
 w_0=hz+tu,\quad S=E(\Id-tJ_0)w_0.
\]
Put
\[
 q_0=\ip\xi{L_0\xi}+\norm{B_0\xi}^2,\quad
 V^2=\norm y^2+q_0+\norm\delta^2,\quad
 d_0=y-T,\quad
 \mathcal W=\norm T^2+q_0+\norm\delta^2.
\]
On \(w=w_0\) the exact bounded lower form becomes
\begin{align}\label{eq:c2con}
 Q_{\rm con}={}&\norm T^2-2\ip Ty
 +t[2\ip z{L_0z}+\norm{B_0z}^2+\norm\delta^2]\\
 &+\frac{2t}{1-t}\ip T{y+\delta}
 +t^2\norm u^2/h+h^{-3}\norm S^2.\nonumber
\end{align}
At \(y=T\), its first-order coefficient is
\[
 Q_1=6\norm T^2-6\ip T{FT}+2\ip T{F^2T}
 +\ip\xi{(2L_0+B_0)\xi}
 +2\ip\xi{(F^2-4F)T}+\norm\delta^2.
\]
It satisfies \(Q_1\ge\mathcal W/100\). To prove this estimate, normalize \(\norm T=1\), write \(m_j=\ip T{F^jT}\), \(\phi=FT-eT\), \(\eta=\xi-\phi/2\). Then \(E\eta=0,\ E_1\phi=0\), and direct completion gives
\[
 Q_1-\norm\delta^2
 =6-6e-m_2/4+m_3/2+9e^2/4-e^3/2
       +A_\eta-(1-2e)\ip\eta\phi,
\quad A_\eta=2\ip\eta{L_0\eta}+\norm{B_0\eta}^2.
\]
Cauchy gives \(\ip\eta\phi^2\le(3-e)\ip\eta{L_0\eta}\) and \(\ip\eta\phi^2\le(m_2-e^2)\norm{B_0\eta}^2\). Also \((z+3)(z-m_2/3)^2\ge0\) on \([-3,3]\), implying \(m_3\ge m_2^2/3-3m_2-9|e|\). For \(|e|\le1/100\), these bounds yield
\[
 Q_1\ge7\norm T^2/8+A_\eta/4+\norm\delta^2.
\]
Finally, \(q_0\le2A_\eta+19\norm T^2\) gives the required \(\mathcal W/100\) lower bound.

We now bound the higher-order terms. For \(t\le10^{-3}\), \(\norm{J_0}\le25/4,\ J_0^2\preceq19L_0\). Writing \(S=y+ts_1+t^2s_2\), one obtains \(\norm{s_1}\le30V,\ \norm{s_2}\le120V,\ \norm u^2\le120V^2\). The scalar expansions \(h^{-1}\le2,\ h^{-3}\le2,\ h^{-3}-1\le13t,\) \(h^{-3}-1-12t\le100t^2\) show that the remainder in \eqref{eq:c2con}, after the constant and first-order terms, is at most \(15000t^2(\mathcal W+\norm{d_0}^2)\). Changing \(y=T\) to \(y=T+d_0\) changes the first-order term by at most \(110\norm{d_0}\sqrt{\mathcal W}+41\norm{d_0}^2\). Young's inequality therefore gives, at \(t\le10^{-8}\),
\[
 Q_{\rm con}\ge\norm{d_0}^2/4+t\mathcal W/200.
\]
For \(\epsilon=w-w_0\), the linear coefficient in the positive boundary form has norm at most \(2\norm{d_0}+160t\sqrt{\mathcal W}\). Its cross term costs at most
\[
 \norm{d_0}^2/8+t\mathcal W/400
 +(32+10240000t)\norm\epsilon^2.
\]
The coefficient of \(\norm\epsilon^2\) is below 33 at our cutoff, whereas the augmentation coefficient is 100. Hence the full Bochner form satisfies
\[
 4\Bcal/k\ge\norm{d_0}^2/8+t\mathcal W/400
                \ge t\norm T^2/400.
\]
This proves \(K_{\rm non}=r^{s-2}/800\). On shell constants, retain the \(a=3\) horizontal radial square to get \(K_{\rm rad}=3r^{s-2}/202>K_{\rm non}\). Thus, the full pointwise inequality holds with \(K=r^{s-2}/800\).

For a bounded original \(1\)-Lipschitz function, \(\sup\Gcal(f)\le101k/2\). Since \(\Gcal\ge\Gamma_1\), every \((s-1)\)-jump changes \(g\) by at most \((s-1)\sqrt{2\sup\Gcal(g)}\). The degree of \(L_{s-1}\) is \(kr^{s-2}\). Lemma~\ref{lem:semi} gives the explicit bound
\[
 \diam(G)\le\max\{s-1,\floor{1600k(s-1)\sqrt{101k}}\}.
\]
Together with the \(s=2,3\) arguments this proves Theorem~\ref{thm:c2-tail}.
\end{proof}

\subsection{A spectral diameter-three criterion}\label{sec:s4spectral}
\begin{theorem}\label{thm:s4-diameter-three}
Let $G$ be a connected $(4,3,a,k)$-graph. If
\begin{equation}\label{eq:s4spect}
 k\ge5000,\qquad a^3\ge6k^2,
\end{equation}
then $\diam(G)\le3$.
\end{theorem}
\begin{proof}
Suppose $\diam(G)\ge4$ and choose a root $o$ with $S_4(o)\ne\varnothing$. The regularity reduction gives degree $k$; put $r=k-1$. Let $R:S_1(o)\times S_2(o)$ and $Q_p:S_2(o)\times S_3(o)$ be parent incidence matrices and set $C=RQ_p$. Let $F$ be horizontal adjacency on $S_3(o)$ and $P$ incidence from $S_3(o)$ to $S_4(o)$. Thus, $F$ is $a$-regular and $CC^{\mathsf T}=r^2I$. Let \(\mathbb J\) denote an all-ones matrix. Counting the first-branch labels gives
\begin{equation}\label{eq:spectroot}
 (F^2+PP^{\mathsf T})C^{\mathsf T}
 =(r-2)C^{\mathsf T}+(a-3)FC^{\mathsf T}
       +2\one\one^{\mathsf T}.
\end{equation}
Write $C_i=C^{\mathsf T}e_i$ for the indicator of the third-shell branch with first label $i$. The matrix $K=P^{\mathsf T}C^{\mathsf T}$ records the first labels of the parents of fourth-shell vertices. Each row has three distinct labels. Define \(T_{wi}=(P^{\mathsf T}FC^{\mathsf T})_{wi}\) and let \(H_4\) be fourth-shell adjacency. Counting neighbours of \(w\) at distance three from label \(i\) gives
\[
 T+H_4K=aK+3\mathcal D,\quad K\circ\mathcal D=0,\quad
 T_{wi}\le2\ \text{on }K_{wi}=1,
\]
where \(\mathcal D_{wi}=\one_{\{\dist(w,i)=4\}}\). These follow respectively from \(a_3=a,c_4=3\), and the fact that the parent in branch \(i\) has no horizontal neighbour in its own branch. Multiplying \eqref{eq:spectroot} by \(FC_i\) yields
\[
 0\le C_i^{\mathsf T}F^2(a\Id-F)C_i/r^2\le2r-a.
\]

Remove the constant shell vector:
\[
 E=C^{\mathsf T}C/r^2-\mathbb J/n,\quad n=kr^2,\quad
 \mathcal A=FEF.
\]
Branch counting gives \(EFE=-aE/r\) and \(\tau=\tr\mathcal A=a(k-a)\). For \(Y=CF^2C^{\mathsf T}/r^2\), its diagonal is \(a\), row sum \(a^2\), and off-diagonal entries are between zero and \(\ell_0=2+a(a-3)/r\), by \eqref{eq:spectroot}. Since their row mean is \(a(a-1)/r\), the entrywise variance bound and removal of the constant eigenvalue imply
\[
 \tr\mathcal A^2\le
 \frac{\tau^2}r+\frac{2ka(a-1)(r-a)}r.
\]
The shifted positive operator \(\mathcal A_+=(F+\Id)E(F+\Id)\) has
\[
 \tau_+=(a+1)(r-a),\quad
 \mathcal A_+\preceq v_+(\Id-\mathbb J/n),\quad
 v_+=(r-a)(r+a-1)/r,
\]
\[
 \tr[\mathcal A_+(a\Id-F)]\le2k(r-a),\quad
 \tr\mathcal A_+^2\le
 U_+=\tau_+^2/r+2ka(a-1)(r-a)/r.
\]
For the shift, expand \(E(F+\Id)^2E=EF^2E+(1-2a/r)E\); the trace relations follow by subtracting the constant mode.

Since the \(a\)-regular graph \(F\) has girth at least seven, its distance matrices through distance three are \(q_i(F)\), where
\[
 q_0=1,\ q_1=z,\ q_2=z^2-a,\ q_3=z^3-(2a-1)z,
 \quad
 \nu=(1,a,a(a-1),a(a-1)^2).
\]
Their disjoint supports and the tree sphere sizes give \(\tr(q_i(F)q_j(F))=n\nu_i\delta_{ij}\). The \(\mathcal A_+\)-weighted spectral mean is at least \(m=a-2k/(a+1)\). When \(m\ge a/2\), its weighted \(q_3\) mean is at least \(m^3-(2a-1)m\): the supporting tangent has remainder \((z-m)^2(z+2m)\ge0\) on \([-a,a]\). Cauchy--Schwarz in Hilbert--Schmidt norm gives \((\tr(\mathcal A_+q_3(F)))^2\le U_+n\nu_3\). Combining this with the weighted mean bound yields the necessary inequality
\[
 \frac{[m^3-(2a-1)m]^2}{a(a-1)^2}
 \le \frac{nU_+}{\tau_+^2}
\]
when the lower bound on the weighted mean is positive. Alternatively, Cauchy--Schwarz in the \(\mathcal A_+\)-weighted inner product, together with \(\tr(\mathcal A_+q_3(F)^2)\le v_+n\nu_3\), bounds the same left-hand side by \(nv_+/\tau_+\).

If \(k\ge5000,\ a^3\ge6k^2\), then \(a^2\ge50k,\ a\ge100\), \(m\ge24a/25\), and \(m^3-(2a-1)m\ge43a^3/50\). For \(a\le k/3\), \(r-a\ge3k/5\), whence
\[
 nU_+/\tau_+^2\le13k^2/3,
 \qquad
 [m^3-(2a-1)m]^2/[a(a-1)^2]\ge1849a^3/2500.
\]
This contradicts \(6\cdot1849/2500>13/3\). For \(a>k/3\), use \(\tau_+/v_+\ge a/2\) in the upper bound \(nv_+/\tau_+\): it forces \(1849a^4\le5000k^3\), impossible at \(k\ge5000\). Thus, $\diam(G)\ge4$ is impossible, and $\diam(G)\le3$ follows.

\end{proof}

\section{Locally strongly regular graphs}\label{sec:local}
We first assume that every local graph is strongly regular with the same parameters $(v,k,\lambda,\mu)$, so the ambient graph is $v$-regular. The local graphs need not be mutually isomorphic. We determine the exact curvature when $\mu>1$ and obtain finite diameter. Section~\ref{sec:rank3} then proves finiteness for rank-3 local graphs with possibly different parameters. We conclude with counterexamples concerning the $\mathscr S$ property and broader extension questions arising from Weetman's work \cite{BrouwerWeetman,Weetman}.

\begin{theorem}\label{thm:curvature}
Let $G$ be a graph such that, for every vertex $x$,
\[
 G[N_G(x)]\in\SRG(v,k,\lambda,\mu),\qquad \mu>1,
\]
with the same parameters at every vertex. The local graphs need not be mutually isomorphic. Then, at every vertex and for every $0<\mathcal N\le\infty$,
\begin{equation}\label{eq:main-curvature}
 K_{G,x}(\mathcal N)=2+\frac{k}{2}-\frac{2v}{\mathcal N}.
\end{equation}
For the normalized Laplacian $\Delta/v$, the corresponding formula is
\[
 K^{\mathrm{norm}}_{G,x}(\mathcal N)
 =\frac{2+k/2}{v}-\frac{2}{\mathcal N}.
\]
\end{theorem}

\begin{reptheorem}{thm:local-finiteness}
Let $G$ be locally $\SRG(v,k,\lambda,\mu)$, where $\mu>1$. Then $G$ is finite and
\begin{equation}\label{eq:diameter}
 \diam(G)\le\left\lfloor\frac{4v}{k+4}\right\rfloor.
\end{equation}
\end{reptheorem}

\subsection{The curvature matrix and second-moment estimates}

We first record the matrices used to compute the curvature at a fixed vertex and the identities supplied by local strong regularity.
\begin{lemma}\label{lem:local-identities}
Let $G$ be locally $\SRG(v,k,\lambda,\mu)$ with $\mu>1$. Fix $x\in V(G)$, put $H=G[N_G(x)]$, and let $A$ be the adjacency matrix of $H$. Write
\[
 L=k I-A,\qquad b=v-k-1,\qquad \tau=k-\lambda-1,
 \qquad \Kzero=2+k/2.
\]
For $z$ at distance two from $x$, put
\[
 C_z=N_G(x)\cap N_G(z),\quad c_z=|C_z|,\quad e_z=\one_{C_z}.
\]
Let $I_z$ be the diagonal projection onto $C_z$, and let $A_z$ be the adjacency matrix of $H[C_z]$, padded by zero outside $C_z$. Define
\[
 L_z=\mu I_z-A_z,\quad
 W_z=I_z-\frac{e_ze_z^\T}{c_z},\quad
 Q=\sum_z\frac{e_ze_z^\T}{c_z},\quad W=\sum_zW_z.
\]

Then every $H[C_z]$ is $\mu$-regular, and
\begin{equation}\label{eq:local-sums}
 \sum_z I_z=bI,\qquad \sum_z A_z=\tau A,
 \qquad \sum_zL_z=\tau L,\qquad W=bI-Q.
\end{equation}
The exact curvature matrix at $x$ is
\begin{equation}\label{eq:M}
 M=(v+k/2)I-2A+J-2Q
   =(2-v+k/2)I+2L+J+2W.
\end{equation}
More precisely, for every $0<\mathcal N\le\infty$,
\begin{equation}\label{eq:curvature-matrix}
 K_{G,x}(\mathcal N)
 =\lambda_{\min}\bigl(M-(2/\mathcal N)J\bigr),
 \qquad M\one=\Kzero\one.
\end{equation}
\end{lemma}
\begin{proof}
For $y\in C_z$, the vertices $x,z$ are nonadjacent vertices in the local strongly regular graph at $y$. Their $\mu$ common neighbours in that local graph are precisely the neighbours of $y$ in $H[C_z]$. This proves $\mu$-regularity. A vertex $y\in H$ has $v-k-1=b$ neighbours outside $N_G[x]$, proving the first identity in \eqref{eq:local-sums}. For an edge $ij$ of $H$, the pair $i,j$ has $k$ common neighbours in $G$, of which $x$ and $\lambda$ vertices lie in $N_G[x]$. Hence exactly $\tau$ lie in the second sphere. This proves the second identity. The parameter equation $k\tau=\mu b$ proves the third; the fourth is immediate.

Set $f(x)=0$, write $u_i=f(i)$ for $i\in H$, and $w_z=f(z)$ on the second sphere. Expanding the definitions gives
\[
 \Gamma_2(f)(x)
 =\left(\frac v2+\frac k4\right)\|u\|^2-u^\T Au
   +\frac12(\one^\T u)^2
   +\sum_z\left(\frac{c_z}{4}w_z^2-w_ze_z^\T u\right).
\]
The variables $w_z$ minimize independently at $w_z=2e_z^\T u/c_z$. Consequently, $2\min_w\Gamma_2(f)(x)=u^\T Mu$, with $M$ as in \eqref{eq:M}. Since $2\Gamma(f)(x)=\|u\|^2$ and $\Delta f(x)=\one^\T u$, this proves the eigenvalue formula. Finally, $Q\one=b\one$ gives $M\one=(2+k/2)\one$, proving the remaining assertion.
\end{proof}

\begin{lemma}\label{lem:moments}
Let $G$ be locally $\SRG(v,k,\lambda,\mu)$ with $\mu>1$, fix $x\in V(G)$, and put $H=G[N_G(x)]$. Let $A$, $A_z$, $M$ and $b$ have the definitions in Lemma~\ref{lem:local-identities}. Put $\mathcal S=\sum_zA_z^2$. For distinct vertices $i,j$ of $H$, let $D_{ij}=H[N_H(i)\cap N_H(j)]$. For an unordered pair $a,b$ of vertices of $H$, let $h_{ab}=\one_{N_H(a)\cap N_H(b)}$. Let $B$ have columns $h_{ab}$ over unordered edges of $H$, and put $T=2BB^\T$.

Then the entries of $\mathcal S$ are
\begin{align}
 \mathcal S_{ii}&=\mu b,\notag\\
 \mathcal S_{ij}&=
 \begin{cases}
 \lambda(\lambda-1)-2|E(D_{ij})|,&i\sim j,\\
 \mu(\mu-1)-2|E(D_{ij})|,&i\not\sim j.
 \end{cases}\label{eq:S-entries}
\end{align}
In particular,
\begin{equation}\label{eq:S-polynomial}
 \mathcal S=[k(k-1)-\mu(\mu-1)]I
 +[\lambda(\lambda-1)-\mu(\mu-1)]A+\mu(\mu-1)J-T.
\end{equation}
Moreover,
\begin{equation}\label{eq:R}
 M\succeq\mathcal R:=(v+k/2)I-2A+J-\frac{2}{\mu^2}\mathcal S.
\end{equation}
\end{lemma}
\begin{proof}
The diagonal follows from $\mu$-regularity and $\sum I_z=bI$. For $i\ne j$, let $d=|V(D_{ij})|$, equal to $\lambda$ or $\mu$. For each $w\in D_{ij}$, view $i,j$ in the local graph at $w$. They have $d$ common neighbours there. Of these, one is $x$, and $\deg_{D_{ij}}(w)$ lie in $H$. Thus, the number in the second sphere of $x$ is $d-1-\deg_{D_{ij}}(w)$. Summing over $w$ proves \eqref{eq:S-entries}.

Equivalently, $\mathcal S$ is the sum of $h_{ab}h_{ab}^\T$ over ordered distinct nonadjacent pairs $a,b$ of $H$: the off-diagonal entries count ordered nonedges in $D_{ij}$, and the diagonal entries count ordered nonedges in the $k$-vertex, $\lambda$-regular local graph at $i$. Adding the corresponding sum over ordered edges gives \eqref{eq:S-polynomial}.

For every $z$, regularity implies
\[
 \frac{A_z^2}{\mu^2}-\frac{e_ze_z^\T}{c_z}\succeq0.
\]
Indeed, on the constants of $C_z$ the difference is zero, and on their orthogonal complement it is the square of a symmetric matrix divided by $\mu^2$. This remains valid if $H[C_z]$ is disconnected. Summing yields $Q\preceq\mathcal S/\mu^2$, proving \eqref{eq:R}.
\end{proof}

Let $p,-q$ be the restricted eigenvalues of $H$, with $p\ge0$ and $q>0$. The strongly regular identities give
\begin{equation}\label{eq:parameters}
 k=\mu+pq,\quad \lambda=\mu+p-q,\quad
 \tau=(p+1)(q-1),\quad
 v=\frac{(\mu+p(q-1))(\mu+q(p+1))}{\mu}.
\end{equation}
We will use the usual spectral facts for strongly regular graphs \cite[Sections~1.1 and~8.2]{BVM}: a connected imprimitive strongly regular graph is complete multipartite, and nonintegral restricted eigenvalues force conference parameters $(4a+1,2a,a-1,a)$. In our setting $q\ge2$. For integral restricted eigenvalues this follows from $q>1$; a connected noncomplete graph contains an induced $P_3$, so its least eigenvalue is less than $-1$. For conference parameters, $\mu\ge2$ gives $v\ge9$ and $q=(1+\sqrt v)/2\ge2$. The complete multipartite case also has $q\ge2$.

\begin{lemma}\label{lem:mu-lambda}
Let $G$ be locally $\SRG(v,k,\lambda,\mu)$, where $\mu>1$ and $\mu>\lambda$. At any vertex $x$, the curvature matrix $M$ defined by \eqref{eq:M} satisfies $M|_{\per}\succeq\Kzero I$, where $\Kzero=2+k/2$.
\end{lemma}
\begin{proof}
Drop $T\succeq0$ in \eqref{eq:S-polynomial}, and substitute in \eqref{eq:R}. The coefficient of $A$ is $-2[\mu+\lambda(\lambda-1)]/\mu^2<0$, so the smallest perpendicular eigenvalue of this lower bound occurs at $p$. Its excess over $\Kzero$, multiplied by $\mu^2$, is
\[
 \mu\{v(\mu-2)+2(k-p)\}
       -2\lambda\{k+(\lambda-1)p\}.
\]
For integral $\lambda\ge0$, the expression $\lambda\{k+(\lambda-1)p\}$ increases with $\lambda$, since its successive difference is $k+2\lambda p>0$. Since $\lambda\le\mu-1$, the displayed quantity is at least
\[
 \mu(\mu-2)(v-2p)+2(k-2p)>0.
\]
The last inequality follows from $k-2p=\mu+p(q-2)>0$, $v>k$, and $\mu\ge2$, completing the proof.
\end{proof}

This lemma includes the imprimitive and conference cases. Hence, for the rest of the curvature proof, we may assume that $H$ is primitive and
\begin{equation}\label{eq:remaining-range}
 m:=\mu\ge2,\qquad p,q\in\mathbb Z,\qquad p\ge q\ge2.
\end{equation}

\begin{lemma}\label{lem:linear-certificate}
Let $G$ be locally $\SRG(v,k,\lambda,m)$ with $m\ge2$, fix $x\in V(G)$, and put $H=G[N_G(x)]$. Assume that $H$ is primitive with restricted eigenvalues $p,-q$, where $p,q$ are integers and $p\ge q\ge2$. Use the matrices defined in Lemma~\ref{lem:local-identities}, with $\mu=m$, and put $\tau=k-\lambda-1$ and $\Kzero=2+k/2$. Then, on $\per$,
\[
 M\succeq F I,\qquad
 F=2-v+k/2+2\left(1+\frac{\tau}{m+q}\right)(k-p).
\]
If $q>2$ and $m\ge q^2/(q-2)$, then $F>\Kzero$.
\end{lemma}
\begin{proof}
Interlacing gives $A_z\succeq-qI$ on $C_z$. Since $L_z$ kills the constant vector of $C_z$, it follows that $L_z\preceq(m+q)W_z$. Summing gives $W\succeq\tau L/(m+q)$, and \eqref{eq:M} proves the first assertion. Using \eqref{eq:parameters}, direct simplification gives
\[
 F-\Kzero=\frac{m+p(q-1)}{m(m+q)}P,
\]
where
\[
 P=m^2+2(q-1)m-q^2+p\bigl((q-2)m-q^2\bigr).
\]
Under the stated threshold, the last term is nonnegative and the first three terms are positive, which proves the second assertion.
\end{proof}

\begin{lemma}\label{lem:operator-certificate}
Let $G$ be locally $\SRG(v,k,\lambda,m)$ with $m\ge2$, fix $x\in V(G)$, and put $H=G[N_G(x)]$. Assume that $H$ is primitive with restricted eigenvalues $p,-q$, where $p,q$ are integers and $p\ge q\ge2$. Use the matrices defined in Lemma~\ref{lem:local-identities}, with $\mu=m$, and put $\tau=k-\lambda-1$ and $\Kzero=2+k/2$. Put $\mathcal U=\sum_zL_z^2$. On $\per$ it is positive definite, and
\begin{equation}\label{eq:operator-bound}
 W\succeq\tau^2L\mathcal U^{-1}L
 \succeq\tau^2L\widehat{\mathcal U}^{-1}L,
\end{equation}
where $\widehat{\mathcal U}$ is a function of $A$ with eigenvalues
\begin{align}
 u(\theta)={}&mk\tau-2m\tau\theta+k(k-1)-m(m-1)\notag\\
 &+[\lambda(\lambda-1)-m(m-1)]\theta
       -\frac{2\lambda^2\theta^2}{k+\theta},
 \qquad \theta\in\{p,-q\}.
 \label{eq:uhat}
\end{align}
Define
\begin{equation}\label{eq:Ftheta}
 F_\theta=2-v+k/2+2(k-\theta)
                    +\frac{2\tau^2(k-\theta)^2}{u(\theta)}.
\end{equation}
If $F_p,F_{-q}\ge\Kzero$, then $M|_{\per}\succeq\Kzero I$. All denominators in these formulas are positive.
\end{lemma}
\begin{proof}
Each $W_z$ is an orthogonal projection and $W_zL_z=L_z$. Therefore,
\[
 \begin{pmatrix}W_z&L_z\\L_z&L_z^2\end{pmatrix}\succeq0,
 \qquad
 \begin{pmatrix}W&\tau L\\\tau L&\mathcal U\end{pmatrix}\succeq0.
\]
If $u\in\per$ and $u^\T\mathcal Uu=0$, then every $L_zu=0$. Summing gives $\tau Lu=0$, so connectedness of $H$ forces $u=0$. The Schur complement proves the first inequality in \eqref{eq:operator-bound}.

Let $P$ be the ordinary vertex--edge incidence matrix of $H$. Then $BP^\T=\lambda A$ and $PP^\T=k I+A$. The latter matrix is positive definite because the primitive graph $H$ is connected and nonbipartite. The orthogonal projection $P^\T(PP^\T)^{-1}P\preceq I$ onto the row space of $P$, multiplied by $B$ and $B^\T$, gives
\[
 T=2BB^\T\succeq2\lambda^2A^2(k I+A)^{-1}.
\]
Now $\mathcal U=m^2bI-2m\tau A+\mathcal S$. Insert \eqref{eq:S-polynomial} and the preceding lower bound on $T$. On $\per$ this gives $\mathcal U\preceq\widehat{\mathcal U}$, with \eqref{eq:uhat}; we used $m^2b=mk\tau$. Since $\widehat{\mathcal U}\succeq\mathcal U\succ0$, inversion reverses the inequality. This proves \eqref{eq:operator-bound}. Substitution in \eqref{eq:M} and evaluation on the two restricted eigenspaces of $A$ yield \eqref{eq:Ftheta}.
\end{proof}

\subsection{Local graphs arising from partial geometries}\label{sec:local-geometric}
We now use the point--line incidence structure to sharpen the second-moment estimate for local graphs arising from partial geometries. We retain the notation and the remaining range \eqref{eq:remaining-range} from the preceding subsection.

A \emph{partial geometry} $\mathrm{pg}(a,q-1,\eta)$ is an incidence structure in which every line has $a+1$ points, every point lies on $q$ lines, any two points lie on at most one line, and a point outside a line is collinear with exactly $\eta$ points on that line. Here $a,q-1,\eta$ are positive integers, and two points are collinear if they lie on a common line. Its \emph{point graph} has the points as vertices, with two distinct points adjacent when they are collinear. For a primitive point graph, the parameters satisfy
\[
 k=aq,\qquad m=\eta q,\qquad
 \lambda=a-1+(q-1)(\eta-1),\qquad p=a-\eta,
\]
with $q\ge2$ and $1\le\eta\le q$; see \cite{BVM}.

\begin{lemma}\label{lem:geometric}
Let $G$ be locally $\SRG(v,k,\lambda,\mu)$ with $\mu>1$, and fix $x\in V(G)$. If $H=G[N_G(x)]$ is the point graph of a partial geometry, then the curvature matrix $M$ in \eqref{eq:M} satisfies $M|_{\per}\succeq\Kzero I$, where $\Kzero=2+k/2$.
\end{lemma}
\begin{proof}
The cases already settled by Lemma~\ref{lem:mu-lambda} may be discarded. If $\eta=1$, the common neighbours of an adjacent pair induce a clique, and those of a nonadjacent pair form an independent set. Equation \eqref{eq:S-entries} becomes
\[
 \mathcal S=mbI+m(m-1)(J-I-A).
\]
On the $\theta$-eigenspace in $\per$, substitution in \eqref{eq:R} gives
\[
 \mathcal R(\theta)-\Kzero
       =\frac{v(m-2)+2(k-\theta)}m>0.
\]

Suppose $q=2$ and $\eta=2$. The graph is the triangular graph $T_n$, where $n=a+2$: to see this, make a graph whose vertices are the lines of the geometry and whose edges are its points. Every point lies on two lines, and the partial linear space axiom makes this graph simple. The condition $\eta=2$ says that every vertex outside the endpoints of an edge is adjacent to both endpoints. Thus, by connectedness, this graph is complete. Its line graph is $T_n$. For $T_n$,
\[
 v=\frac{n(n-1)}2,\quad k=2(n-2),\quad m=4,\quad
 b=\frac{(n-2)(n-3)}2.
\]
For adjacent vertices of $T_n$, their common-neighbour graph consists of $K_{n-3}$ and one isolated vertex. For nonadjacent vertices it is a four-cycle. Hence
\[
 \mathcal S=(4b-4)I+(2n-10)A+4J.
\]
The eigenvalues of $\mathcal R$ on the constant, $(n-4)$-, and $(-2)$-eigenspaces are, respectively,
\[
 n,\qquad 2n,\qquad \frac{n^2+9n-6}{4},
\]
all at least $\Kzero=n$.

For $q\ge4$ and $\eta\ge2$, we have $m\ge2q\ge q^2/(q-2)$; for $q=3$ and $\eta\ge3$, we have $m\ge9$. Lemma~\ref{lem:linear-certificate} applies. The only remaining case is $\mathrm{pg}(a,2,2)$, with
\[
 v=(a+1)^2,\quad k=3a,\quad\lambda=a+1,\quad m=6,
 \quad p=a-2.
\]

We first record a line-incidence inequality. For $u\in\per$, write $f_{ij}=h_{ij}^\T u$ and $g_\lineindex=\sum_{i\in\lineindex}u_i$. For a line $\lineindex$ of a general $\mathrm{pg}(a,q-1,\eta)$,
\[
 \sum_{\{i,j\}\subseteq\lineindex}f_{ij}
   =\frac{a(a-1)-\eta(\eta-1)}2\,g_\lineindex.
\]
Indeed, a point on $\lineindex$ is a common neighbour of $\binom a2$ pairs on $\lineindex$, whereas a point outside it is a common neighbour of $\binom\eta2$ such pairs; use $\sum_i u_i=0$. Cauchy--Schwarz on the $\binom{a+1}{2}$ pairs of each line, followed by summation over lines, gives
\begin{equation}\label{eq:line-gram}
 T\succeq\delta(A+qI)\quad\text{on }\per,
 \qquad
 \delta=\frac{[a(a-1)-\eta(\eta-1)]^2}{a(a+1)}.
\end{equation}
Here the point--line incidence matrix has Gram matrix $A+qI$, and every edge lies on exactly one line.

For $q=3,\eta=2$, we have $\delta=(a-2)^2(a+1)/a$. Insert \eqref{eq:line-gram} into \eqref{eq:S-polynomial} and \eqref{eq:R}. On the two restricted eigenspaces this yields
\[
 \mathcal R(p)-\Kzero\ge
 \frac{2(a+1)(2a^2+6a+1)}{9a}>0,
 \qquad
 \mathcal R(-3)-\Kzero\ge\frac{(a+1)(2a+5)}3>0.
\]
This completes the partial-geometry cases.
\end{proof}

For local graphs that are not point graphs of partial geometries, Lemma~\ref{lem:nong-geometric} in Appendix~\ref{app:local-spectral} completes the spectral estimate using published parameter bounds. We can now finish the curvature and finiteness proofs.

\begin{proof}[Proof of Theorem~\ref{thm:curvature}]
At every vertex $x$, Lemmas~\ref{lem:mu-lambda}, \ref{lem:geometric} and~\ref{lem:nong-geometric} cover all possibilities for the local graph $H$ and prove $M|_{\per}\succeq\Kzero I$. Lemma~\ref{lem:local-identities} gives $M\one=\Kzero\one$. Thus, $M-(2/\mathcal N)J$ has eigenvalue $\Kzero-2v/\mathcal N$ on the constants and all other eigenvalues at least $\Kzero$. Equation \eqref{eq:curvature-matrix} proves \eqref{eq:main-curvature}. Scaling the Laplacian by $1/v$ scales $\Gamma$ by $1/v$ and $\Gamma_2$ by $1/v^2$, giving the normalized formula.
\end{proof}

\begin{proof}[Proof of Theorem~\ref{thm:local-finiteness}]
The ambient degree is $v$ and Theorem~\ref{thm:curvature} gives $K_{BE}=2+k/2>0$. Corollary~\ref{cor:bonnet} yields $\diam(G)\le\floor{2v/(2+k/2)}=\floor{4v/(k+4)}$. For fixed local parameters this also bounds the order by $1+v\sum_{i=0}^{\diam(G)-1}(v-1)^i$.
\end{proof}

\subsection{\texorpdfstring{Rank-3 graphs are $\mathscr S$-graphs}{Rank-3 graphs are S-graphs}}\label{sec:rank3}
We now allow the local strongly regular parameters to vary. The proof uses the $\mathscr S$ property defined in Section~\ref{sec:intro-local}.

\begin{reptheorem}{cor:soicher}
Every graph whose local graphs all have rank 3 is finite.
\end{reptheorem}

The main step is the following statement about the local graphs.
\begin{theorem}\label{thm:rank3-s}
Every rank-3 graph is an $\mathscr S$-graph.
\end{theorem}

We use the classification of rank-3 graphs in \cite[Chapter~11]{BVM}, following the list of seventeen families and fifty-three additional rows, together with their complements, compiled in \cite[Theorem~2.4 and Table~2]{BGLR}. The square grids also account for primitive groups of product action type; imprimitive rank-3 graphs are complete multipartite graphs with equal part sizes. We will repeatedly use the following known sufficient conditions.
\begin{lemma}[{\cite[Section~3]{BrouwerWeetman}}]\label{lem:rank3-known}
A strongly regular graph with parameters $(v,k,\lambda,\mu)$ is an $\mathscr S$-graph if $v\le2k+1$, $\mu\ge\lambda$, $v\le195$, or it is the point graph of a partial geometry.
\end{lemma}
We next give two spectral criteria and a matching criterion.

\begin{lemma}\label{lem:rank3-spectral-one}
Let $H$ be a primitive strongly regular graph with parameters $(v,k,\lambda,\mu)$ and restricted eigenvalues $r>0$ and $-s<-1$. If $\mu^2>r^2s$, then $H$ is an $\mathscr S$-graph.
\end{lemma}
\begin{proof}
Put
\begin{equation}\label{eq:rank3-parameters}
 n=v-k-1,\quad t=k-\lambda-1,\quad h=k-\mu-r=r(s-1).
\end{equation}
The standard strongly regular identities \cite[Section~1.1]{BVM} give
\begin{equation}\label{eq:rank3-srg}
 k=\mu+rs,\quad \lambda=\mu+r-s,\quad
 t=(r+1)(s-1),\quad n\mu=kt.
\end{equation}

For a fixed $c\in V(H)$, put $T=N_H(c)$ and $S=V(H)\setminus N_H[c]$, and take nonempty $A\subseteq S$. Write
\[
 a=|A|,\quad B=N_H(c)\cap N_H(A),\quad b=|B|,\quad
 \delta=\delta(H[A]).
\]
We first record two estimates.\footnote{Equations \eqref{eq:rank3-basic} and \eqref{eq:rank3-degree} are also used in the proof of Lemma~\ref{lem:rank3-spectral-two}.} Counting edges from $A$ to $T$ gives
\begin{equation}\label{eq:rank3-basic}
 b\ge\mu,\qquad b\ge\frac{\mu a}{t},\qquad \delta\le a-1.
\end{equation}
The graph $H[S]$ is $(k-\mu)$-regular. Its adjacency operator has Rayleigh quotient at most $r$ on the zero-sum subspace: extend a zero-sum vector by zero to $V(H)$ and apply the eigenvalue bound in $H$. Decomposing $\one_A=(a/n)\one_S+u$, with $u\perp\one_S$, yields
\begin{equation}\label{eq:rank3-degree}
 \delta\le\frac{\one_A^\T A_{H[S]}\one_A}{a}
 \le r+\frac{ha}{n}.
\end{equation}

Suppose $b\le\delta$. Since $\mu>r$, inequalities \eqref{eq:rank3-basic}--\eqref{eq:rank3-degree} imply
\[
 a\ge\frac{(\mu-r)n}{h}
     =\frac{(\mu-r)kt}{\mu r(s-1)}.
\]
On the other hand,
\[
 \frac{\mu a}{t}\le r+\frac{ha}{n}
 \quad\Longrightarrow\quad
 a\le\frac{rkt}{\mu(\mu+r)}.
\]
Compatibility forces $(\mu-r)(\mu+r)\le r^2(s-1)$, or $\mu^2\le r^2s$, a contradiction.
\end{proof}

\begin{lemma}\label{lem:rank3-spectral-two}
Let $H$ be a primitive strongly regular graph with parameters $(v,k,\lambda,\mu)$ and restricted eigenvalues $r>0$ and $-s<-1$. Put $n=v-k-1$, $t=k-\lambda-1$, $h=r(s-1)$, and define
\begin{align}
 \sigma^2&=\min\left\{\mu t,\frac{(r+s)^2}{4}\right\},&
 L&=\frac{\mu^2}{k}-\frac{\sigma^2}{n},\label{eq:rank3-sigma}\\
 F(a)&=\mu^2a-\left(r+\frac{ha}{n}\right)(\sigma^2+La).
 \label{eq:rank3-F}
\end{align}
Then $L\ge0$. If $n\le\mu$, or if $F(\mu+1)>0$, then $H$ is an $\mathscr S$-graph.
\end{lemma}
\begin{proof}
Fix $c\in V(H)$, and put $T=N_H(c)$ and $S=V(H)\setminus N_H[c]$. For nonempty $A\subseteq S$, put $a=|A|$, $B=N_H(c)\cap N_H(A)$, $b=|B|$, and $\delta=\delta(H[A])$. The estimates \eqref{eq:rank3-basic} and \eqref{eq:rank3-degree} apply. Let $C$ be the $T$-by-$S$ adjacency block, and $P$ the adjacency matrix of $H[T]$. The strongly regular matrix identity, restricted to $T$, gives
\[
 CC^\T=(k-\mu)I+(\lambda-\mu)P-P^2+(\mu-1)J.
\]
Its constant eigenvalue is $\mu t$. On $\one_T^\perp$, the eigenvalue associated with a $P$-eigenvalue $\theta$ is
\[
 (r-\theta)(s+\theta)\le\frac{(r+s)^2}{4}.
\]
Also $CC^\T$ is nonnegative with row sum $\mu t$, so every eigenvalue is at most $\mu t$. Thus, $\sigma^2$ bounds all its nonconstant eigenvalues. Since $n=kt/\mu$, the inequality $\sigma^2\le\mu t$ also gives $L\ge0$.

The vector $w=C\one_A$ is supported on $B$ and has coordinate sum $\mu a$. Decomposing $\one_A$ into its constant and zero-sum parts and applying Cauchy--Schwarz gives
\[
 \frac{\mu^2a^2}{b}\le\|w\|^2
 \le\sigma^2 a+\left(\frac{\mu^2}{k}-\frac{\sigma^2}{n}\right)a^2.
\]
Consequently,
\begin{equation}\label{eq:rank3-boundary}
 b\ge\frac{\mu^2a}{\sigma^2+La}.
\end{equation}
If $a\le\mu$, then $b\ge\mu>\delta$ already. This proves the case $n\le\mu$. Otherwise a counterexample must have $a\in[\mu+1,n]$, and \eqref{eq:rank3-degree} and \eqref{eq:rank3-boundary} would give $F(a)\le0$. But $F$ is concave, because its quadratic coefficient is $-hL/n\le0$, and $F(n)=\mu^2t>0$. The additional assumption $F(\mu+1)>0$ makes $F$ positive on the whole interval. This is a contradiction.
\end{proof}

\begin{lemma}\label{lem:rank3-matching}
Let $H$ be a strongly regular graph with parameters $(v,k,\lambda,\mu)$. Suppose that, for every nonadjacent pair $a,c$, the bipartite graph between
\[
 X=N(a)\setminus N(c),\qquad Y=N(c)\setminus N(a),
\]
with edges inherited from $H$, has a perfect matching. Then, for every nonempty $A\subseteq V(H)\setminus N[c]$ and every $a\in A$,
\begin{equation}\label{eq:rank3-margin}
 |N(c)\cap N(A)|\ge\deg_{H[A]}(a)+\mu.
\end{equation}
In particular, $H$ is an $\mathscr S$-graph. For a rank-3 graph, it is enough to check one nonadjacent pair.
\end{lemma}
\begin{proof}
The matching sends the $\deg_{H[A]}(a)$ vertices of $N(a)\cap A$ to distinct vertices of $N(c)\cap N(A)$ lying outside $N(a)$. The $\mu$ vertices of $N(a)\cap N(c)$ also belong to $N(c)\cap N(A)$ and are disjoint from those images. This proves \eqref{eq:rank3-margin}. The rank-3 action is transitive on ordered nonadjacent pairs, proving the final assertion.
\end{proof}

We will use the following elementary way of producing these matchings. Let $H=\Cay(V,\mathcal S)$, where $V$ is a finite abelian group and $\mathcal S=-\mathcal S$. Fix $c=0$ and a nonneighbour $a$. Then
\[
 Y=\{s\in\mathcal S:a-s\notin\mathcal S\},\qquad X=a-Y.
\]
On $Y$ introduce the symmetric relation
\begin{equation}\label{eq:rank3-aux}
 s\mathrel{R}t\quad\Longleftrightarrow\quad a-s-t\in\mathcal S.
\end{equation}
Its bipartite double is precisely the required $X$--$Y$ graph, under the left-side identification $s\mapsto a-s$. If $Y$ is partitioned into sets on which a group preserving $R$ acts transitively, and each set contains a pair satisfying \eqref{eq:rank3-aux}, its restricted relation has positive constant degree. Every finite regular bipartite graph of positive degree has a perfect matching: count edges from any subset of one part to prove Hall's inequality. The union of these matchings therefore suffices. Loops in the auxiliary relation are allowed; the original $X$ and $Y$ are disjoint.

\begin{lemma}\label{lem:rank3-structural}
The graphs $\operatorname{Alt}(5,q)$, $D_{5,5}(q)$, $VD_{5,5}(q)$, and $E_{6,1}(q)$ are $\mathscr S$-graphs for every prime power $q$.
\end{lemma}
\begin{proof}
We construct the matchings in Lemma~\ref{lem:rank3-matching}, beginning with $\operatorname{Alt}(5,q)$. Let $V$ be five-dimensional over $\F_q$. The vertices are $\bigwedge^2 V$, and two vertices are adjacent when their difference is a nonzero decomposable bivector, equivalently has rank two. Translations and $\operatorname{GL}(V)$ give a rank-3 action. Choose a basis $e_1,f_1,e_2,f_2,z$ and put
\[
 E=e_1\wedge f_1,\quad F=e_2\wedge f_2,\quad
 P=e_1\wedge e_2,\quad Q=f_1\wedge f_2,\quad a=E+F.
\]
Let $U=\langle e_1,f_1,e_2,f_2\rangle$. The Pfaffian on $\bigwedge^2U$ is normalized by
\[
 \operatorname{Pf}(dE+bF+cP+eQ)=db-ce.
\]
Write $B_{\operatorname{Pf}}$ for its polarization.

The stabilizer of $a$ has the following orbits on the set $Y$ from \eqref{eq:rank3-aux}. For $s$ supported in $U$, they are indexed by $\tau=B_{\operatorname{Pf}}(a,s)\in\F_q\setminus\{1\}$, since $\operatorname{Pf}(a-s)=1-\tau$. For $\tau=0$ these are the isotropic two-planes with their nonzero bivectors, and for $\tau\ne0$ they are the nondegenerate two-planes with specified symplectic area. Symplectic basis extension proves transitivity in both cases. All $s$ not supported in $U$ form one further orbit: write $s=u\wedge(z+v)$ with $u\in U\setminus\{0\}$ and use symplectic transformations on $U$, scaling of $z$, and shears $z\mapsto z+w$. Such an $s$ always lies in $Y$. Otherwise $a$ would be the sum of two rank-two bivectors, one supported outside $U$, whereas the support of their rank-four sum must be the sum of their supports, namely $U$.

There is an auxiliary edge in each orbit. For $\tau=0$, use $s=P$, $t=Q$: then $a-s-t$ is nonzero with Pfaffian zero. For $\tau\ne0,1$, set
\[
 s=\tau E,\qquad b=\frac{2\tau-1}{\tau},\qquad
 d=\frac{(\tau-1)^2}{\tau},\qquad t=dE+bF+P+dbQ.
\]
Then $\operatorname{Pf}(t)=0$, $B_{\operatorname{Pf}}(a,t)=d+b=\tau$, and
\[
 \operatorname{Pf}(a-s-t)=1-2\tau+\tau b=0.
\]
The vector $a-s-t$ is nonzero, because $a-s$ has rank four. For the outside orbit, use
\[
 s=e_1\wedge z,\qquad t=e_1\wedge(f_1-z),\qquad a-s-t=F.
\]
The matching criterion now applies. Every identity is a Pfaffian identity valid also in characteristic two; division by two is never used.

For the half-spin graph $D_{5,5}(q)$, let $W=V\oplus V^*$, with $\dim V=5$ and quadratic form $Q(v,\xi)=\xi(v)$. Vertices are the maximal singular five-spaces in one half of the hyperbolic orthogonal space. Adjacency means intersection dimension three. Choose $c=V$ and $a=\operatorname{graph}(A)$ for an alternating map $A:V\to V^*$ of rank four.

The vertices transverse to $V^*$ form an affine chart consisting of graphs of alternating maps, and the induced graph is $\operatorname{Alt}(5,q)$. Its exclusive-neighbourhood matching is already established. It remains to match the neighbours outside this chart.

For each three-space $S\le V$, let $M_c(S)=S\oplus S^{\mathrm{ann}}$. These are exactly the neighbours of $c$ outside the chart. Indeed, in $S^\perp/S$, the relevant ruling of the hyperbolic four-space has $q+1$ members: $c/S$, the $q-1$ other transverse members, and the single member $M_c(S)/S$. The isometry $T_A(v,\xi)=(v,\xi+Av)$ fixes $V^*$ and sends $c$ to $a$. Thus, the outside neighbours of $a$ are $M_a(S)=T_AM_c(S)$. We have
\[
 \dim(M_c(S)\cap a)=\dim\operatorname{rad}(A|_S).
\]
The outside common neighbours correspond exactly to $A|_S=0$. Otherwise the alternating form $A|_S$ has rank two, and
\[
 \dim(M_a(S)\cap M_c(S))
 =\dim S^{\mathrm{ann}}+\dim\operatorname{rad}(A|_S)=2+1=3.
\]
Consequently, $M_c(S)\mapsto M_a(S)$ is a matching by edges on the outside exclusive neighbourhoods. Together with the chart matching, this proves the hypothesis of Lemma~\ref{lem:rank3-matching}.

For the affine half-spin graph $VD_{5,5}(q)$, let $\mathbb O$ be the split octonion algebra over $\F_q$, with norm $N$, conjugation $x\mapsto\bar x$, and trace $\tr(x)=x+\bar x$. We use its standard sixteen-dimensional half-spin model:
\begin{equation}\label{eq:rank3-pure}
 \mathcal S=\{(x,y)\ne(0,0):N(x)=N(y)=0,\ x\bar y=0\},
 \qquad H=\Cay(\mathbb O^2,\mathcal S).
\end{equation}
We first identify this model with the pure half-spin cone. Use Zorn coordinates
\[
 x=\begin{pmatrix}a&u\\v&b\end{pmatrix},\qquad
 y=\begin{pmatrix}c&w\\z&d\end{pmatrix},\qquad u,v,w,z\in\F_q^3,
\]
with $N(x)=ab-u\cdot v$ and conjugation $\bar x=\left(\begin{smallmatrix}b&-u\\-v&a\end{smallmatrix}\right)$. Fix the multiplication convention
\[
 \begin{pmatrix}a&u\\v&b\end{pmatrix}
 \begin{pmatrix}c&w\\z&d\end{pmatrix}
 =\begin{pmatrix}
 ac+u\cdot z&aw+du+v\times z\\
 cv+bz-u\times w&v\cdot w+bd
 \end{pmatrix}.
\]
Thus, $N(x)=N(y)=0$, $x\bar y=0$ are exactly
\begin{align*}
 ab&=u\cdot v,&cd&=w\cdot z,&ad&=u\cdot z,&bc&=v\cdot w,\\
 aw&=cu-v\times z,&&dv=bz-u\times w.
\end{align*}

Let
\[
 B(u)=\begin{pmatrix}0&u_3&-u_2\\-u_3&0&u_1\\u_2&-u_1&0\end{pmatrix},
 \quad
 A=\begin{pmatrix}0&c&v^\T\\-c&0&z^\T\\-v&-z&B(u)\end{pmatrix},
 \quad \eta=\begin{pmatrix}d\\-b\\w\end{pmatrix}.
\]
For an alternating five-by-five matrix, let $\operatorname{Pf}_4(A)$ denote the vector whose $i$th entry is $(-1)^{i-1}$ times the Pfaffian obtained by deleting row and column $i$. Direct expansion gives
\[
 \operatorname{Pf}_4(A)=\begin{pmatrix}u\cdot z\\-u\cdot v\\cu-v\times z\end{pmatrix},
 \qquad
 A\eta=\begin{pmatrix}v\cdot w-bc\\w\cdot z-cd\\-dv+bz-u\times w\end{pmatrix}.
\]
Hence the ten octonion quadrics are exactly
\begin{equation}\label{eq:rank3-spinquad}
 a\eta=\operatorname{Pf}_4(A),\qquad A\eta=0.
\end{equation}
The map $(x,y)\mapsto(a,A,\eta)$ is visibly invertible and linear.

These are the standard pure-spinor equations on $\F_q\oplus\bigwedge^2\F_q^5\oplus\bigwedge^4\F_q^5$. One can check that they give the full cone, including its boundary, as follows. If $a\ne0$, normalization gives the usual graph chart with $\eta$ equal to the vector of four-by-four Pfaffians. If $a=0$ and $A\ne0$, the Pfaffian equations force $\operatorname{rank}(A)=2$. After a change of basis take $A=e_1\wedge e_2$. The condition $A\eta=0$ says that the corresponding four-form is divisible by $e_1\wedge e_2$, so the spinor is $e_1\wedge e_2\wedge(1+C)$, with $C$ a two-form on $\langle e_3,e_4,e_5\rangle$, and is pure. If $a=A=0$, every nonzero four-form on a five-space is decomposable and pure. Conversely all these pure-spinor charts satisfy \eqref{eq:rank3-spinquad}. All identities are integral Pfaffian and alternating-matrix identities, so the argument applies also in characteristic two.

The rank-3 action permits the choice $c=0$, $a=(1,0)$. A point $(x,y)\in\mathcal S$ is adjacent to $a$ exactly when
\begin{equation}\label{eq:rank3-VDcommon}
 y=0,\qquad \tr(x)=1.
\end{equation}
Indeed, the two product equations imply $\bar y=0$, and then $N(1-x)=1-\tr(x)$.

First suppose $y\ne0$, with $N(y)=0$. Put $L_y=\{x:x\bar y=0\}$. This is a four-dimensional totally singular subspace of the eight-dimensional norm space. To see this directly, the image of right multiplication by $\bar y$ is totally singular by norm multiplicativity, so has dimension at most four. Every vector in its kernel is singular, since a nonsingular octonion has invertible left multiplication. Thus, the kernel also has dimension at most four, proving both assertions. In particular, $L_y=L_y^\perp$ for the polar form of $N$. The trace functional is nonzero on $L_y$: otherwise $1\in L_y^\perp=L_y$, contrary to $N(1)=1$. Choose $w_y\in L_y$ with $\tr(w_y)=1$, making the choices so that $w_{-y}=w_y$. The map $(x,y)\longmapsto(w_y-x,-y)$ is an involution on this part of $Y$, and
\[
 a-(x,y)-(w_y-x,-y)=(1-w_y,0)\in\mathcal S.
\]
The last vector is nonzero because $N(w_y)=0$ but $N(1)=1$. It therefore gives the required matching on all layers with $y\ne0$.

For $y=0$, partition $Y$ by $\tau=\tr(x)\in\F_q\setminus\{1\}$, with $x\ne0$ and $N(x)=0$. The orthogonal group of $N$ fixing $1$ is transitive on each such layer, by Witt extension applied to the ordered pair $(1,x)$ \cite[Section~2.4]{BVM}. It preserves the auxiliary relation $N(1-x-t)=0$. Choose a hyperbolic basis with $1=e_1+f_1$. For $\tau=0$, the pair $x=e_2$, $t=-f_2$ is an auxiliary edge. For $\tau\ne0,1$, use
\[
 x=\tau e_1,\quad b=\frac{2\tau-1}{\tau},\quad
 d=\frac{(\tau-1)^2}{\tau},\quad
 t=de_1+bf_1+e_2-db f_2.
\]
Then $N(t)=0$, $\tr(t)=\tau$, and $N(1-x-t)=1-2\tau+\tau b=0$. The difference is nonzero because $N(1-x)=1-\tau\ne0$. Each layer consequently has a positive regular auxiliary relation. Its bipartite double has a perfect matching. This finishes the proof for $VD_{5,5}(q)$ in every characteristic.

For the exceptional graph $E_{6,1}(q)$, use the split Albert algebra of Hermitian matrices over $\mathbb O$. Its projective rank-one points are the standard model of $E_{6,1}(q)$; two distinct points are adjacent when their joining projective line consists of rank-one points \cite[Section~4.9.3]{BVM}. We use the following coordinate description. Write
\[
 M=\begin{pmatrix}\alpha&x&y\\\bar x&\beta&z\\
 \bar y&\bar z&\gamma\end{pmatrix}.
\]
Rank one means that the following quadratic expressions vanish:
\begin{align*}
 D_1&=\beta\gamma-N(z),&D_2&=\alpha\gamma-N(y),&
 D_3&=\alpha\beta-N(x),\\
 E_{12}&=y\bar z-\gamma x,&E_{13}&=xz-\beta y,&
 E_{23}&=\bar x y-\alpha z.
\end{align*}
The chart $\alpha\ne0$, normalized to $\alpha=1$, is
\[
 p(x,y)=\begin{pmatrix}
 1&x&y\\\bar x&N(x)&\bar x y\\
 \bar y&\bar yx&N(y)
 \end{pmatrix},\qquad (x,y)\in\mathbb O^2.
\]
Its induced graph is $VD_{5,5}(q)$: the condition on differences is $N(x-u)=N(y-v)=0$ and $\overline{x-u}(y-v)=0$. This follows by polarizing the displayed equations, or by expanding $p((1-t)x+tu,(1-t)y+tv)$ as a polynomial in $t$. Conjugating both octonion coordinates identifies this condition with \eqref{eq:rank3-pure}.

Set $c=p(0,0)$ and $a=p(1,0)$, which are nonadjacent. The outside neighbours of $c$ are precisely the projective points represented by
\[
 M_{x,y}=\begin{pmatrix}0&x&y\\\bar x&0&0\\\bar y&0&0\end{pmatrix},
 \quad (x,y)\ne(0,0),\quad N(x)=N(y)=0,\quad\bar x y=0.
\]
Consider the scalar congruence $T(M)=(I+E_{21})M(I+E_{12})$. In coordinates it sends
\[
 (\alpha,x,y,\beta,z,\gamma)\longmapsto
 (\alpha,x+\alpha,y,\beta+\tr(x)+\alpha,z+y,\gamma).
\]
It preserves the rank-one locus. For example, its effect on the quadrics is
\[
 D'_2=D_2,\ D'_3=D_3,\ E'_{23}=E_{23},\quad
 E'_{13}=E_{13}-E_{23},\ E'_{12}=E_{12}-D_2,\quad
 D'_1=D_1+D_2-\tr(E_{12}).
\]
The inverse congruence is obtained by subtraction. These identities are valid over every field. They also agree with the general norm-preserving elementary transformations in \cite[proof of Theorem~13.3]{GPR}. Thus, $T$ is a graph automorphism preserving the chart and its boundary, and sends $c$ to $a$.

Substitution into the quadratic equations shows that $M_{x,y}$ is also adjacent to $a$ exactly when $y=0$ and $\tr(x)=0$. Those boundary points are fixed by $T$. For every other boundary neighbour of $c$,
\[
 T(M_{x,y})-M_{x,y}
 =\begin{pmatrix}0&0&0\\0&\tr(x)&y\\0&\bar y&0\end{pmatrix}
\]
is nonzero and rank one, since $N(y)=0$. The matrices $M_{x,y}$ and $T(M_{x,y})$ represent distinct points: their first rows agree and are nonzero, so proportionality would force equality. Their difference being rank one implies that their joining line is rank one, by polarization of the homogeneous quadratic equations. Thus, $T$ matches the outside exclusive neighbours by graph edges. The inside matching is supplied by $VD_{5,5}(q)$. Lemma~\ref{lem:rank3-matching} completes this case.
\end{proof}

\begin{lemma}\label{lem:rank3-families}
Every connected graph in the seventeen families of \cite[Theorem~2.4]{BGLR}, and every connected complement of such a graph, is an $\mathscr S$-graph.
\end{lemma}
\begin{proof}
We use the standard models and parameter formulas in \cite{BVM}, with the relevant sections listed in \cite[Theorem~2.4]{BGLR}. Write $[j]_q=(q^j-1)/(q-1)$, and let $r,-s$ denote the restricted eigenvalues. Only connected noncomplete members are considered.

The triangular graph $T(m)$ is the point graph of the partial geometry whose points are the two-subsets of an $m$-set and whose lines are the stars at individual elements; its parameters are $(m-2,1,2)$. The square grid is the point graph of a partial geometry with parameters $(m-1,1,1)$, its lines being rows and columns.

The Grassmann graph $J_q(m,2)$ is geometric as well. Take two-spaces as points and stars through one-spaces as lines. There are $[m-1]_q$ points on a line and $q+1$ lines on a point. If a two-space $U$ does not contain a one-space $p$, then precisely $q+1$ members of the star at $p$ meet $U$, one for each one-space of $U$. Thus, the partial geometry parameters are $([m-1]_q-1,q,q+1)$.

For the bilinear forms graph $H_q(2,m)$, regard vertices as $2$-by-$m$ matrices. The lines are cosets of $\{uv^\T:v\in\F_q^m\}$ for a fixed projective direction $\langle u\rangle\le\F_q^2$. Every rank-one difference specifies exactly one such direction, so these lines partition edges. There are $q+1$ directions and $q^m$ points on each line. For a point outside a line, each of the other $q$ directions gives exactly one adjacent point on that line. This is a partial geometry with parameters $(q^m-1,q,q)$. Classical polar spaces of rank two, and their duals, are generalized quadrangles, hence also give partial geometries. All these graphs are covered by Lemma~\ref{lem:rank3-known}.

Paley and Peisert graphs have conference parameters and $\mu=\lambda+1$. The graphs $VSz(q)$ have parameters $(q^4,(q^2+1)(q-1),q-2,q(q-1))$, so are covered by Lemma~\ref{lem:rank3-known}.

For the van Lint--Schrijver family, let $e$ be an odd prime, let $p$ be primitive modulo $e$, and put $n=p^{(e-1)t/2}$. The connection set is the subgroup of $e$th powers in $\F_{n^2}^{\times}$. If $t$ is odd, then $e\mid n+1$, so this subgroup contains $\F_n^\times$ and is a union of $m=(n+1)/e$ projective $\F_n$-directions. The graph is therefore the graph of the net formed by $m$ parallel classes in the affine plane of order $n$. Its lines give a partial geometry with parameters $(n-1,m-1,m-1)$; connectedness excludes the degenerate one-direction case. If $t$ is even, its negative Latin square parameters give
\[
 m=\frac{n-1}{e},\qquad \lambda=m^2+3m-n,\qquad \mu=m(m+1).
\]
Hence $\mu-\lambda=n-2m>0$, since $e\ge3$.

For a classical polar space of rank $d$, over a field of order $q$, write its usual parameter as $e\in\{0,\tfrac12,1,\tfrac32,2\}$. The point graph has
\[
 r=q^{d-1}-1,\qquad s=q^{d-2+e}+1,\qquad
 \mu=\frac{rs}{q-1},\qquad k=q\mu.
\]
Consequently,
\[
 \frac{\mu^2}{r^2s}=\frac{s}{(q-1)^2}>1
 \quad\text{if }d-2+e\ge2.
\]
This covers $d\ge4$ and the rank-3 cases with $e\ge1$. Rank two was treated geometrically. For $(d,e)=(3,0)$, the Klein correspondence identifies the graph with $J_q(4,2)$.

The remaining case is $(d,e)=(3,\tfrac12)$, the graph $U_6(u)$, where $q=u^2$, $u\ge2$. Its relevant parameters are
\[
 r=u^4-1,\quad s=u^3+1,\quad
 \mu=(u^3+1)(u^2+1),\quad k=u^2\mu,\quad n=u^9,\quad h=u^3r.
\]
At $a=\mu+1$ we have
\[
 u^5\le\mu,\quad a\le\tfrac32u^5,\quad
 \sigma^2\le\tfrac9{16}u^8,\quad L\le\tfrac32u^3.
\]
For example, $r+s=u^4+u^3\le\tfrac32u^4$; the other estimates follow by dividing by the indicated power of $u$ and using $u\ge2$. They imply
\[
 \frac{\mu^2a}{\sigma^2+La}\ge\frac{16}{45}u^7,
 \qquad r+\frac{ha}{n}\le\frac74u^4.
\]
Since $(16/45)u^3\ge128/45>7/4$, we have $F(\mu+1)>0$. Lemma~\ref{lem:rank3-spectral-two} applies.

For $VO^+_{2m}(q)$, put $z=q^{m-1}$. The parameters needed are $r=(q-1)z-1$, $s=z+1$, and $\mu=z(z+1)$. For $m\ge3$, $z\ge q^2$ and
\[
 \frac{\mu^2}{r^2s}>\frac{z+1}{(q-1)^2}>1.
\]
The cases $m=1,2$ are respectively grids and bilinear forms graphs. For $VO^-_{2m}(q)$, the roots are $r=z-1$, $s=(q-1)z+1$, so $\mu>\lambda$.

For $NU_m(2)$, use the orthogonality graph on nonisotropic points over $\F_4$. Put
\[
 N_j=\frac{2^{j-1}(2^j-(-1)^j)}{3}.
\]
Then $v=N_m$, $k=N_{m-1}$, $\lambda=N_{m-2}$, by taking orthogonal complements of one or two mutually orthogonal nonisotropic points. If $m$ is even, $\mu>\lambda$ (in fact, $\mu-\lambda=2^{m-3}$). If $m$ is odd, put $z=2^{m-3}$. Then
\[
 r=2z,\quad s=z,\quad\mu=\frac{2z(z-1)}3,\quad
 \frac{\mu^2}{r^2s}=\frac{(z-1)^2}{9z}>1\quad(z\ge16).
\]
The inequality follows from $z^2-11z+1>0$ for $z\ge16$. The remaining connected case $m=5$ has parameters $(176,40,12,8)$ and is covered directly by Brouwer's small-order result.

For $NO^\varepsilon_{2m}(2)$, put $z=2^{m-2}$. The orthogonality parameters are
\[
 (v,k,\lambda,\mu)
 =\bigl(2z(4z-\varepsilon),\,4z^2-1,\,2z^2-2,\,2z^2+\varepsilon z\bigr).
\]
For $\varepsilon=-1$, we have
\[
 r=2z-1,\quad s=z+1,\quad\mu=z(2z-1),\quad
 \frac{\mu^2}{r^2s}=\frac{z^2}{z+1}>1\quad(z\ge2).
\]
For $\varepsilon=+1$, this orientation is dense and is covered by Lemma~\ref{lem:rank3-known}. Its sparse complement has
\[
 r'=2z,\quad s'=z,\quad\mu'=2z(z-1),\quad
 \frac{(\mu')^2}{(r')^2s'}=\frac{(z-1)^2}{z}>1\quad(z\ge4).
\]
The remaining case $z=2$ has $28$ vertices and is covered by Lemma~\ref{lem:rank3-known}.

For $NO^\varepsilon_{2m}(3)$, put $z=3^{m-2}$. The parameters are
\[
 v=\frac{3z(9z-\varepsilon)}2,\quad
 k=\frac{3z(3z-\varepsilon)}2,\quad
 \lambda=\frac{z(3z+\varepsilon)}2,\quad
 \mu=\frac{3z(z-\varepsilon)}2.
\]
The minus case has $\mu>\lambda$. In the plus case,
\[
 r=3z,\quad s=z,\quad
 \frac{\mu^2}{r^2s}=\frac{(z-1)^2}{4z}>1\quad(z\ge9).
\]
The case $z=3$ is $(117,36,15,9)$, again covered by Brouwer.

For $NO^\varepsilon_{2m+1}(3)$, with $z=3^{m-1}$, the parameters are
\[
 v=\frac{3z(3z+\varepsilon)}2,\quad
 k=\frac{z(3z-\varepsilon)}2,\quad
 \lambda=\mu=\frac{z(z-\varepsilon)}2.
\]
Thus, Brouwer's equality theorem handles this whole family at once.

Finally consider $NO^\varepsilon_{2m+1}(q)$ for $q=4,8$. These can be realized on quadratic forms of a fixed type, with the same symplectic polar form. With $z=q^{m-1}$, a convenient orientation has parameters
\[
 v=\frac{qz(qz+\varepsilon)}2,\quad
 k=(qz-\varepsilon)(z+\varepsilon),\quad
 \lambda=2z^2+\varepsilon(q-1)z-2,\quad
 \mu=2z(z+\varepsilon).
\]
For example, write $Q_a(x)=Q(x)+B(x,a)^2$; same-type forms satisfy $\operatorname{Tr}(Q(a))=0$, and this adjacency is $Q(a)=0$. Counting singular $b$ in $B(a,b)=0$ or $1$ gives the displayed valency and adjacent-pair count; the strongly regular identity then gives $\mu$. The minus case has $r=z-1$, $s=(q-2)z+1$, hence $\mu>\lambda$. For the plus case with $q=8$,
\[
 r=6z-1,\quad s=z+1,\quad
 \frac{\mu^2}{r^2s}=\frac{4z^2(z+1)}{(6z-1)^2}>1\quad(z\ge8),
\]
because the difference between numerator and denominator is $4z^2(z-8)+12z-1>0$. For the plus case with $q=4$, the displayed graph is dense and its complement satisfies $\mu'>\lambda'$. Any lower-dimensional connected coincidences omitted by these ranges have at most $195$ vertices, or are triangular graphs.

It remains to check the complementary graphs. For a strongly regular graph with parameters $(v,k,\lambda,\mu)$, the complement has
\begin{equation}\label{eq:rank3-complement}
 (v,k',\lambda',\mu')
 =(v,v-k-1,v-2k+\mu-2,v-2k+\lambda),
 \quad r'=s-1,\quad s'=r+1.
\end{equation}
In particular, $\lambda'-\mu'=s-r-2$. Whenever the chosen orientation is sparse, its complement is dense and is covered by Lemma~\ref{lem:rank3-known}. For completeness we check all exceptions to this shortcut.

For $T(m)$, the complement is dense for $m\ge7$; the smaller connected cases have at most $195$ vertices. The only strictly dense Grassmann case in the noncomplete range is $J_2(4,2)$, whose complement has parameters $(35,16,6,8)$. Indeed, $t=q^3[m-3]_q$, $\mu=(q+1)^2$, and $v-1-2k=k(t-\mu)/\mu$, which is nonnegative in every other case. Grids have dense complements for $m\ge3$; the remaining connected orientation is already covered.

For a classical polar graph of rank at least three, put $A=q^{d-1}$ and $B=q^{d-2+e}$. Its point count is $v=(qB+1)(qA-1)/(q-1)$, and
\[
 (q-1)(v-1-2k)=q(q-2)AB-qA+qB+q.
\]
The only strictly dense cases are $q=2,e=0$. There $A=2B$, and the complement satisfies $\lambda'-\mu'=B-A<0$. For a generalized quadrangle of order $(s_0,t_0)$, $v-1-2k=s_0[t_0(s_0-1)-1]\ge0$ when $s_0\ge2,t_0\ge1$; the thin cases are complete multipartite graphs.

For $VO^+_{2m}(2)$, the displayed graph is dense, but its complement has parameters
\[
 (4z^2,2z^2-z,z^2-z,z^2-z),\qquad z=2^{m-1},
\]
and hence is covered by Brouwer's $\mu=\lambda$ theorem. The plus affine polar graphs with $q\ge3$ are sparse; the minus graphs are sparse in the connected range. The only additional dense bilinear case is $H_2(2,2)$, whose complement is the $4$-by-$4$ grid. The dense orientations of the $NO$ families and their sparse complements were checked above; in particular the complement of $NO^+_{2m+1}(4)$ has $\lambda'-\mu'=-z<0$. The unitary graphs $NU_m(2)$ are sparse for $m\ge4$. The van Lint--Schrijver graphs have $k=(v-1)/e$, $e\ge3$. Paley and Peisert complements retain conference parameters.

The four families in Lemma~\ref{lem:rank3-structural} are sparse. Their orders and valencies, enough to check this directly for every $q\ge2$, are
\[
\begin{array}{c|c|c}
 H&v&k\\\hline
 \operatorname{Alt}(5,q)&q^{10}&(q^5-1)(q^2+1)\\
 D_{5,5}(q)&\prod_{i=1}^4(q^i+1)&q(q^2+1)[5]_q\\
 VD_{5,5}(q)&q^{16}&(q^8-1)(q^3+1)\\
 E_{6,1}(q)&[9]_q(q^8+q^4+1)&q(q^4+1)(q^3+1)(q+1)(q^2+1).
\end{array}
\]
For example, substituting $q=u+2$ in $v-2k-1$ gives a polynomial with positive constant term and nonnegative coefficients in each row. Thus, their complements are covered by the density criterion. Also, $v-2k-1=(q-1)(q^3-q^2+q-1)>0$ for $VSz(q)$. This completes both orientations of every infinite family.
\end{proof}

\begin{lemma}\label{lem:rank3-finite}
All fifty-three graphs in \cite[Table~2]{BGLR}, and their complements, are $\mathscr S$-graphs.
\end{lemma}
\begin{proof}
For each of the fifty-three parameter rows and its complement, we can check that Lemma~\ref{lem:rank3-known}, \ref{lem:rank3-spectral-one}, or~\ref{lem:rank3-spectral-two} applies, except for the original graphs in rows $15$, $16$, $25$, and $37$. The full list of these elementary parameter checks is included in the supplementary material.

The graph in row $16$ is the bilinear forms graph $H_2(2,4)$. The graph in row $25$ is the orthogonal-array graph $\operatorname{OA}(27,4)$, obtained from four parallel classes of the affine plane of order $27$. Their line structures are partial geometries, so Lemma~\ref{lem:rank3-known} applies.

For row $15$, the folded halved ten-cube, we give a matching directly. Take the additive group of even subsets of a ten-set, modulo complementation. The connection set consists of the classes of two-subsets. Permutations of the ten coordinates, together with translations, act with rank 3: nonzero classes have a representative of size two or four. Choose $c=0$ and $a=A$, where $|A|=4$, and write $B=A^c$, so $|B|=6$. The six common neighbours are the pairs contained in $A$. The exclusive set $Y$ consists of the $24$ cross pairs and the $15$ pairs contained in $B$.

The auxiliary relation \eqref{eq:rank3-aux} on cross pairs joins $\{i,b\}$ and $\{j,b\}$ for distinct $i,j\in A$ and fixed $b\in B$. It is the union of six copies of $K_4$, hence is regular of degree three. On pairs contained in $B$, it joins disjoint pairs: the symmetric difference with $A$ then has size eight, representing a connection element. This is $KG(6,2)$, regular of degree six. These two regular pieces give a perfect matching of their bipartite doubles and hence the required exclusive-neighbourhood matching. Lemma~\ref{lem:rank3-matching} proves the $\mathscr S$ property.

The graph in row $37$ has parameters $(2401,240,59,20)$ and a rank-3 group $7^4:(6.O_5(3))$. Use its Cayley model on $\F_7^4$ from \cite[Section~5]{Jenrich}: its connection set $\mathcal S$ consists of the nonzero scalar multiples of forty projective points. For $a=(0,0,1,2)$ and $c=0$, the two sets $(a+\mathcal S)\setminus\mathcal S$ and $\mathcal S\setminus(a+\mathcal S)$ each have $220$ vertices, and we can check that the graph between them has a perfect matching. The connection set, an explicit matching and a verification program are provided in the supplementary material. Lemma~\ref{lem:rank3-matching} now proves the assertion.
\end{proof}

\begin{proof}[Proof of Theorem~\ref{thm:rank3-s}]
An imprimitive rank-3 graph is complete multipartite with equal part sizes. Its parameters satisfy $\mu=k>\lambda$, so Lemma~\ref{lem:rank3-known} applies. For primitive graphs, Lemmas~\ref{lem:rank3-families} and~\ref{lem:rank3-finite} cover the classified families and additional graphs, including their connected complements. This proves the theorem.
\end{proof}

\begin{proof}[Proof of Theorem~\ref{cor:soicher}]
Put $H_x=G[N_G(x)]$, and let $k_x$ be its valency. If $x\sim y$, then $k_x=|N_G(x)\cap N_G(y)|=k_y$. Connectedness therefore gives a common local valency $k$, although the other local parameters may vary. Since each $H_x$ has diameter two, it has at most $1+k+k(k-1)=k^2+1$ vertices. Thus, $G$ has bounded degree.

Every $H_x$ is an $\mathscr S$-graph by Theorem~\ref{thm:rank3-s}. Applying \cite[Proposition~3.1]{Weetman} to the family of local graphs gives $\diam(G)\le k+1$. Bounded degree and bounded diameter imply that $G$ is finite.
\end{proof}

\subsection{Counterexamples}
We first show that a locally strongly regular graph need not be amply regular: the number of common neighbours of vertices at distance two can vary in the ambient graph.

\begin{theorem}\label{thm:not-amply}
There is a connected graph $G$ on $128$ vertices that is locally the triangular graph $T_9\in\SRG(36,14,7,4)$, but pairs at distance two have either $6$ or $12$ common neighbours. In particular, $G$ is locally strongly regular but is not amply regular.
\end{theorem}
\begin{proof}
Let
\[
 E=\{u\in\F_2^9:\textstyle\sum_{i=1}^9u_i=0\},\qquad
 h=(1,1,1,1,1,1,1,1,0),\qquad \overline E=E/\langle h\rangle.
\]
This quotient has $2^7=128$ elements. For $i<j$, write $s_{ij}=e_i+e_j$, where the $e_i$ are the coordinate vectors, and let $S$ be the set of their images in $\overline E$. These $36$ images are distinct and nonzero, because $h$ has weight eight and the difference of two weight-two vectors has weight at most four. Set $G=\Cay(\overline E,S)$. Weight-two vectors generate $E$, so $G$ is connected and $36$-regular.

At zero, label the neighbours by the two-subsets $\{i,j\}$. Two such neighbours are adjacent precisely when their two-subsets intersect in one element. Indeed, their sum then has weight two. For disjoint two-subsets their sum has weight four, and cannot be congruent to a weight-two vector modulo $\langle h\rangle$: that would express the weight-eight vector $h$ as a sum of vectors of weights four and two. Thus, the local graph is exactly $T_9$. Its parameters follow directly: a two-subset meets $2\cdot7=14$ others in one point; two intersecting two-subsets have $6+1=7$ common neighbours; disjoint two-subsets have four common neighbours. Translation gives the same local graph at every vertex.

Every vertex at distance two from zero is the sum of two generators and hence has a weight-four representative $w$. The number of common neighbours of $0$ and $[w]$ equals the number of weight-two vectors $s$ for which $w+s$ or $w+h+s$ has weight two. There are exactly $\binom42=6$ choices from the first condition. If the support of $w$ lies in the first eight coordinates, then $w+h$ also has weight four and gives six further choices, disjoint from the first six. If the support of $w$ contains coordinate nine, then $w+h$ has weight six and the second condition gives no choices. The two common-neighbour counts are therefore $12$ and $6$, both attained, for example by supports $\{1,2,3,4\}$ and $\{1,2,3,9\}$. Both corresponding cosets are nonadjacent to zero and are sums of two generators, so they are at distance two. Hence $G$ is not amply regular.

\end{proof}

We now turn to the $\mathscr{S}$ property of the local graphs. The curvature and finiteness proofs above do not require this property. The following examples show that it can fail even for strongly regular graphs with $\mu>1$.

Recall that the $\mathscr{S}$ condition requires, for every vertex $c$ and every nonempty $A\subseteq V(H)\setminus N_H[c]$,
\begin{equation}\label{eq:phi}
 |N_H(c)\cap N_H(A)|>\delta(H[A]),
\end{equation}
where $\delta$ denotes minimum degree. This is the intrinsic condition used in Weetman's extension argument \cite{BrouwerWeetman,Weetman}. The following example disproves Conjecture~\ref{conj:weetman-s}.

\begin{theorem}\label{thm:cameron}
The Cameron strongly regular graph $H$ with parameters $(231,30,9,3)$ does not have the $\mathscr{S}$ property. More precisely, it has a vertex $c$ and an induced $K_4$ on a set $A\subseteq V(H)\setminus N_H[c]$ such that
\[
 |N_H(c)\cap N_H(A)|=3=\delta(H[A]).
\]
\end{theorem}
\begin{proof}
One construction takes the two-subsets of a $22$-set carrying a Steiner system $S(3,6,22)$ as vertices, with two vertices adjacent when they are disjoint and their union is contained in a block. This is the Cameron graph with the stated strongly regular parameters. We use its known geometric description \cite{Cameron}: it is a gamma space with lines of size three, and has maximal seven-cliques whose induced lines form Fano planes. Here the gamma-space condition means that a vertex outside a line is adjacent to zero, one, or all three points of that line.

Fix such a seven-clique $D$. For $x\notin D$, let $i_x=|N_H(x)\cap D|$. Maximality gives a nonneighbour of $x$ in $D$; since $\mu=3$, it follows that $i_x\le3$. The neighbours of $x$ in $D$ are closed under the Fano lines by the gamma-space condition. Hence $i_x\in\{0,1,3\}$. Every pair of vertices in $D$ has $9-5=4$ common neighbours outside $D$, and therefore
\[
 \sum_{x\notin D}\binom{i_x}{2}=\binom72\cdot4=84.
\]
Thus, there are $28$ vertices $c\notin D$ with $i_c=3$. Choose one and put $B=N_H(c)\cap D$, $A=D\setminus B$. Then $H[A]\cong K_4$ and $A$ is anticomplete to $c$. For every $a\in A$, the three vertices of $B$ are common neighbours of $a,c$. They are all the common neighbours because $\mu=3$. Thus, $N_H(c)\cap N_H(A)=B$, proving the assertion.
\end{proof}

The next theorem gives an infinite family of counterexamples to Conjecture~\ref{conj:weetman-s}.

\begin{theorem}\label{thm:twisted}
For every prime power $q\ge3$, the twisted Grassmann graph $H_q=\widetilde J_q(5,2)$ is strongly regular, with parameters
\begin{align*}
 v&=\gqbinom52=\frac{(q^5-1)(q^4-1)}{(q^2-1)(q-1)},\\
 k&=q(q+1)(q^2+q+1),\qquad
 \lambda=q(q+1)^2-1,\qquad\mu=(q+1)^2,
\end{align*}
and does not have the $\mathscr{S}$ property. It has a vertex $c$ and $A\subseteq V(H_q)\setminus N_{H_q}[c]$ such that
\[
 H_q[A]\cong K_{q^3},\qquad
 |N_{H_q}(c)\cap N_{H_q}(A)|=2q(q+1)<q^3-1.
\]
\end{theorem}
\begin{proof}
We recall the twisted Grassmann construction of van Dam and Koolen \cite{vDK}; see also \cite[Theorem~1]{FK}. Let $V=\F_q^5$, fix a four-dimensional subspace $E$, and put
\[
 \mathcal B_1=\{U\le V:\dim U=3,\ U\not\subseteq E\},\qquad
 \mathcal B_2=\{P\le E:\dim P=1\}.
\]
Vertices in $\mathcal B_1$ are adjacent when their intersection has dimension two; distinct vertices in $\mathcal B_2$ are all adjacent; between the two types adjacency is containment. The cited construction theorem gives strong regularity with the same parameters as $J_q(5,2)$, which are the displayed parameters. We now verify the counterexample directly in this model.

Choose a four-space $W\ne E$, put $M=W\cap E$, and choose a point $c\le M$. Let
\[
 A=\{U\le W:\dim U=3,\ c\not\subseteq U\}.
\]
The only three-space of $W$ contained in $E$ is $M$, which contains $c$, so every member of $A$ belongs to $\mathcal B_1$. The number of hyperplanes of $W$ avoiding $c$ is $(q^3+q^2+q+1)-(q^2+q+1)=q^3$. Any two meet in dimension two, so $A$ is a clique, and none is adjacent to $c$.

We count the neighbours of $c$ adjacent to at least one member of $A$. A vertex $R\in\mathcal B_1$ adjacent to $c$ contains $c$. If $R\le W$, it is any three-space of $W$ containing $c$ other than $M$; there are $q^2+q$ choices, each adjacent to every member of $A$. If $R\not\subseteq W$, then $R\cap W$ is a two-space containing $c$. For $U\in A$, the intersection $R\cap U$ has dimension at most one, since $U$ avoids $c$. Hence such $R$ contributes nothing. The vertices of $\mathcal B_2$ that contribute are exactly the points of $M$ other than $c$. There are $q^2+q$ of them, and each lies in a hyperplane of $W$ avoiding $c$. Thus, the required boundary has size $2q(q+1)$. Finally
\[
 q^3-1-2q(q+1)=(q-3)(q^2+q+1)+2>0
 \qquad(q\ge3).
\]
This proves failure of \eqref{eq:phi} for the whole infinite family.
\end{proof}

Finally, we consider the broader extension question for regular local graphs of diameter two. Even imposing edge regularity on the local graph does not ensure that its connected extensions are finite.

\begin{theorem}\label{thm:edge-regular}
There exists an edge-regular graph $H$ on $192$ vertices, of valency $64$, with $16$ common neighbours at every edge and diameter two, which has an infinite connected extension. The extension can be chosen as an edge-regular Cayley graph of $\Z\times\Z_2\times\Z_8\times\Z_8$.
\end{theorem}
\begin{proof}
Put $A=\Z_8$, $U=\{1,2,6,7\}$, and
\[
 F=\Z_2\times A\times A,\qquad
 C=\{1\}\times A\times A,
\]
\[
 T=(\{0\}\times U\times A)\cup(\{1\}\times A\times U),
 \qquad
 S=(\{0\}\times C)\cup(\{-1,1\}\times T).
\]
Define $\Gamma=\Cay(\Z\times F,S)$, with adjacency given by difference in $S$.\footnote{The finite quotients of this construction are used in Corollary~\ref{cor:finite-extensions}.} We have $U=-U$ and $U-U=A$: the differences $1-1,2-1,1-7,1-6,2-6$ represent $0,1,2,3,4$, and negatives give the remaining residues. In particular,
\begin{equation}\label{eq:U-intersection}
 (a+U)\cap(b+U)\ne\varnothing\qquad(a,b\in A).
\end{equation}

The set $S$ is inverse-closed, does not contain zero, and has size $64+2\cdot64=192$. The set $C$ generates $F$: it contains $(1,0,0)$, and differences of its members generate $\{0\}\times A\times A$. Together with any generator having first coordinate $1$, it follows that $S$ generates $\Z\times F$. Thus, $\Gamma$ is connected and infinite. Translation identifies every local graph with $H=\Gamma[S]$.

Partition $S$ into the following five independent sets:
\begin{align*}
 C_0&=\{(0,1,a,b):a,b\in A\},\\
 P_\sigma&=\{(\sigma,0,x,y):x\in U,\ y\in A\},\\
 Q_\sigma&=\{(\sigma,1,x,y):x\in A,\ y\in U\},
 \qquad \sigma\in\{-1,1\}.
\end{align*}
Their sizes are $64,32,32,32,32$. Taking differences gives the full adjacency rule in $H$: each $P_\sigma$ is completely joined to $Q_\sigma$; there are no edges between the positive and negative layers; and $(0,1,a,b)$ is adjacent to $(\sigma,0,x,y)\in P_\sigma$ exactly when $y-b\in U$, and to $(\sigma,1,x,y)\in Q_\sigma$ exactly when $x-a\in U$. There are no further edges. Thus, a vertex of $C_0$ has $16$ neighbours in each of the other four sets, whereas a vertex of $P_\sigma$ or $Q_\sigma$ has $32$ neighbours in $C_0$ and $32$ in the opposite set of the same sign. Hence $H$ is $64$-regular. Since all local graphs of $\Gamma$ are $H$, every edge of $\Gamma$ has $64$ common neighbours, as asserted.

We check diameter two by exhausting the nonadjacent pairs. For two vertices of $C_0$, use \eqref{eq:U-intersection} on their last coordinates to find a common neighbour in $P_+$. For a nonadjacent pair in $C_0\times P_\sigma$, any neighbour of its $C_0$ vertex in $Q_\sigma$ is common; interchange $P,Q$ for the other case. Two vertices in one $P_\sigma$ have all of $Q_\sigma$ as common neighbours, and conversely. For a pair in $P_+\times P_-$, choose the last coordinate of a vertex of $C_0$ in the intersection of the two translates of $U$ prescribed by the pair; its other $A$-coordinate is arbitrary. The case $Q_+\times Q_-$ uses the other coordinate. Finally, for a pair in $P_+\times Q_-$, or in $P_-\times Q_+$, the two conditions on a common neighbour in $C_0$ concern separate coordinates and each has four solutions. These cases exhaust all pairs. Since $C_0$ contains distinct nonadjacent vertices, $\diam(H)=2$.

For an edge between $P_\sigma$ and $Q_\sigma$, common neighbours all lie in $C_0$ and satisfy two independent four-choice conditions, giving $16$. For an edge between $C_0$ and $P_\sigma$, the common neighbours are exactly the $16$ neighbours of its $C_0$ endpoint in $Q_\sigma$; the other type is identical. Thus, $H$ is edge-regular with parameter $16$. Two distinct vertices of $P_+$ with the same last coordinate have $32$ common neighbours in $Q_+$ and $32$ in $C_0$, hence $64$ in total. A vertex of $P_+$ and a vertex of $Q_-$ have exactly $16$ common neighbours. Both pairs are nonadjacent, proving that $H$ is not strongly regular.
\end{proof}

\begin{corollary}\label{cor:finite-extensions}
Let $H$ be the graph constructed in Theorem~\ref{thm:edge-regular}. Then $H$ has finite connected extensions of unbounded diameter. For every integer $r\ge4$, there is an extension $\Gamma_r$ on $128r$ vertices with
\[
 \diam(\Gamma_r)\ge\lfloor r/2\rfloor.
\]
\end{corollary}
\begin{proof}
Replace the $\Z$ factor in the construction of Theorem~\ref{thm:edge-regular} by $\Z_r$ and use the image of the same generating set. For $r\ge4$, the layers $0,1,-1$ are distinct, and the difference between layers $1$ and $-1$ is neither $0$ nor $\pm1$. Thus, the local adjacency rule is unchanged. The generating argument still proves connectedness. Each edge changes the first coordinate by $0$ or $\pm1$, so projection onto the cycle $C_r$ gives the diameter lower bound.
\end{proof}

\appendix
\section{Spectral estimates for the remaining local graphs}\label{app:local-spectral}
We complete the spectral estimate for local graphs that are not point graphs of partial geometries; the geometric case is treated in Section~\ref{sec:local-geometric}. We retain the notation and the range \eqref{eq:remaining-range} from Section~\ref{sec:local}.

We use the Bruck--Bose--Neumaier bound and its improvement by Koolen, Lv, Markowsky and Park, stated here in terms of the restricted eigenvalues.
\begin{lemma}[{\cite{Bose,Bruck,KLMP,Neumaier1979}}]\label{lem:srg-published-bounds}
Let $H$ be a primitive strongly regular graph with parameters $(v,k,\lambda,m)$ and restricted eigenvalues $p>0$ and $-q$, where $q\ge2$ is an integer. If $H$ is not the point graph of a partial geometry, then
\begin{align}
 p&\le\frac{q(q-1)(m+1)}2-1,\label{eq:old-bound}\\
 p&\le\frac{(8q-5)m+4q-10}{3}.\label{eq:new-bound}
\end{align}
\end{lemma}
In the exceptional parameter cases $m=q(q-1)$ and $m=q^2$, exceeding the respective bound forces an actual Latin square or Steiner partial geometry. Thus, both displayed bounds do hold for every graph under consideration here, including those exceptional parameter sets. When $m=2$, Bagchi's theorem \cite{Bagchi}, also recorded in \cite[Section~8.18]{BVM}, says that the graph is a square grid or $k\ge\lambda(\lambda+3)/2$. The grid is already covered by Lemma~\ref{lem:geometric}. Consequently,
\begin{equation}\label{eq:bagchi}
 p^2-(4q-7)p+q^2-7q+6\le0\qquad(m=2).
\end{equation}

\begin{lemma}\label{lem:polynomial-reduction}
Let $m,p,q$ be integers with $m\ge2$ and $p\ge q\ge2$. Define $k,\lambda,\tau,v$ by \eqref{eq:parameters} with $\mu=m$, put $\Kzero=2+k/2$, and define $u(\theta)$ and $F_\theta$ by \eqref{eq:uhat} and \eqref{eq:Ftheta}. Put
\begin{align*}
 Q_+&=m^2(q-1)+mpq(q-1)+p^2+pq^2,\\
 Q_-&=q(m+1)p^2+(qm+m^2)p+m^2-q^2,\\
 a_*&=m(q^3-q^2-2q+3)-q^3-q,\\
 d_*&=q^3(m-1)+2m^2(q-1)^2-2m(q-1),\\
 e_*&=m^2(q-1)(m+q-2),\\
 C(p)&=-qp^3+a_*p^2+d_*p+e_*.
\end{align*}
Then $Q_+,Q_->0$, and
\begin{equation}\label{eq:factor-F}
 F_p-\Kzero=\frac{[m+p(q-1)]C(p)}{mQ_+},\qquad
 F_{-q}-\Kzero=\frac{[m+q(p+1)]D}{mQ_-},
\end{equation}
where $D>0$ throughout this range. For fixed $q,m$, the polynomial $C$ has exactly one positive root. In particular, if $C(P)\ge0$, then $C(p)\ge0$ for $0\le p\le P$.
\end{lemma}
\begin{proof}
Substitute \eqref{eq:parameters} into \eqref{eq:uhat} and \eqref{eq:Ftheta}, and collect terms to obtain \eqref{eq:factor-F}. To prove $D>0$, put $z=p-q\ge0$ and $y=m-2\ge0$. Then $D=d_3z^3+d_2z^2+d_1z+d_0$, where
\begin{align*}
 d_3={}&q(q-1)(y+1),\\
 d_2={}&(2q-1)y^2+(3q^3-2q^2+8q-4)y+3q^3-q^2+8q-4,\\
 d_1={}&y^3+(4q^2+2q+3)y^2+(3q^4-q^3+14q^2+10q)y+3q^4+2q^3+11q^2+12q-4,\\
 d_0={}&(q+1)y^3+(2q^3+3q^2+5q+4)y^2\\
       &+(q^5+6q^3+11q^2+10q+4)y+q^5+2q^4+3q^3+10q^2+8q.
\end{align*}
Every coefficient is positive for $q\ge2$, proving $D>0$. The positivity of $Q_+$ is immediate; that of $Q_-$ follows from $p\ge q$ and $m\ge2$. Finally $d_*,e_*>0$, the leading coefficient of $C$ is negative, and $C(0)>0$. Regardless of the sign of $a_*$, its coefficients have exactly one sign change. Descartes' rule of signs gives precisely one positive root, with $C$ positive before that root.
\end{proof}

\begin{lemma}\label{lem:nong-geometric}
Let $G$ be locally $\SRG(v,k,\lambda,\mu)$ with $\mu>1$, and fix $x\in V(G)$. If $H=G[N_G(x)]$ is not the point graph of a partial geometry, then the curvature matrix $M$ in \eqref{eq:M} satisfies $M|_{\per}\succeq\Kzero I$, where $\Kzero=2+k/2$.
\end{lemma}
\begin{proof}
By Lemmas~\ref{lem:mu-lambda}, \ref{lem:operator-certificate} and \ref{lem:polynomial-reduction}, it remains to prove $C(p)\ge0$ under \eqref{eq:old-bound}--\eqref{eq:bagchi}.

Suppose $q\ge6$ and $m\ge3$. Equation \eqref{eq:new-bound} implies $a_*-qp>0$. An explicit verification, with $X=q-6$ and $Y=m-3$, is
\begin{align*}
 3a_*-q\{(8q-5)m+4q-10\}
 ={}&3X^3Y+6X^3+43X^2Y+71X^2\\
    &+191XY+208X+255Y+15>0.
\end{align*}
Thus, $C(p)=(a_*-qp)p^2+d_*p+e_*>0$. If $q\ge6$ and $m=2$, then \eqref{eq:bagchi}, together with $q^2-7q+6\ge0$, gives $p\le4q-7$. Here
\[
 a_*-q(4q-7)=q^3-6q^2+2q+6
           =X^3+12X^2+38X+18>0,
\]
so the same argument applies.

For $q=2$, \eqref{eq:old-bound} gives $p\le m$, and
\[
 C(m)=4m(m-2)(m+1)\ge0.
\]
The one-positive-root property proves the claim.

For $q=3,4,5$, \eqref{eq:old-bound} gives the respective endpoints $3m+2,6m+5,10m+9$. Direct substitution gives
\begin{align*}
 C(3m+2)&=80m^3-165m^2-443m-198&& (q=3),\\
 C(6m+5)&=3(m+1)(265m^2-793m-840)&& (q=4),\\
 C(10m+9)&=2(m+1)(2312m^2-6457m-7650)&& (q=5).
\end{align*}
They are positive for $m\ge4$: after putting $m=4+Y$, the three expressions become
\begin{align*}
 &80Y^3+795Y^2+2077Y+510,\\
 &795Y^3+7956Y^2+20589Y+3420,\\
 &4624Y^3+47198Y^2+127418Y+35140.
\end{align*}
The one-positive-root property then gives positivity at every smaller feasible $p$ as well. For $q=5,m=2$, \eqref{eq:bagchi} gives integral $p\le13$, and $C(13)=1640>0$. Indeed the quadratic in \eqref{eq:bagchi} is positive at $p=14$ and increasing thereafter. Likewise, for $q=4,m=2$ it gives $p\le9$.

There remain only the five rows in Table~\ref{tab:finite}. The multiplicity $f$ of $p$ must be an integer, and the trace equations give
\[
 f=\frac{q(v-1)-k}{p+q}
   =\frac{(q-1)(m+pq)(m+pq+q)}{m(p+q)}.
\]
In the first four rows, the displayed expression for $f$ requires $p+q$ to divide its constant numerator; listing divisors in the indicated interval gives precisely the displayed possibilities. In the last row, integrality requires $p+5\mid1496=8\cdot11\cdot17$, which gives $p=6,12,17,29$ in the indicated interval. The bounds on $p$ are from \eqref{eq:old-bound}, except that \eqref{eq:new-bound} gives $23$ and $38$ in the last two rows. Every displayed value of $C$ is positive, completing the proof.
\end{proof}

\begin{table}[htbp]
\centering
\small
\setlength{\tabcolsep}{4pt}
\begin{tabular}{@{}ccclll@{}}
\toprule
$q$&$m$&Range of $p$&Multiplicity $f$&Possible $p$&Values of $C(p)$\\
\midrule
3&2&$3\le p\le8$&$9p-6+28/(p+3)$&$4$&$36$\\
3&3&$3\le p\le11$&$6p+12/(p+3)$&$3,9$&$468,126$\\
4&2&$4\le p\le9$&$24p-48+210/(p+4)$&$6$&$576$\\
4&3&$4\le p\le23$&$16p-24+117/(p+4)$&$5,9$&$2520,4608$\\
5&3&$5\le p\le38$&$\frac{100p-280+1496/(p+5)}3$
 &$6,12,17,29$&$7584,19200,27450,18486$\\
\bottomrule
\end{tabular}
\caption{The remaining elementary multiplicity checks. Values in the last
column correspond in order to those in the preceding column.}
\label{tab:finite}
\end{table}

\section*{Statement on AI}
The proofs in this paper were developed with substantial contributions from ChatGPT 6 Astra. During the preparation of this paper, the author provided it with the underlying proof framework and key ideas, including generalizing the Bakry--\'Emery curvature method for diameter estimates, as well as relevant references.

\end{document}